\documentclass[letterpaper, 11pt]{article}
\usepackage{amsmath}
\usepackage{amsfonts}
\usepackage{graphics}
\usepackage{graphicx}
\usepackage{amsthm, amssymb}
\usepackage{verbatim}
\usepackage{pinlabel}

\newtheorem{thm}{Theorem}[section]
\newtheorem*{thm12}{Theorem \ref{thm:12}}
\newtheorem*{thminfinity}{Theorem \ref{thm:infinity}}
\newtheorem*{thmq}{Theorem \ref{thm:thmq}}
\newtheorem*{thmqnoth}{Theorem \ref{thm:qnoth}}
\newtheorem{cor}[thm]{Corollary}
\newtheorem*{infty2}{Corollary \ref{cor:infty2}}
\newtheorem{lem}[thm]{Lemma}

\newtheorem{qu}[thm]{Question}
\newtheorem{prop}[thm]{Proposition}

\begin{document}

\title{Horowitz--Randol pairs of curves in $q$-differential metrics}
\author{Anja Bankovic}
\date{}
\maketitle

\vspace{0.5cm}

\textbf{Abstract}: The Euclidean cone metrics coming from $q$--differentials on a closed surface of genus $g \geq 2$ define an equivalence relation on homotopy classes of closed curves declaring two to be equivalent if they have the equal length in every such metric.  We prove an analog of the result of Randol for hyperbolic metrics (building on the work of Horowitz): for every integer $q \geq 1$, the corresponding equivalence relation has arbitrarily large equivalence classes.  In addition, we describe how these equivalence relations are related to each other.

\section{Introduction}
\label{sec:intro}

The work of Horowitz \cite{horowitz} and Randol \cite{randol} provides examples of the following: for every $n >0$ there exist $n$ distinct homotopy classes of curves $\gamma_1,\ldots,\gamma_n$ on a compact oriented surface $S$ such that for every hyperbolic metric $m$, $l_m(\gamma_i)=l_m(\gamma_j)$, for all $i,j$, where $l_m(\gamma)$ represents the length of the geodesic representative of $\gamma$ in metric $m$. In \cite{clein1} Leininger studies this and related phenomenon for curves on $S$. He also asks whether other families of metrics exhibit similar behaviour.

\begin{qu} \label{qu:clein} \cite{clein1} Do there exist pairs of distinct homotopy classes of curves $\gamma$ and $\gamma'$ which have the same length with respect to every metric in a given family of path metrics. \end{qu}

For an arbitrary family of metrics one expects the answer to be no. Here we study this question for the metrics coming from $q$-differentials on a closed oriented surface $S$, for all $q \geq 1$. More precisely, let $Flat(S)$ denote the set of non-positively curved Euclidean cone metrics on $S$ and $Flat(S,q)$ those that come from $q$-differentials. (See Section~\ref{sec:Background} for more details.) Let $\mathcal{C}(S)$ denote the set of homotopy classes of homotopically nontrivial closed curves on $S$. For every $q\in \mathbb{Z}_+$, define an equivalence relation on $\mathcal{C}(S)$ by declaring $\gamma \equiv_{q} \gamma'$ if and only if $l_m(\gamma)= l_m(\gamma')$, for every $m \in Flat(S,q)$.\\

In \cite{clein1} Leininger answers Question~\ref{qu:clein} in the affirmative for metrics in $Flat(S,2)$. In fact, writing $\gamma \equiv_{h} \gamma '$  if and only if  $l_m(\gamma)= l_m(\gamma')$,  for every hyperbolic metric $m$, he proves:

\begin{thm} \label{thm:clein2} \cite{clein1} For every $\gamma$, $\gamma' \in \mathcal{C}(S)$, $\gamma \equiv_{h} \gamma '$ $\Rightarrow$ $\gamma \equiv_{2} \gamma '$. \end{thm}

Consequently there are arbitrary large $\equiv _2$ - equivalence classes. In this paper we resolve Question~\ref{qu:clein} for all families $Flat(S, q)$, $q\geq 1$, proving $\equiv_{q}$ is nontrivial. In fact there are arbitrary large $\equiv_q$ - classes of curves.

\begin{thmq} For every $q_0,k\in\mathbb{Z}_+$ there are $k$ distinct homotopy classes of curves $\gamma_1,\ldots,\gamma_k \in \mathcal{C}(S)$ such that $\gamma_i\equiv_q\gamma_j$, for all $i,j$ and for all $q\leq q_0$. Thus for every $q \in \mathbb{Z}_+$, the relation  $\equiv_q$ is non-trivial.  \end{thmq}

However, in the limit this phenomenon disappears. To describe this, define $\gamma \equiv_{\infty} \gamma '$ if and only if  $l_m(\gamma)= l_m(\gamma')$, for every $m \in Flat(S, q),$ for every $q \in \mathbb{Z}_+$. 

\begin{thminfinity}  The equivalence relation $\equiv_{\infty}$ is trivial.\end{thminfinity}

In fact, the argument proves something stronger.

\begin{infty2} Let $\displaystyle \{q_i\}_{i=1}^\infty$ be an infinite sequence of distinct positive integers. If $\gamma \equiv_{q_i} \gamma'$ for every $i=1,\, 2, \ldots,$ then $\gamma=\gamma'$ in $\mathcal{C}(S)$.\end{infty2}

Therefore, not all $\equiv_q$ are the same equivalence relations. However we have:

\begin{thm12} For every $\gamma,\gamma'\in\mathcal{C}(S)$, $\gamma \equiv_{1} \gamma' \, \, \Leftrightarrow \, \, \gamma\equiv_{2}\gamma'$.\end{thm12}

In \cite{clein1} it is also shown that the implication in Theorem~\ref{thm:clein2} can not be reversed. As a consequence of our construction we see that a similar statement is true for any $q\in\mathbb{Z}_+$. 

\begin{thmqnoth} For every $q \in \mathbb{Z}_+$, there exist $\gamma, \gamma'\in \mathcal{C}(S)$ so that  $\gamma\equiv_q\gamma'$ but  $\gamma\not\equiv_h\gamma' $. \end{thmqnoth}

Although we will work exclusively with closed surfaces, there are versions of the theorems for punctured surfaces. The proofs of these require slightly different arguments. So for the sake of simplifying the exposition, the main body of the paper treats only closed surfaces. We will discuss modifications to the statements and proofs for punctured surfaces in Section~\ref{sec:punctures}.\\

The outline of the paper is as follows. In Section~\ref{sec:Background} we define Euclidean cone and flat metrics, give standard definitions and state known theorems.  Section~\ref{sec:1and2} contains the proof of Theorem~\ref{thm:12}. In Section~\ref{sec:infinity} we show that for every metric $m\in Flat(S)$ there is a sequence of metrics in  $\bigcup _{q \in \mathbb{Z}_+}Flat(S,q)$ that converge to $m$ and using this we prove that the equivalence relation $\equiv_{\infty}$ is trivial [Theorem~\ref{thm:infinity}]. Then in Section~\ref{sec:q} we describe constructions of curves reflecting properties of the metrics in $Flat(S,q)$ and we prove Theorem~\ref{thm:thmq}. Finally, in Section~\ref{sec:punctures} we sketch proofs of these theorems for punctured surfaces.\\

{\bf Notes}: The construction of Horowitz and Randol is studied in \cite{anderson}, where a number of variants are also surveyed.  These include analogous  results for hyperbolic $3$--manifolds \cite{masters} and for metric graphs \cite{kapovich}.\\

\textbf{Acknowledgment}: I would like to thank my advisor Christopher Leininger for very useful discussions, suggestions and comments. I would also like to thank the referee for carefully reading the paper.\\

\section{Background}
\label{sec:Background}

Let $S$ denote a closed oriented surface of genus at least $2$.

\subsection{Euclidean Cone Metrics}
\label{sec:Background1}

A metric $m$  on $S$ is called a \emph{Euclidean cone metric} if it satisfies the following properties:
\begin{enumerate}
\item[(i)]$m$ is a geodesic metric (not necessaraly uniquely geodesic): the distance between 2 points is the length of a geodesic path between them. 

\item[(ii)]There is a finite set $X \subset S$ such that $m$ on $S\setminus X$ is Euclidean, that is, locally isometric to $\mathbb{R}^2$ with the Euclidean metric.

\item[(iii)]For every $x\in X$, there exists $\epsilon > 0$ so that $B_{\epsilon}(x)$ is isometric to some cone. More precisely $B_{\epsilon}(x)$ is isometric to the metric space obtained by gluing together some (finite) number of sectors of $\epsilon $ -balls about 0 in $\mathbb{R}^2$.  Each $x$ therefore has a well defined cone angle $c(x)\in \mathbb{R}_+$ which is the sum of the angles of the sectors used in construction. See Figure~\ref{fig:tripolukruga}. For any $x \in S\setminus X$ we define $c(x)=2\pi$. 

\end{enumerate}

See Figure~\ref{fig:genus2} for an example of a Euclidan cone metric on a genus $2$ surface.\\ 

\begin{figure}[htb]
\centering
\includegraphics[height=1in]{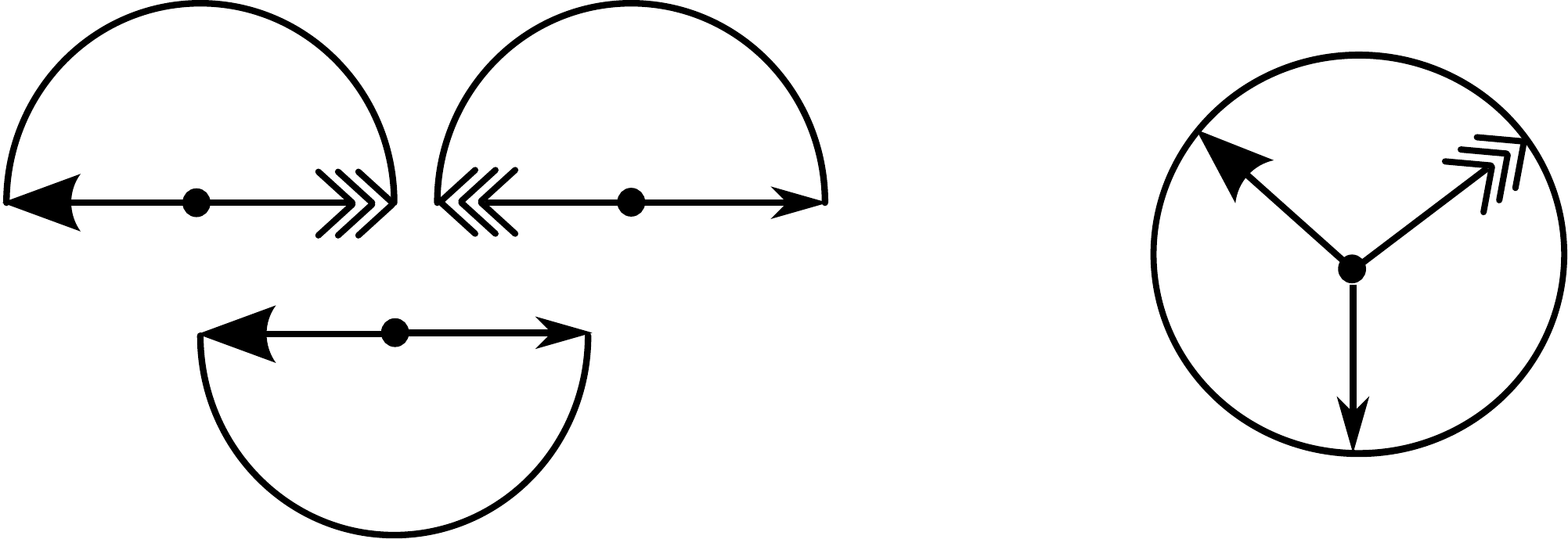}
\caption{An example of a cone angle $c(x)=3\pi$.}
\label{fig:tripolukruga}
\setlength{\unitlength}{1in}
   \begin{picture}(0,0)(0,0)
     \put(-0.3,1.2){$x$}
     \put(-1.12,1.2){$x$}
     \put(-0.78,0.75){$x$}
     \put(1,1.07){$x$}
   \end{picture}
\end{figure}
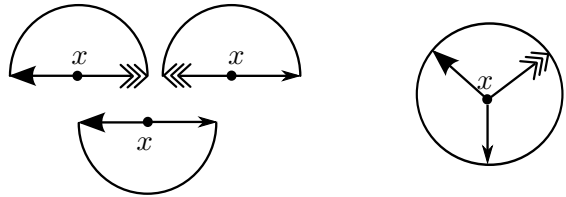


\begin{figure}[htb]
\centering
\includegraphics[height=1in]{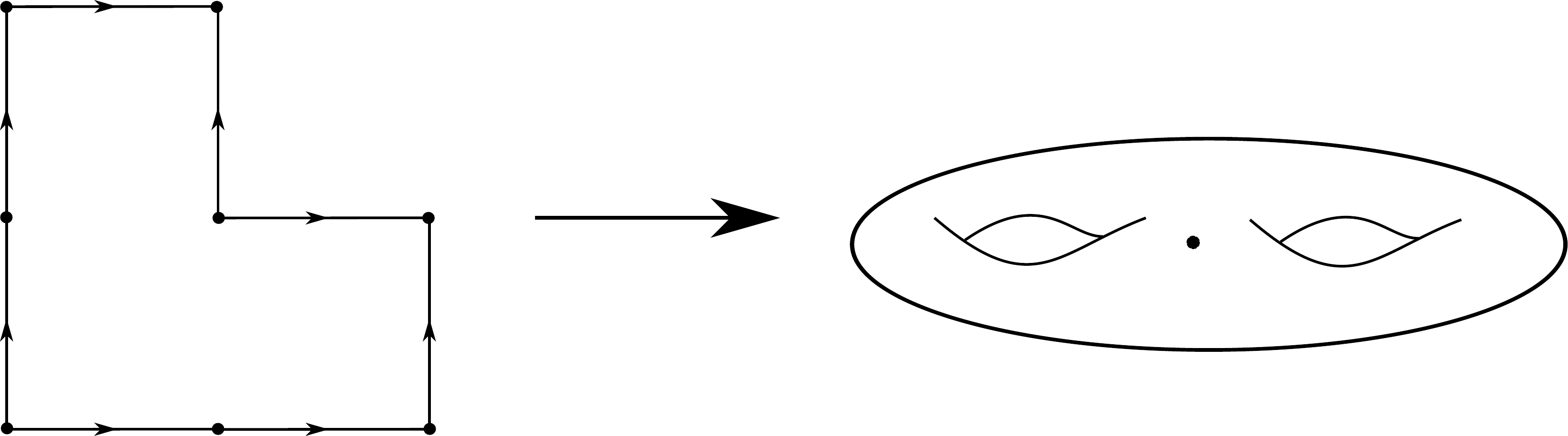}
\caption{An example of a surface with Euclidean cone metric obtained by gluing sides of a polygon in $\mathbb{C}$ by translations as indicated. The result is a surface of genus 2 with a single cone point $x$ with $c(x)=6\pi$.}
\label{fig:genus2}
 \begin{picture}(0,-13)(0,-13)
 	 \put(-138,100){$1$}
     \put(-91,100){$1$}
      \put(-138,65){$2$}
     \put(-55,65){$2$}
      \put(-115,125){$3$}
     \put(-115,41){$3$}
      \put(-81,90){$4$}
     \put(-81,41){$4$}
     \put(65,87){$x$}
   \end{picture}
\end{figure}
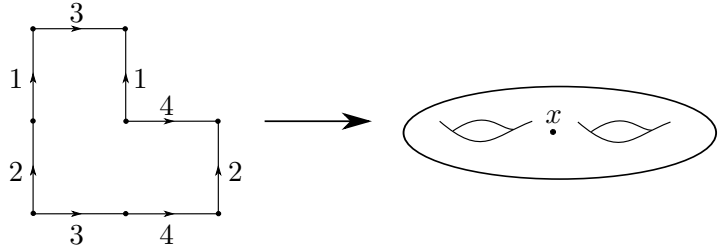

The \emph{holonomy} homomorphism associated to $m$ at any point $y \in S\setminus X$ is a homomorphism $$\rho_y : \pi_1 (S\setminus X,y) \rightarrow O(T_y(S\setminus X))$$ where $O(T_y(S\setminus X))$ is a group of orthogonal transformations of the tangent space of $S\setminus X$ at $y$. This is obtained by parallel translating a vector in $T_y(S\setminus X)$ along a loop in $S\setminus X$ based at $y$. Since our surface is oriented, the image is a subgroup of $SO(T_y(S\setminus X))$ - the group of rotations. \\

An orientation preserving isometry $\phi:T_y(S\setminus X) \rightarrow \mathbb{C}$ determines an isomorphism $SO(T_y(S\setminus X)) \rightarrow SO(2)$ independent of the choice of isometry $\phi$. We therefore view the holonomy homomorphism as a homomorphism to $SO(2)$. For a Euclidean cone metric $m$ define $Hol =Hol(m) \leq SO(2)$ to be the image of the holonomy homomorphism.\\ 

We will construct Euclidean cone surfaces by gluing sides of polygons by maps $\{ \rho_i \circ \tau_i \}_{i=1}^k$, which are compositions of translations $\tau_i$ and rotations $\rho_i$. Given $\gamma \in \pi_1(S\setminus X, y)$, $\rho_y(\gamma)$ is given by the composition of the rotations for the side gluings of the sides of the polygons crossed by $\gamma$. Therefore  $Hol \leq \langle \rho_1,\rho_2,...,\rho_k  \rangle$. See Figure~\ref{fig:genus2hol}.\\
 
\begin{figure}[htb]
\centering
\includegraphics[height=1in]{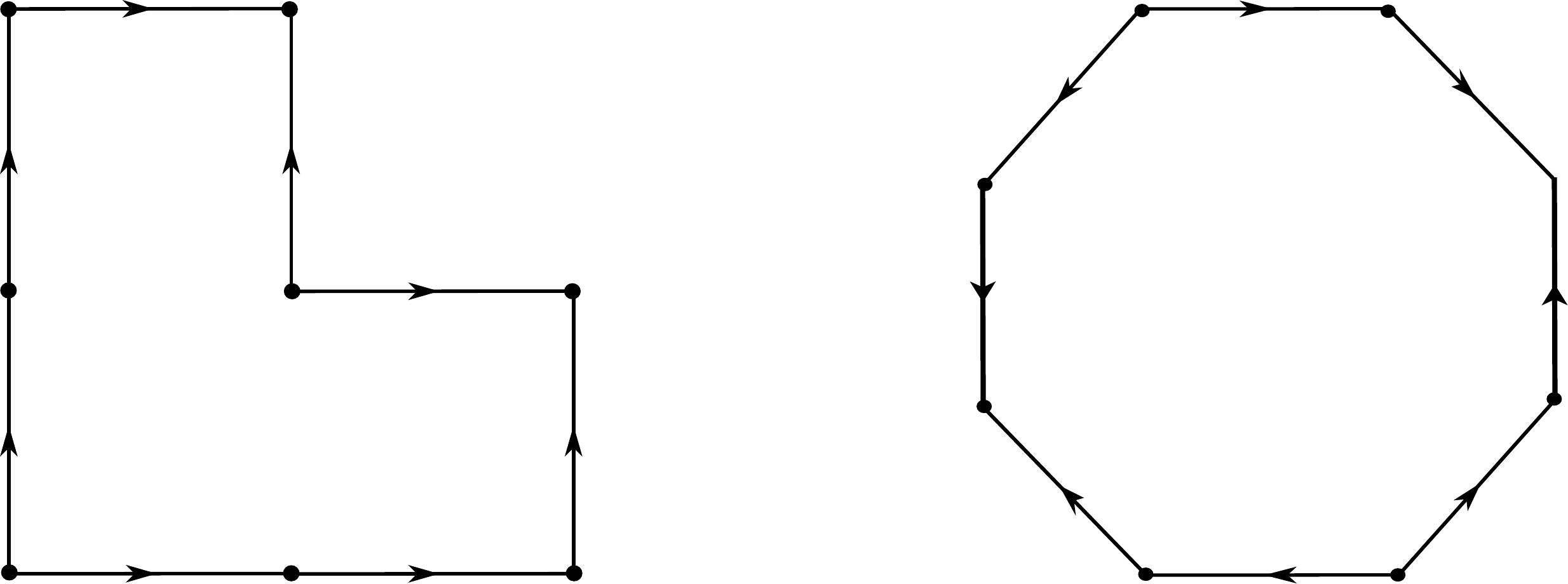}
\caption{Genus 2 surfaces with $Hol={Id}$ (on the left) and $Hol=\langle \rho_{\frac{\pi}{2}} \rangle$ (on the right).}
\label{fig:genus2hol}
\begin{picture}(0,0)(0,0)
 	 \put(-105,85){$1$}
     \put(-58,85){$1$}
      \put(-105,50){$2$}
     \put(-23,50){$2$}
      \put(-80,112){$3$}
     \put(-80,27){$3$}
      \put(-45,76){$4$}
     \put(-45,27){$4$}
     
      \put(56,112){$1$}
      \put(87,97){$2$}
      
      \put(56,27){$3$}
      \put(27,97){$4$}
      
     \put(88,43){$2$}
     \put(26,43){$4$}
     
     \put(98,69){$1$}
     \put(16,69){$3$}
   \end{picture}
\end{figure}
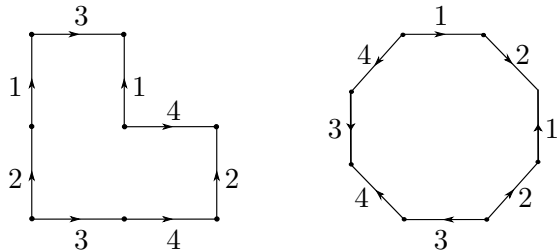

In fact, we can obtain any Euclidean cone surface by gluing sides of a single \emph{generalized Euclidean polygon} immersed in $\mathbb{C}$. To explain, consider a triangulation of $S$ by Euclidean triangles for which the vertex set is precisely the set of cone points.  Such a triangulation exists by Theorem 4.4. in \cite{masur}, for example (this is actually a $\Delta$--complex structure as in \cite{hatcher} instead of a proper triangulation, but the distinction is unimportant). We get a dual graph of the triangulation constructed by defining a vertex for each triangle and an edge for each pair of triangles that share an edge. See Figure~\ref{fig:dualgraph}. This graph has a maximal tree. Now by cutting along the edges of the triangles whose dual edges do not belong in the maximal tree we get a simply connected surface that is a union of Euclidean triangles which isometrically immerses in the plane. This is the generalized Euclidean polygon, which we denote $P$. The surface $S$ can be reconstructed from $P$ by gluing pairs of edges. See Figure~\ref{fig:immersed}. If we glue $P$ by translations and rotations $\{ \rho_i \circ \tau_i \}_{i=1}^k$, then $Hol=\langle \rho_1,\rho_2,\ldots, \rho_k  \rangle$.\\

\begin{figure}[htb]
\centering
\includegraphics[height=1in]{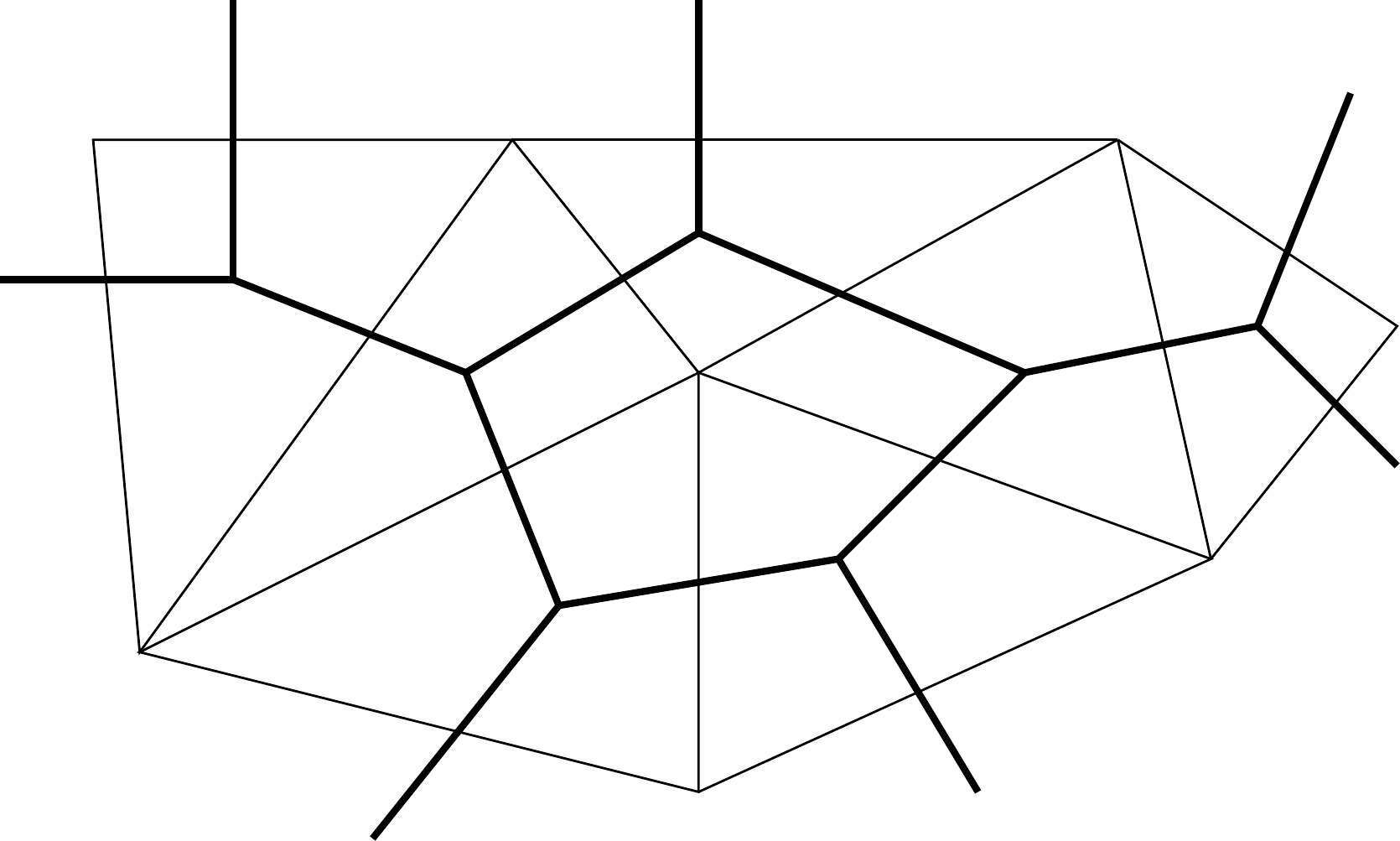}
\caption{The dual graph of a triangulation.}
\label{fig:dualgraph}
\end{figure}

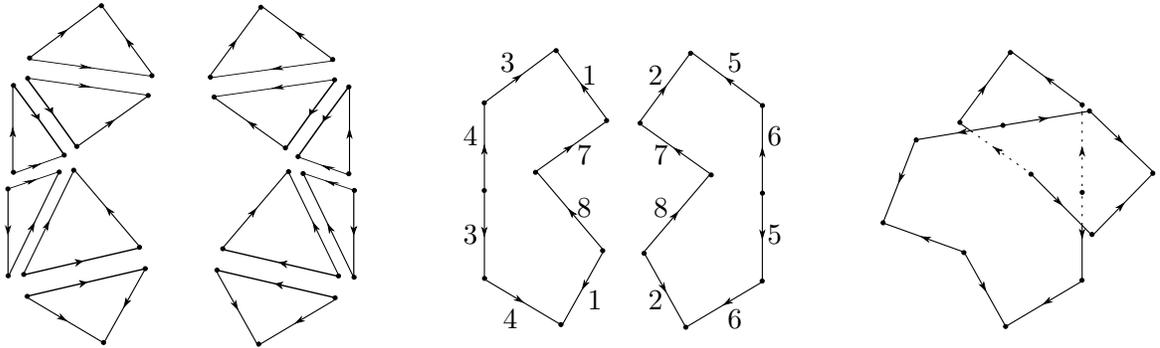
\begin{figure}[htb]
\centering
\includegraphics[height=1.8in]{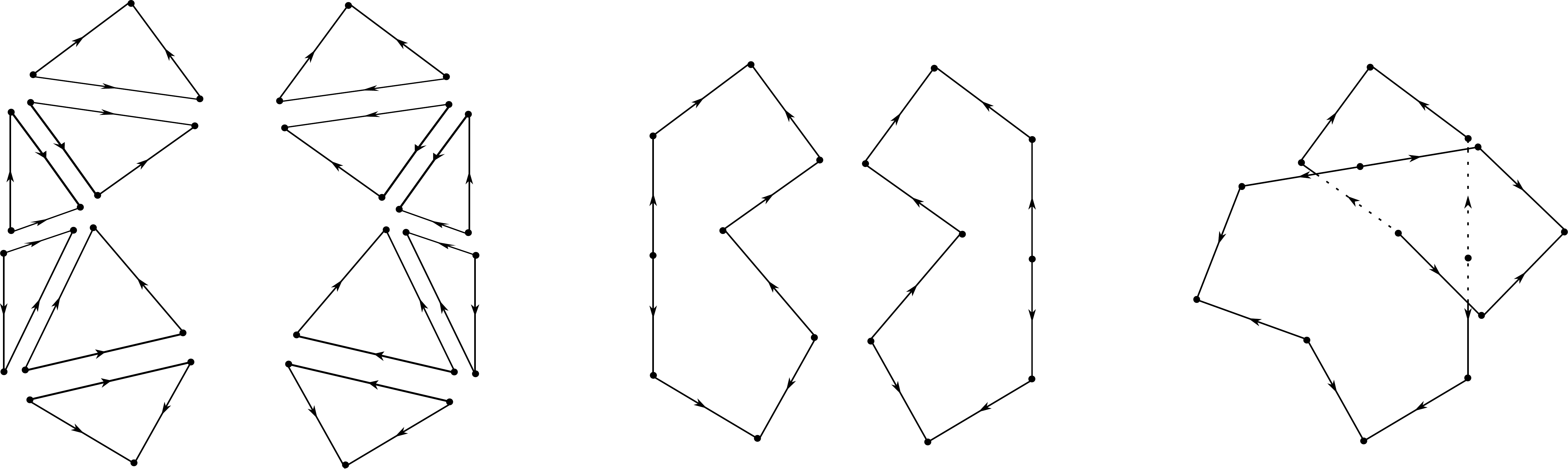}
\caption{A genus 2 surface obtained by gluing triangles (on the left), two polygons (in the middle) and by gluing an immersed polygon (on the right).}
\label{fig:immersed}
\begin{picture}(-8,0)(-8,0)
 	 
     \put(0,150){$1$}
      \put(25,150){$2$}
      \put(55,155){$5$}
      \put(70,127){$6$}
      
       \put(70,90){$5$}
      \put(2,65){$1$}
      \put(25,65){$2$}
      \put(55,58){$6$}
      
      \put(-31,155){$3$}
      \put(-45,127){$4$}
        \put(-45,90){$3$}
        \put(-30,58){$4$}
        
        \put(-2,120){$7$}
        \put(27,120){$7$}
        
        \put(-2,100){$8$}
        \put(27,100){$8$}
     \end{picture}
\end{figure}

Another important tool for us is the following well known fact:

\begin{prop} \label{prop:Gauss} \emph{(Gauss-Bonnet formula)} Let $R$ be a closed surface of genus $g \geq 0$ equipped with a Euclidean cone metric $m$. Then $$2\pi \chi(R)= \displaystyle \sum_{x \in X}(2\pi-c(x))$$ where $X$ is the set of cone points. \end{prop}

\noindent The  proof of the Gauss-Bonnet formula  for compact surfaces with geometric structure of constant curvature can be found in \cite{neumann}.\\

\subsection {Flat Metrics}

A Euclidean cone metric is called NPC (non-positively curved) if it is locally CAT(0). By the Gromov Link Condition this is equivalent to $c(x)\geq 2\pi$ for all $x\in S$ (See \cite{bridson}). For example the surface in Figure~\ref{fig:genus2} has one cone point and its angle is $6\pi$, and so is NPC.\\

Define \begin{quote} $Flat (S) =\{m \, | \, m$  is NPC Euclidean  cone metric on $S\}.$ \end{quote}

\noindent By a metric on $S$ we mean a metric inducing the given topology.\\

We are interested in the following class of metrics: For any $q \in \mathbb{Z_+}$ define
$$Flat(S, q)=\{m\in Flat(S)\, | \, Hol(m) \leq \langle \rho_\frac{2\pi}{q} \rangle \}$$
where $\rho_\theta$ is a rotation by angle $\theta$.\\

Alternatively, $Flat(S,q)$ is the space of metrics coming from $q$--differentials on $S$. For any $q \in\mathbb{Z}_+$, a $q$--differential is a complex structure and a family of holomorphic functions $\varphi_j$ on $z_j(U_j)$ for all coordinate neighborhoods $(U_j, z_j)$ of $S$, so that on $U_j\cap U_i$ they satisfy $$\varphi_j(z_j)=\varphi_i(z_i)\left(\frac{dz_i}{dz_j}\right)^q.$$  See  Chapter $II$ of \cite{farkas}, for example. It is customary to denote the $q$--differential $\varphi$ in coordinates $z_i$ by $\varphi = \varphi_i(z_i) dz_i^q$.\\

To explain how to get a metric in $Flat(S,q)$, suppose we are given a complex structure and a holomorphic $q$--differential $\varphi$. We can pick a small disk neighborhood $U$ of any point $p_0\in S$ with $\varphi(p_0)\not=0$, containing no zeros of $\varphi$, and define preferred coordinates $\zeta$ for $\varphi$ by $$\zeta(p)=\displaystyle \int_{p_0}^p \sqrt[q]{\varphi}.$$ In these coordinates $\varphi=d \zeta^q$. Let $\zeta_1,\zeta_2$ be two preferred coordinates, so that on the overlap of their domains we have $d \zeta_1^q = d \zeta_2^q$.  Since this is possible if and only if $\zeta_2=e^{\frac{2\pi ik}{q}}\zeta_1+w$ for some $k\in\mathbb{Z}_+$ and $w\in\mathbb{C}$, the preferred coordinates give an atlas of charts on $S \setminus \{zeros(\varphi)\}$ to $\mathbb{C}$ with transition functions of the form $T(z)=e^{\frac{2\pi ik}{q}}z+w$ where $k\in\mathbb{Z}_+$, $w\in\mathbb{C}$. Pulling back the Euclidean metric we get Euclidean metric on $S\setminus\{zeros(\varphi)\}$. The completion of this metric is obtained by adding back in $\{zeros(\varphi)\}$ and at a zero of order $k$ we have a cone angle $2\pi+\frac{2\pi k}{q}$. Therefore, the metric lies in $Flat(S, q)$. Conversely, take any metric in $Flat(S, q)$, choose local coordinates away from the singularities which are local isometries and so that the transition functions are translations and rotations by integer multiples of $\frac{2\pi}{q}$. Since these are holomorphic transformations preserving $dz^q$, this determines a complex structure and $dz^q$ determines a holomorphic $q$--differential, and this extends over the singularities (compare with \cite{minsky} and \cite{strebel}, for example, for the case $q=2$).\\

To give some idea of how ``big'' $Flat(S,q)$ is, we calculate the dimension. By the Riemann--Roch Theorem the dimension of the space of holomorphic $q$--differentials on a Riemann surface of genus $g$ is $2(2q-1)(g-1)$.  The space $\mathcal Q_q$ of all holomorphic $q$--differentials on $S$ is a vector bundle over Teichm\"uller space, and since every two $q$--differentials that differ by some rotation define the same metric on $S$, we get the real dimension of $Flat(S,q)$: $$dim(Flat (S,q))=dim(\mathcal{Q}_q)+dim(\mathcal{T}(S))-1$$
 $$\hspace{2.9cm}=2(2q-1)(g-1)+6g-7.$$ 
 
\vspace{0.5cm}

\noindent For every $q\in\mathbb{Z}_+$,  $Flat(S,q) \subset Flat(S)$ and $\displaystyle \lim_{q \to \infty}dim(Flat(S,q))=\infty$. Thus, $$dim(Flat(S))=\infty.$$\\

Observe that $q_1|q_2$ if and only if $\rho_{\frac{2\pi}{q_1}}\in \langle \rho_\frac{2\pi}{q_2} \rangle$. In fact $q_1|q_2$ if and only $Flat(S,q_1)\subseteq Flat(S,q_2)$.\\

\subsection{Closed curves}

Given a metric $m$ and a homotopy class of curve $\gamma \in \mathcal{C}(S)$  we will define the length function: $l_m(\gamma)= \inf \{l_m(c) | c \in \gamma \}$. For every curve $\gamma \in \mathcal{C}(S)$ there is a geodesic representative on $S$ due to Arzela--Ascoli theorem. Therefore $l_m(\gamma)$ is the length of its $m$-geodesic representative.\\

\begin{prop} For $m$ in $Flat(S)$, a closed curve $\gamma$ is an $m$-geodesic if and only if $\gamma$ is a closed Euclidean geodesic or a concatination of Euclidean segments between cone points such that angles between consecutive segments are $\geq \pi$ on each side of the curve $\gamma$. (See Figure~\ref{fig:geodesics}.)  \end{prop}

\begin{figure}[htb]
\centering
\includegraphics[height=1.2in]{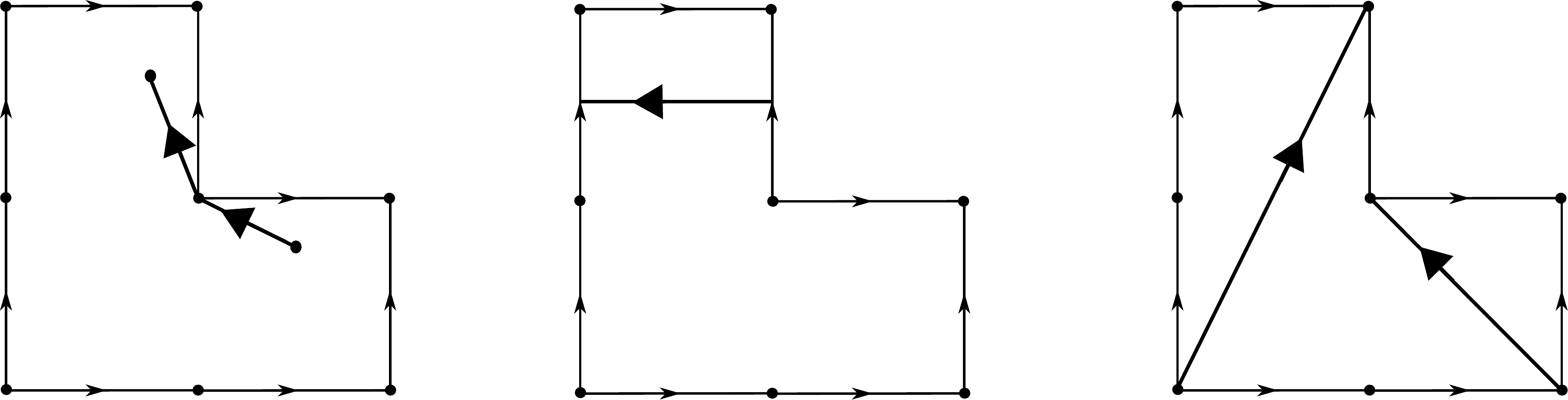}
\caption{A geodesic through a cone point (on the left), a closed geodesic with no cone points (in the middle) and a closed self-intersecting geodesic containing a cone point (on the right). See Figure~\ref{fig:genus2} for the gluings.}
\label{fig:geodesics}
\end{figure}

\begin{proof}
Assume $\gamma$ is a geodesic. Away from cone points $m$ is Euclidean, thus in the complement of the cone points geodesics are straight Euclidian segments. If $\gamma$ enters a cone point $x$ and exits at an angle less than $\pi$ we can find a path shorter than $\gamma$ in the neighborhood of the cone point. Therefore all geodesics have to make an angle greater than equal to $\pi$ on both sides around a cone point.\\

Conversely, because of the non-positive curvature, to see that paths satisfying this conditions are geodesics we just need to show that they locally minimize the length. For this we only need to check near the cone points. Let $x$ be a cone point on $\gamma$ and denote each ray of $\gamma$ coming out of $x$ inside a small ball $B$ around $x$ containing no other cone point, with $\gamma_-$ and $\gamma_+$. Construct two different straight line rays starting at $x$ and making angles $\frac{\pi}{2}$ with $\gamma_-$ on either side, and do the same for $\gamma_+$. See   Figure~\ref{fig:projection}. Notice that these rays define two non-intersecting neighborhoods of $\gamma_-$ and $\gamma_+$ since angles on each side of $\gamma$ at $x$ are greater than or equal to $\pi$. Now define a projection inside $B$ in the following way. Every point in the region bounded by the two neighborhoods of  $\gamma_-$ and $\gamma_+$ projects orthogonally onto $\gamma$, and every other point maps to $x$. This projection is distance non-increasing since orthogonal projection and projection to a point do not increase distances (and the two projections agree on the overlap of their domains). Therefore $\gamma$ is a local geodesic and that completes the proof.
 \end{proof}

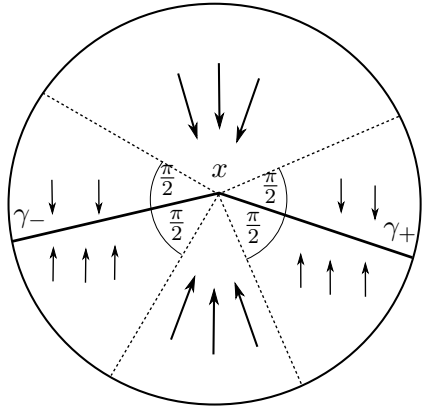
\begin{figure}[htb]
\centering
\includegraphics[height=2.2in]{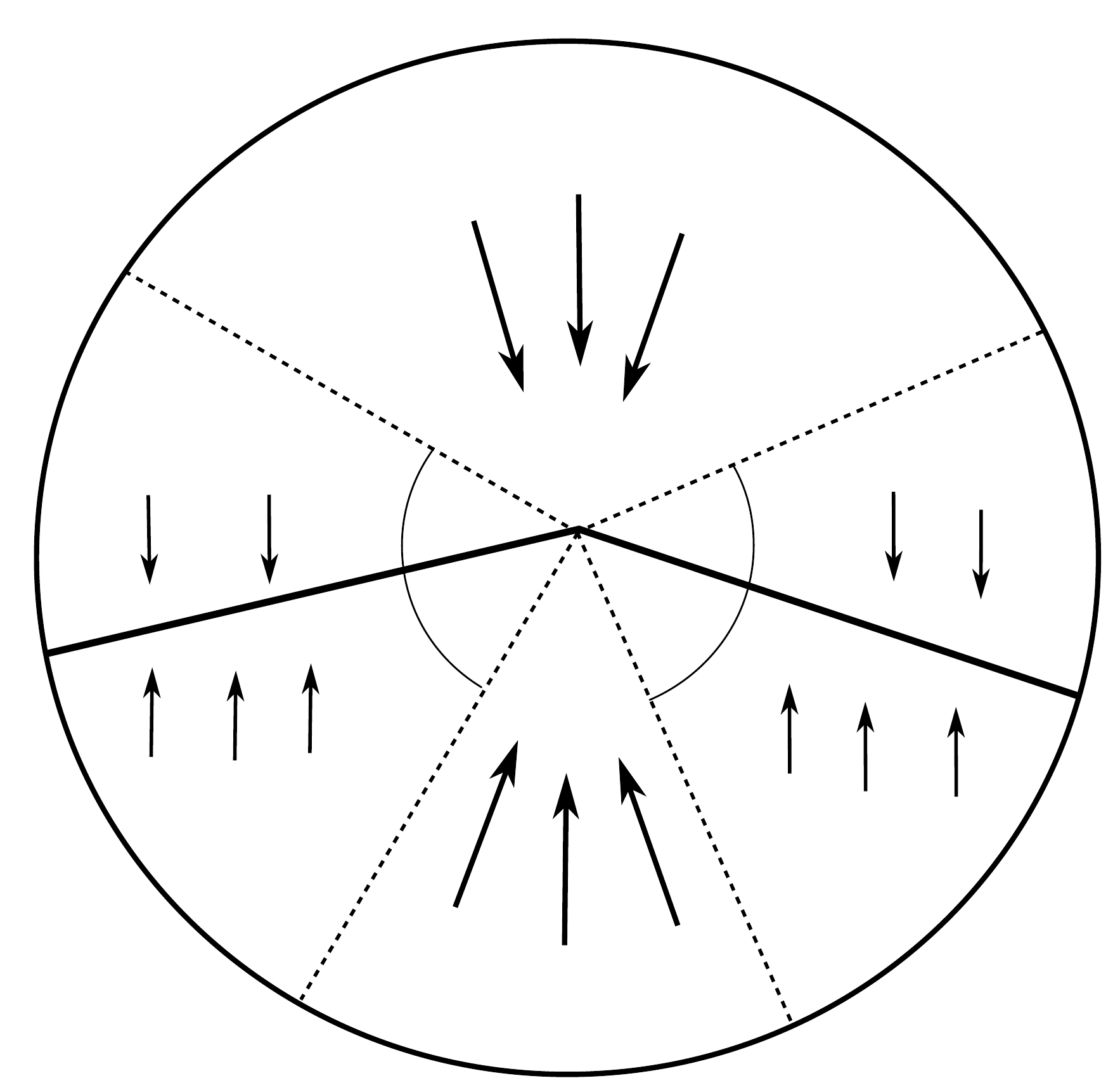}
\caption{Projection onto $\gamma$.}
\label{fig:projection}
\begin{picture}(0,2.5)(0,2.5)
 	 
     \put(0,127){$x$}
    \put(-20.5,124){$\frac{\pi}{2}$}
    \put(18.5,121){$\frac{\pi}{2}$}
    
    \put(-17,107){$\frac{\pi}{2}$}
    \put(12,106){$\frac{\pi}{2}$}
    
    \put(-75,110){$\gamma_-$}
     \put(65,105){$\gamma_+$}
     \end{picture}
\end{figure}

By Theorem~II.6.8 (4) in \cite{bridson}, if an $m$-geodesic representative of a curve in $\mathcal{C}(S)$ is not unique in its homotopy class then the set of geodesic representatives foliates a cylinder in $S$, and each geodesic representative has the same length in $m$. (See Figure~\ref{fig:geodesics2}.)\\

\begin{figure}[htb]
\centering
\includegraphics[height=1.2in]{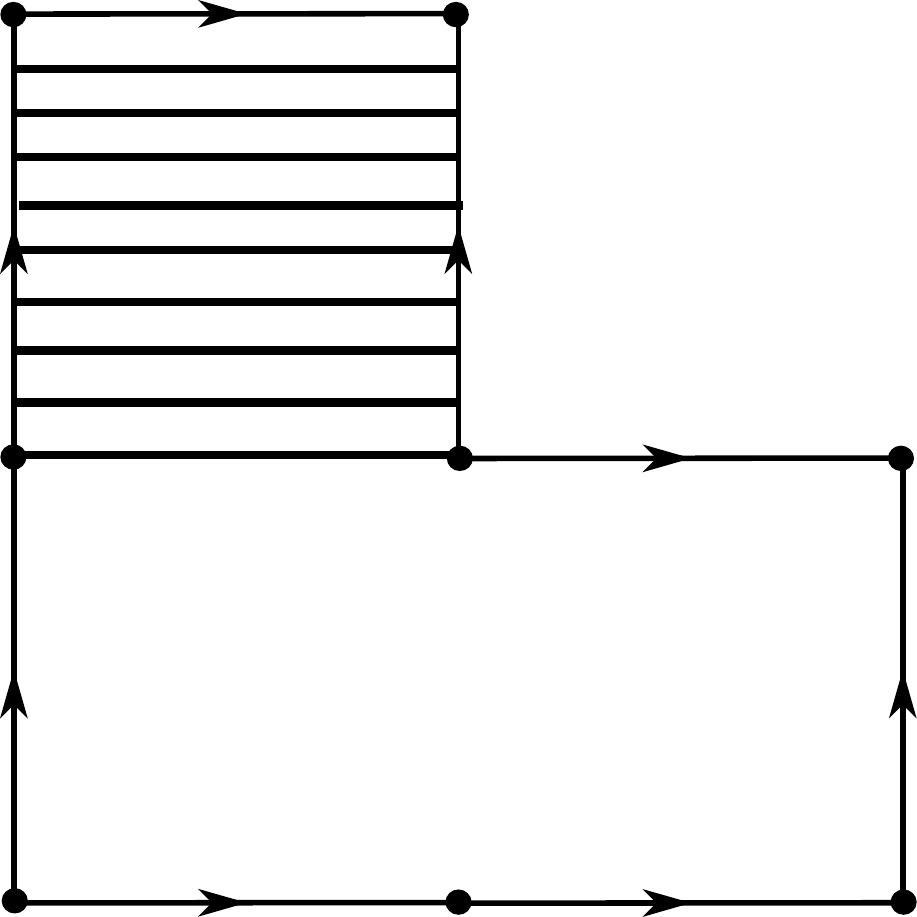}
\caption{A cylindar on a genus $2$ surface foliated by closed geodesics.  See Figure~\ref{fig:genus2} for the gluings.}
\label{fig:geodesics2}
\end{figure}

For two curves $\alpha, \beta \in \mathcal{C}(S)$ we can define the geometric intersection number, $i(\alpha,\beta)=$ minimum number of double points of intersection of any two representatives of $\alpha$ and $\beta$. For hyperbolic metrics geodesic representatives realize geometric intersection number.  \\

Define $\mathcal{S}(S) \subset \mathcal{C}(S)$ to be the set of homotopy classes of simple closed curves on $S$.\\

\noindent Combined results of Thurston's \cite{FLP}, \cite{penner}  and Bonahon's  \cite{bonahon} yield the following:

\begin{thm}\label{thm:Thurston} Given any separating simple closed curve $\beta \in \mathcal{S}(S)$,  there is a sequence $\{(t_n,\beta_n)\}_{n=1}^{\infty} \subset \mathbb{R}_+ \times \mathcal{S}(S)$ where each $\beta_n$ is non-separating, so that for all $\alpha \in \mathcal{C}(S)$, $t_ni(\beta_n,\alpha) \rightarrow i(\beta, \alpha)$ as $n \rightarrow \infty$. \end{thm}

We say that $\gamma$ and $\gamma' \in \mathcal{C}(S)$ are simple intersection equivalent, $\gamma \equiv_{si} \gamma'$, if $i(\gamma,\alpha)=i(\gamma',\alpha)$ for every $\alpha\in \mathcal{S}(S)$. We will also need the next fact.

\begin{thm} \label{thm:clein} \cite{clein1} For every $\gamma$, $\gamma' \in \mathcal{C}(S)$, $\gamma \equiv_{h} \gamma '$ $\Rightarrow$ $\gamma \equiv_{si} \gamma'$ $\Leftrightarrow$ $\gamma \equiv_{2} \gamma '$.\end{thm}


\section{Relations  $\equiv_1$ and $\equiv_2$}
\label{sec:1and2}
We now turn to the proof of 

\begin{thm} \label{thm:12} For every $\gamma,\gamma'\in\mathcal{C}(S)$, $\gamma \equiv_{1} \gamma' \, \, \Leftrightarrow \, \, \gamma\equiv_{2}\gamma'$.\end{thm}

\begin{proof}[Proof:] 

Given $\gamma$, $\gamma'\in \mathcal{C}(S)$, we have  $$\gamma\equiv_{si} \gamma' \, \,  \Rightarrow \, \,  \gamma\equiv_{2}\gamma' \, \,  \Rightarrow \, \,   \gamma\equiv_{1}\gamma'$$ by Theorem~\ref{thm:clein} and the fact $Flat (S,1) \subseteq Flat (S, 2).$ We want to prove $$\gamma \equiv_{1} \gamma '\, \, \Rightarrow \, \, \gamma \equiv_{si}\gamma '.$$

We claim that if $\gamma \equiv_1 \gamma'$, then $i(\alpha, \gamma)=i(\alpha,\gamma')$ for every non-separating curve $\alpha$. Then if $\beta$ is a separating curve, by Theorem~\ref{thm:Thurston} there is a sequence of non-separating curves $\beta_n$ and positive real numbers $t_n$ such that $t_n i(\gamma,\beta_n) \to i(\gamma,\beta)$  and $t_n i(\gamma',\beta_n) \to i(\gamma',\beta)$. Since $t_n i(\beta_n, \gamma)= t_n i(\beta_n, \gamma')$ by the claim, we will have $i(\beta, \gamma)=i(\beta,\gamma')$ for every separating simple closed curve $\beta$, hence every $\beta\in\mathcal{S}(S)$, and thus $\gamma \equiv_{si} \gamma'$.\\

For any $g \geq 2$, we can construct a closed genus g surface  by gluing the arcs in the boundary of a cylinder so that in the resulting surface the core curve is non-separating. Moreover, this construction can be carried out on a Euclidean cylinder so that the resulting Euclidean cone metric has trivial holonomy. See Figure~\ref{fig:pravougaonik5}.\\

\begin{figure}[htb]
\centering
\includegraphics[height=1.25in]{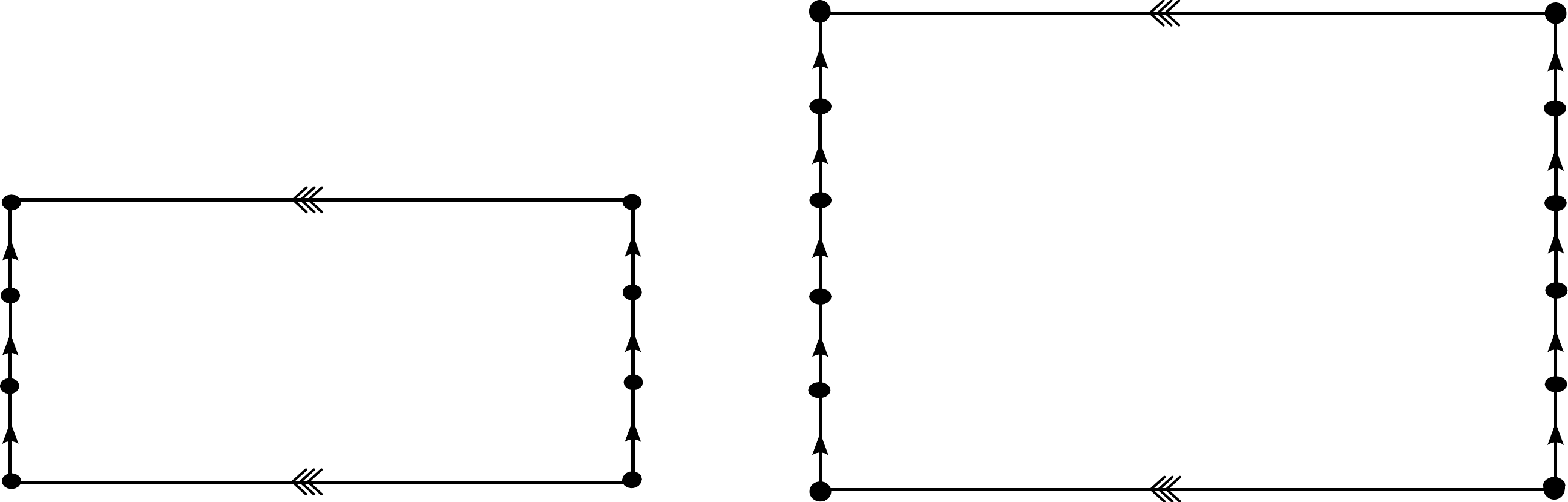}
\caption{Examples of genus 2 and 3 surfaces. Gluing top sides to the bottom sides results in a cylinder, the rest of the gluing produces the closed surfaces.}
\label{fig:pravougaonik5}
\begin{picture}(0,0)(0,0)
 	 \put(-150,94){$1$}
     \put(-150,75){$2$}
      \put(-150,57){$3$}
   
     \put(-23,94){$3$}
     \put(-23,75){$2$}
      \put(-23,57){$1$}
      
      \put(-3,125){$1$}
       \put(-3,108){$2$}
       \put(-3,93){$3$}
     \put(-3,75){$4$}
      \put(-3,57){$5$}
      
       \put(144,125){$5$}
       \put(144,108){$4$}
       \put(144,93){$3$}
     \put(144,75){$2$}
      \put(144,57){$1$}
     
   \end{picture}
\end{figure}
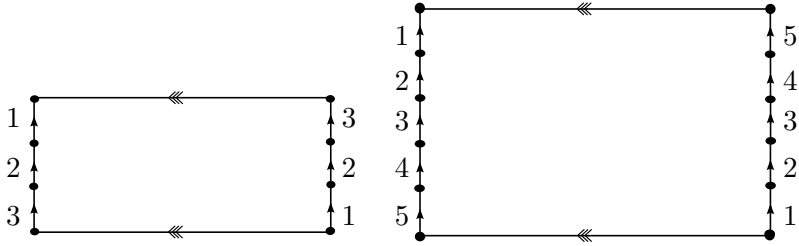

Suppose $S$ has genus $g$ and let $\alpha$ be a non-separating curve on $S$. Let $X_g^\epsilon$ be the surface of genus $g$ obtained by gluing a rectangle $Y_g^\epsilon$ as in Figure~\ref{fig:pravougaonik5} with horizontal side length $1$ and vertical side length $\epsilon >0$. Let $\alpha_g$ be the core nonseparating curve of the cylinder obtained by gluing only the horizontal sides of $Y_g^\epsilon$. See Figure~\ref{fig:pravougaonik4}. We assume that the obvious affine map from $Y_g^1$ to $Y_g^\epsilon$ descends to a homeomorphism $f_\epsilon:X_g^1 \to X_g^\epsilon$ for all $\epsilon > 0$. Choose a homeomorphism  $f:S \rightarrow  X_g^1$ so that $\alpha$ is sent to $\alpha_g$ and let 
 $h_\epsilon=f_\epsilon \circ f:S \rightarrow X_g^\epsilon$. Write $m_\alpha^\epsilon$ to denote the metric obtained by pulling back the Euclidean metric on $X_g^\epsilon$ via $h_{\epsilon}$. Since the edges of $Y_g^\epsilon$ are glued only by translations we have $m_\alpha^\epsilon \in Flat(S,1)$.\\

\begin{figure}[htb]
\centering
\includegraphics[height=1.25in]{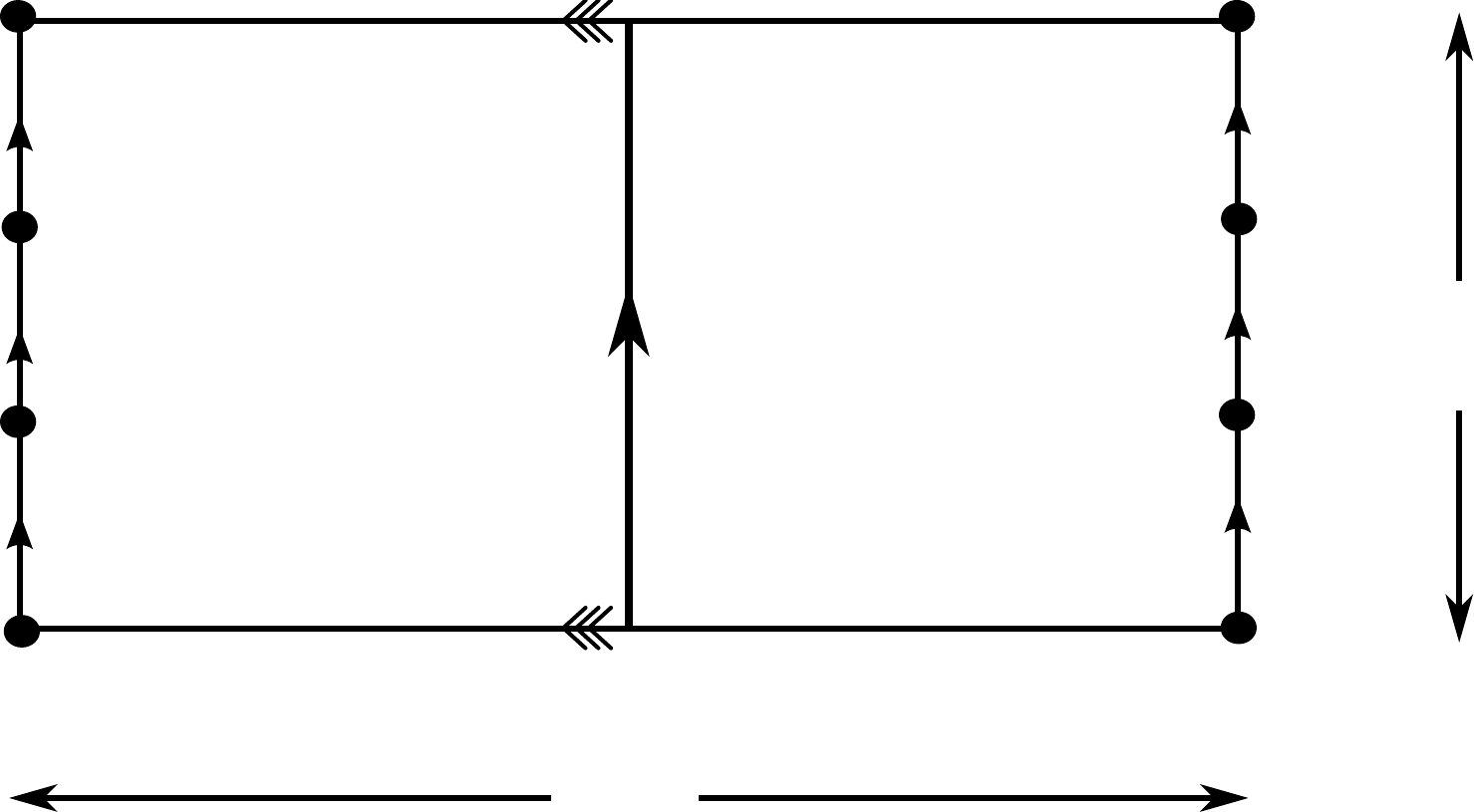}
\caption{The surface $X_2^{\epsilon}$ obtained by gluings from the rectangle  $Y_2^\epsilon$, and curve $\alpha_2$.}
\label{fig:pravougaonik4}
\begin{picture}(0,0)(0,0)
 	 \put(-7,88){$\alpha_2$}
     \put(-14,35){$1$}
     \put(79,86){$\epsilon$}
     
      \put(-90,110){$1$}
     \put(-90,88){$2$}
      \put(-90,65){$3$}
   
     \put(60,110){$3$}
     \put(60,88){$2$}
      \put(60,65){$1$}
     \end{picture}
\end{figure}
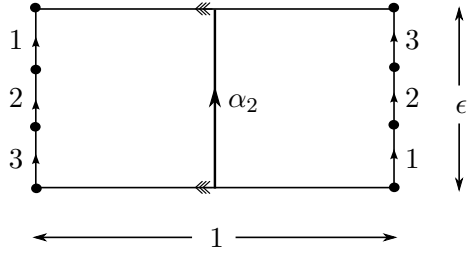

Let $\gamma$ be a curve on $S$. The geodesic representative of $\gamma$ in $m_\alpha^\epsilon$ is sent by $h_\epsilon $ to a closed geodesic on $X_g^\epsilon$. On the cylinder obtained by gluing the horizontal sides of $Y_g^\epsilon$ this is a union of straight lines running from one boundary component of the cylinder to another and along the boundary components.  See Figure~\ref{fig:pravougaonik1}. \\

\begin{figure}[htb]
\centering
\includegraphics[height=1in]{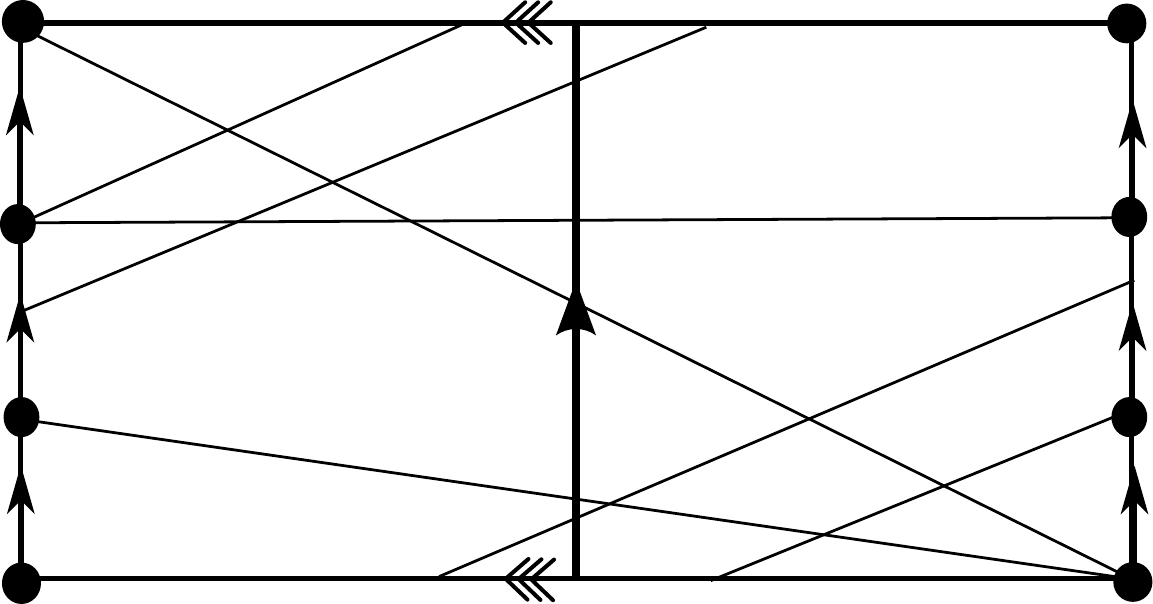}
\caption{ The geodesic representative of $\gamma$ in $m_\alpha^\epsilon$. See Figure~\ref{fig:pravougaonik4} for the gluings.}
\label{fig:pravougaonik1}
\end{figure}

Thus we have \begin{equation} \label{eqn:1} l_{m_\alpha^\epsilon}(\gamma)\geq i(\alpha, \gamma) \cdot 1 \tag{1} \end{equation} since each geodesic segment that crosses the cylinder contributes $1$ to intersection number and at least $1$ to the length.\\

The curve $\gamma$ is homotopic to a curve $\bar{\gamma}$ which is sent by $h_\epsilon$ to a union of straight line segments parallel to the side of length $1$ of the rectangle $Y_g^\epsilon$ and some segments of the vertical $\epsilon$-length sides as in Figure~\ref{fig:pravougaonik2}. \\

\begin{figure}[htb]
\centering
\includegraphics[height=1in]{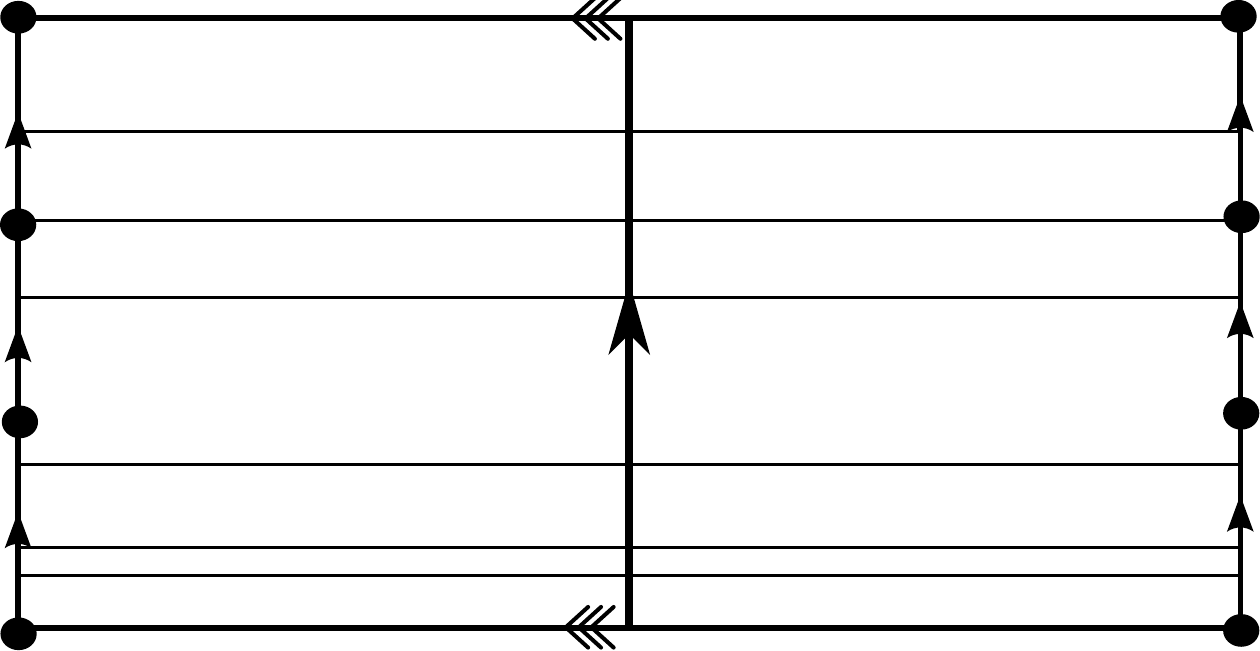}
\caption{The representative $\bar\gamma$ in $m_\alpha^\epsilon$. See Figure~\ref{fig:pravougaonik4} for the gluings.}
\label{fig:pravougaonik2}
\end{figure}

From this we get \begin{equation}\label{eqn:2} \l_{m_\alpha^\epsilon}(\gamma)\leq length(\bar{\gamma})\leq i(\alpha, \gamma)\cdot 1 + n_\gamma \epsilon \tag{2}\end{equation}  where $n_\gamma $ is the number of
vertical segments of $\bar{\gamma}$.\\

Now suppose $\gamma \equiv_{1} \gamma'$. Then for every $\epsilon > 0$ we have $l_{m_\alpha^\epsilon}(\gamma) = l_{m_\alpha^\epsilon}(\gamma')$. By~\eqref{eqn:1} and~\eqref{eqn:2}  $$ l_{m_\alpha^\epsilon}(\gamma)- n_\gamma \epsilon \leq i(\alpha, \gamma) \leq l_{m_\alpha^\epsilon}(\gamma)$$ and $$ l_{m_\alpha^\epsilon}(\gamma')- n_{\gamma'} \epsilon \leq i(\alpha, \gamma') \leq l_{m_\alpha^\epsilon}(\gamma').$$ Letting $\epsilon \to 0$ we get $i(\alpha, \gamma)=i(\alpha, \gamma')$, proving the claim. \\

Hence $\gamma \equiv_{si} \gamma'$, completing the proof. 
\end{proof}

\begin{cor} For every $q\in\mathbb{Z}_+$ and every $\gamma, \gamma'\in \mathcal{C}(S)$, \,  $\gamma \equiv_q \gamma' \, \, \Rightarrow \, \, \gamma \equiv_2 \gamma'$.\end{cor}

\begin{proof}[Proof:] Since $Flat(S,1) \subset Flat(S,q)$ for every $q$, we have $$\gamma \equiv_q \gamma' \, \Rightarrow \,\gamma \equiv_1\gamma' \, \Rightarrow \gamma \equiv_2 \gamma'.$$ 
\end{proof}


\section{The equivalence relation $\equiv_\infty$ }
\label{sec:infinity}

Next we will prove

\begin{thm}\label{thm:infinity} The equivalence relation $\equiv_{\infty}$ is trivial.\end{thm}

To prove this we first prove a weaker statement. Define another equivalence relation on $\mathcal{C}(S)$ by declaring  $\gamma \equiv_{\mathbb{R}} \gamma'$ if and only if $l_m(\gamma)=l_m(\gamma')$ for every $m\in Flat(S)$.

\begin{thm} \label{thm:R} The equivalence relation $\equiv_{\mathbb{R}}$ is trivial.\end{thm}

\begin{proof}[Proof:] To prove this theorem we will show that for every $\gamma, \gamma' \in \mathcal{C}(S)$, $\gamma \not= \gamma'$, there is $m \in Flat (S)$ so that $l_m(\gamma) \neq l_m(\gamma')$.\\

Let $\gamma$ and $\gamma'$ be distinct curves in $\mathcal{C}(S)$. Pick a hyperbolic Riemannian metric $g$ on $S$. If the geodesic representatives of $\gamma$ and $\gamma'$ have the same length in $g$, let $\epsilon>0$ be small enough so that $\gamma'$ does not intersect $\epsilon$-ball around a point $x \in \gamma$. Let $\varphi:S \rightarrow \mathbb{R}$ be a smooth function so that $\varphi \equiv 0$ on $S\setminus B_{\epsilon}(x)$, $\varphi (x)=1$ and $\varphi (S)=[0, 1]$. For $\delta>0$ consider the Riemannian metric $g_{\delta}=g(1-\delta\varphi)$. The metric $g_{\delta}$ is negatively curved for $\delta$ sufficiently small and  the $g_\delta$-length of $\gamma$ is smaller than the $g$-length, while the length of $\gamma'$ is not changed. Thus, for any two curves $\gamma$ and $\gamma'\in \mathcal{C}(S)$ there is a negatively curved Riemannian metric $m'$ such that $l_{m'}(\gamma) \neq l_{m'}(\gamma')$. In fact there are negatively curved metrics where there are no closed curves with the same length (see \cite{randol}, \cite{abraham}, \cite{anosov} for more general result).\\

Being a locally $CAT(k)$ space, $k<0$, implies also being locally $CAT(0)$. Thus $m'$ is a locally $CAT(0)$ metric. Pick a geodesic triangulation of the surface $S$ with metric $m'$ so that each triangle belongs to a $CAT(0)$ neighborhood on $S$ and so that the $m'$-geodesic representatives of $\gamma$ and $\gamma'$ are unions of edges of triangles.  Now build a Euclidean cone metric $m$ on $S$ by replacing the triangles, whose edges are $m'$-geodesics, with Euclidean triangles with the same length sides. By the $CAT(0)$ property the angles of the Euclidean triangles are larger than the angles in the triangles they replaced. The vertices of the triangles are the only possible cone points. There are finitely many of them, and their angles are therefore $\geq 2\pi$. Hence $m \in Flat(S)$. \\

\begin{figure}[htb]
\centering
\includegraphics[height=0.9in]{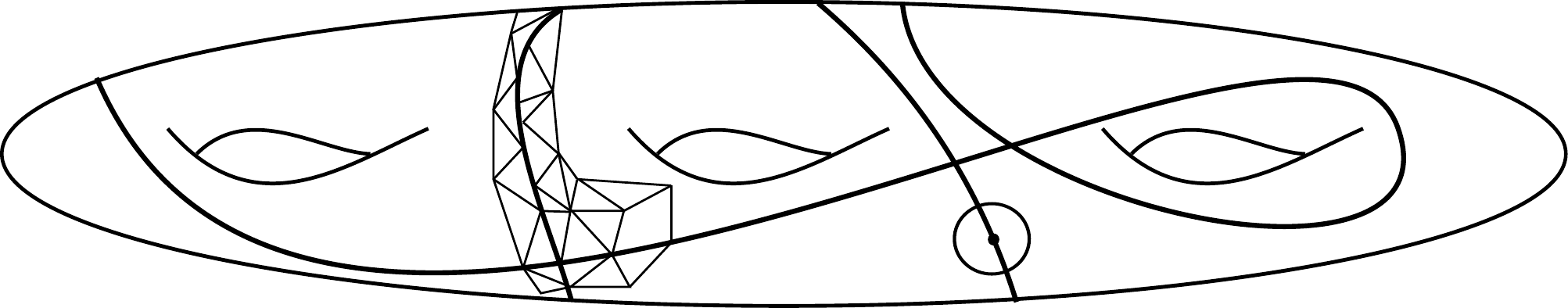}
\caption{Part of a triangulation of a genus 3 surface with curves $\gamma$ and $\gamma'$ as unions of edges of the triangles.}
\label{fig:triangulation}
\begin{picture}(0,-12)(0,-12)
 	 \put(52,45){$B_{\epsilon}(x)$}
     \put(10,86){$\gamma$}
      \put(-68,86){$\gamma$}
     \put(65,82){$\gamma'$}
     \end{picture}
\end{figure}
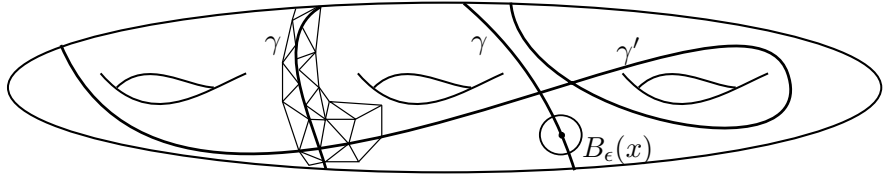

The curves $\gamma$ and $\gamma'$ in the new metric become concatinations of Euclidean segments (sides of triangles). Since the sum of angles of a cone point on one side of the curve in the metric $m$ is  greater than or equal to the sum of the angles in the metric $m'$ we get that the angle around every cone point on each side of $\gamma$ and $\gamma'$ is $\geq \pi$. Therefore $\gamma$ and $\gamma'$ are also geodesics in the metric $m$.\\

\begin{figure}[htb]
\centering
\includegraphics[height=0.9in]{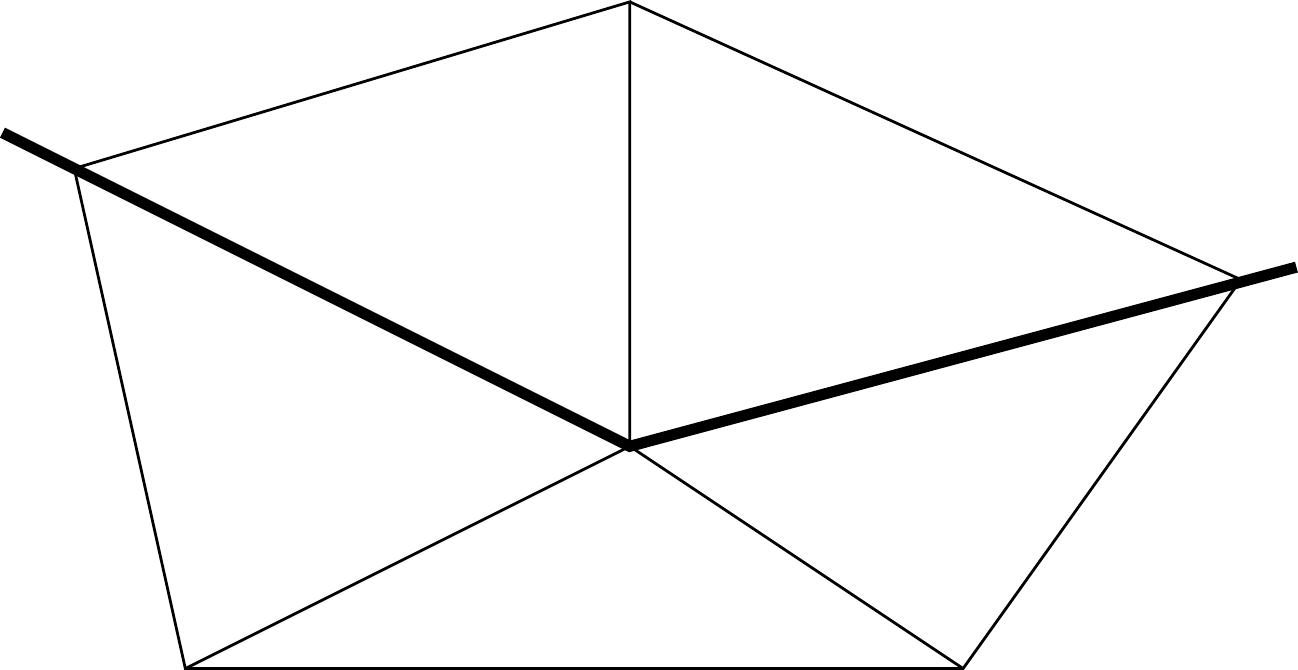}
\caption{Euclidean triangles and a geodesic segment containing edges of the triangles.}
\label{fig:triangulation2}
\begin{picture}(0,0)(0,0)
 	 \put(-15,67){$\theta_1$}
     \put(0,67){$\theta_2$}
      \put(-23,55){$\theta_3$}
     \put(-8,46){$\theta_4$}
     \put(7,52){$\theta_5$}
     
     \put(65,80){$\theta_1+\theta_2\geq \theta_1'+\theta_2'=\pi  $}
     \put(65,60){$\theta_3+\theta_4+\theta_5\geq \theta_3'+\theta_4' +\theta_5'=\pi  $}
     \end{picture}
\end{figure}
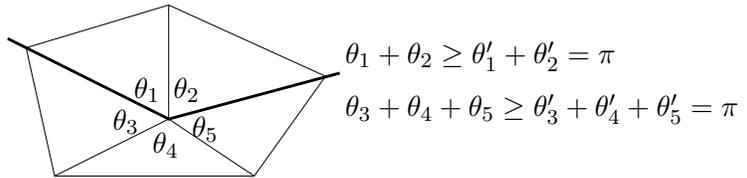

We have
$$l_{m}(\gamma) = l_{m'}(\gamma) \neq l_{m'}(\gamma')=l_{m}(\gamma'),$$ and we proved our theorem.
\end{proof}

\begin{thm} \label{thm:polygon} $\displaystyle\overline{\bigcup _{q \in \mathbb{Z}_+}Flat(S,q)}=Flat(S)$. More precisely, for every $m\in Flat(S)$, there exist a sequence of metrics $\displaystyle m_n \in \bigcup _{q \in \mathbb{Z}_+}Flat(S,q)$, so that $id:(S, m_n) \rightarrow (S, m)$ is $K_n$-bilipschitz and $K_n \rightarrow 1$ as $n\rightarrow \infty$. \end{thm}

\begin{proof}[Proof:] Let $m \in Flat (S)$. Take a triangulation of $S$ by Euclidean triangles with all vertices being cone points and all cone points being vertices.  Using this triangulation, view $S$ as obtained from a generalized Euclidean polygon $P$ isometrically immersed in $\mathbb{C}$ by gluing the edges in pairs by isometries as described in Section~\ref{sec:Background}. We will construct an isometrically immersed polygon $P_\epsilon$ so that the sides of $P_\epsilon$ meet the real axis at angles that are rational multiples of $\pi$, and  $P_\epsilon \to P$ as $\epsilon \to 0$. Since two sides of $P_\epsilon$ which are glued together make angles with the real axis that are rational multiples of $\pi$, they are glued by translating and rotating by an angle in $\mathbb{Q}\pi$. Therefore the holonomy of the flat metric $m_\epsilon$ on $S$, which we get by gluing the sides of the immersed polygon $P_\epsilon$, is generated by rotations by angles in $\mathbb{Q}\pi$. Thus $m_\epsilon$ belongs in $Flat(S,q_\epsilon)$ for some $q_\epsilon\in\mathbb{Z}_+$.\\

To do this construction, we first orient the edges $s_i$ of the polygon $P$ using the boundary orientation coming from the immersion into $\mathbb{C}$. Immersion in  $\mathbb{C}$ makes each oriented edge $s_i$ into an oriented line segment which we may view as a vector, or equivalently a complex number which we denote $z_i$.  Let $\theta_i$ be the argument of $z_i$. We have $z_1+z_2+...+z_{2n}=0$. Assume, by rotating if necessary, that $\theta_1=0$ and by relabeling if necessary  that $z_{2i-1}$ and $z_{2i}$ correspond to the edges that are glued together, for all $i=1,\ldots,n$. We have $|z_1|=|z_2|,$ $|z_3|=|z_4|,\ldots, |z_{2n-1}|=|z_{2n}|$. See Figure~\ref{fig:poligon}.\\

\begin{figure}[htb]
\centering
\includegraphics[height=2in]{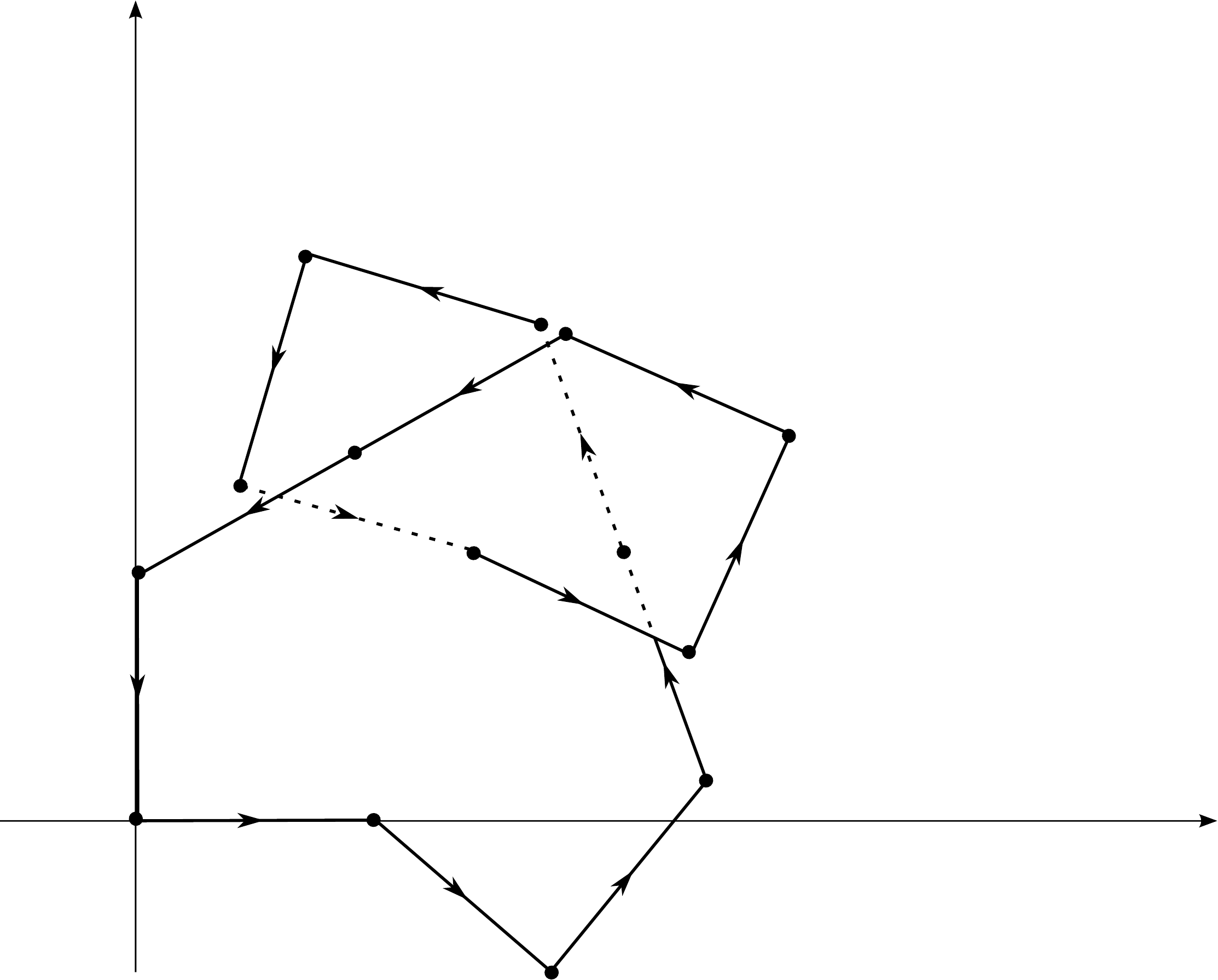}
\caption{An immersed polygon in $\mathbb{R}^2$ with one edge on the $x$-axis.}
\label{fig:poligon}
\begin{picture}(0,0)(0,0)
\put(75,50){$x$}
\put(-80,170){$y$}
\end{picture}
\end{figure}
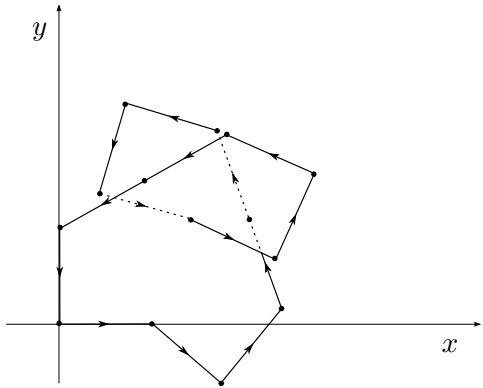

If $z_{2i-1}=-z_{2i}$, for all $i=1,\ldots, n$, then $Hol(S) =\{Id\}$ and therefore $m\in Flat(S,1)$ and we set $m_n=m$ for all $n$.\\

Therefore we assume that there is some $i$ so that $z_{2i-1}\not=-z_{2i}$. After changing indices  if necessary we may assume $z_{2n-1}\not=-z_{2n}$. Fix $\displaystyle 0<\epsilon<\frac{|z_{2n-1}+z_{2n}|}{2} $. Denote $\tau= - z_1- \ldots  - z_{2n-2}$. Since $\tau=z_{2n-1} + z_{2n}$ it follows that $|\tau| > 2\epsilon$. Let $\widetilde{z}_1 = z_1$. For $2 \leq j \leq 2n-4$ choose $\widetilde{z}_j$ so that $|z_j|= |\widetilde{z}_j|$, $\widetilde{\theta}_j \in \mathbb{Q} \pi$ and $|z_j - \widetilde{z}_j| < \frac{\epsilon}{2n}$. Now, choose $\widetilde{z}_{2n-3}$, $\widetilde{z}_{2n-2}$ so that $|\widetilde{z}_{2n-3}| =|\widetilde{z}_{2n-2}|$, $\widetilde{\theta}_{2n-3},$  $\widetilde{\theta}_{2n-2} \in \mathbb{Q}\pi$, $|z_i-\widetilde{z}_i|<\frac{\epsilon}{2n}$ for $i=2n-3, 2n-2$, and also so that the argument of  $\widetilde{z}_1  + \ldots +  \widetilde{z}_{2n-3} + \widetilde{z}_{2n-2}$ is in $\mathbb{Q} \pi$. We arrange this in the following way: First construct $z_i'$, $i=2n-3,2n-2$, so that their arguments are in $\mathbb{Q} \pi$ and $|z_i'|= |z_i|$, $|z_i'-z_i|<\frac{\epsilon}{4n}$ for $i=2n-3,\, 2n-2$. Then construct $\widetilde{z}_i$, $i=2n-3,\, 2n-2$, so that $Arg(z_i')=Arg(\widetilde{z}_i)$, $|z_i'-\widetilde{z}_i|<\frac{\epsilon}{4n}$ and the argument of $\widetilde{z}_1  + \ldots +  \widetilde{z}_{2n-3} + \widetilde{z}_{2n-2}$ is in $\mathbb{Q} \pi$. See Figure~\ref{fig:vectors1}.\\

\begin{figure}[htb]
\centering
\includegraphics[height=1.5in]{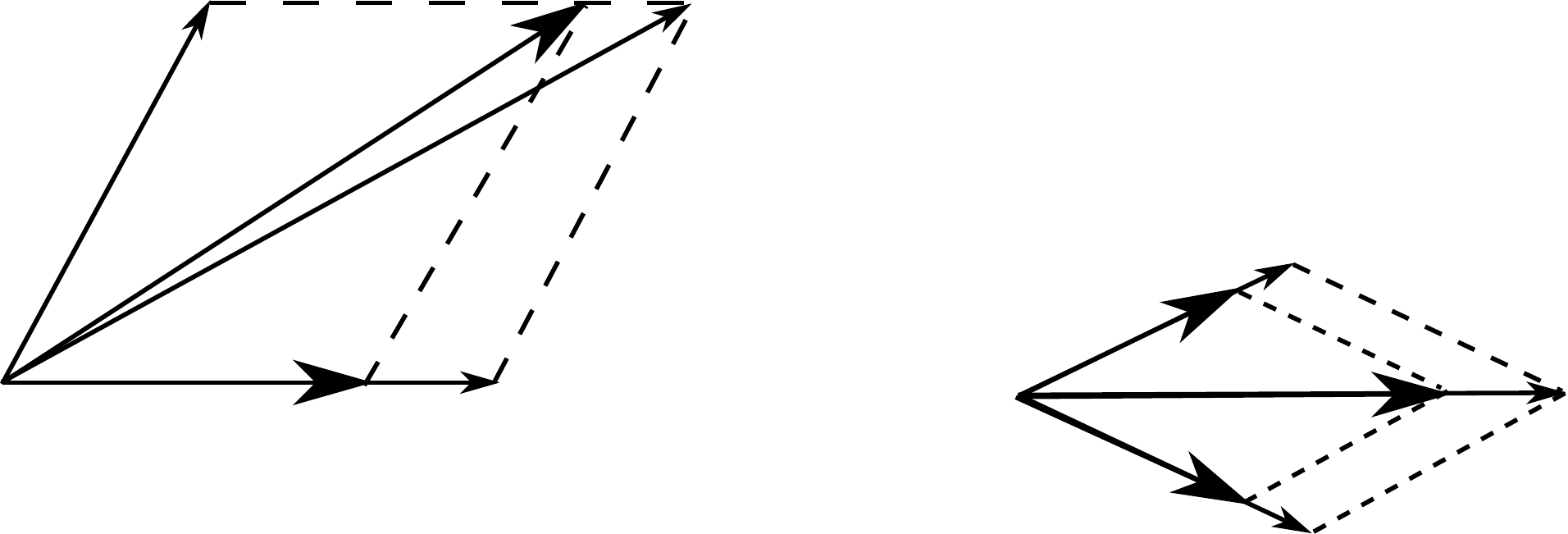}
\caption{Changing the angle of $\widetilde{z}_1  + \ldots +  \widetilde{z}_{2n-3} + \widetilde{z}_{2n-2}$ keeping $|\widetilde{z}_{2n-3}|=|\widetilde{z}_{2n-2}|$.  }
\label{fig:vectors1}
\begin{picture}(0,0)(0,0)
\put(100,27){${z}_{2n-3}'$}
\put(100,95){${z}_{2n-2}'$}

\put(60,37){$\widetilde{z}_{2n-3}$}
\put(60,90){$\widetilde{z}_{2n-2}$}

\put(-190,150){$\widetilde{z}_1  + \ldots + \widetilde{z}_{2n-4}$}
\put(-140,52){$\widetilde{z}_{2n-3}+\widetilde{z}_{2n-2}$}
\put(-67,58){${z}_{2n-3}'+{z}_{2n-2}'$}
\end{picture}
\end{figure}
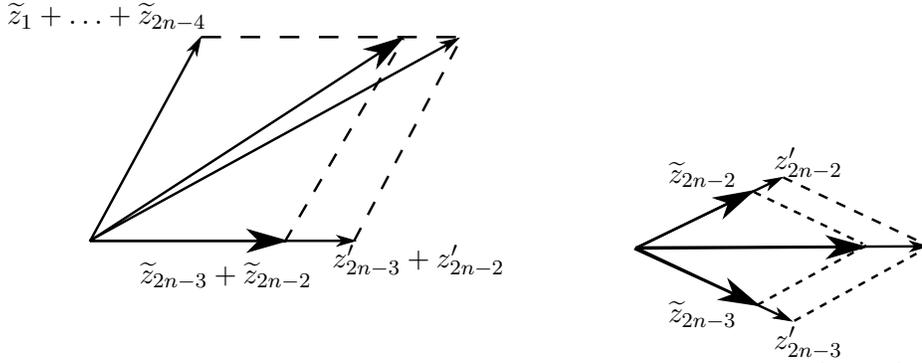

Denote $\widetilde{\tau}= -\widetilde{z}_{1} - \ldots -\widetilde{z}_{2n-2} $. By the triangle inequality  $|\tau- \widetilde{\tau}|\leq\displaystyle \sum_{j=1}^{2n-2}|z_j-\widetilde{z}_{j}|<\epsilon$. Observe that $\widetilde{\tau}\not=0$, because if $\widetilde{\tau}=0$ then $2\epsilon<|\tau|<\epsilon$, which is a contradiction.\\

Since $\tau \not=0$, the point $z_{2n-1}$ is on the perpendicular bisector  $l$ of the vector $\tau$. Let $\widetilde{l}$ be the perpendicular bisector of the vector $\widetilde{\tau}$. Let $\delta=2 $dist$(z_{2n-1}, \widetilde{l})$. Note that as $\epsilon \to 0$, $\delta \to 0$. Now choose $\widetilde{z}_{2n-1} \in \widetilde{l} \cap B(z_{2n-1}, \delta)$ so that the argument of $\widetilde{z}_{2n-1}$ is in $\mathbb{Q} \pi$. See Figure~\ref{fig:vectors2}. Set $\widetilde{z}_{2n}=\widetilde{\tau} - \widetilde{z}_{2n-1}$. Since $\widetilde{z}_{2n-1}$ is on the perpendicular bisector of $\widetilde{\tau}$ then $|\widetilde{z}_{2n-1}|= |\widetilde{z}_{2n}|$ and since $Arg(\widetilde{\tau}) = \frac{1}{2} (\widetilde{\theta}_{2n-1} + \widetilde{\theta}_{2n})$, it follows that $\widetilde{\theta}_{2n} \in \mathbb{Q} \pi$.\\

\begin{figure}[htb]
\centering
\includegraphics[height=1.5in]{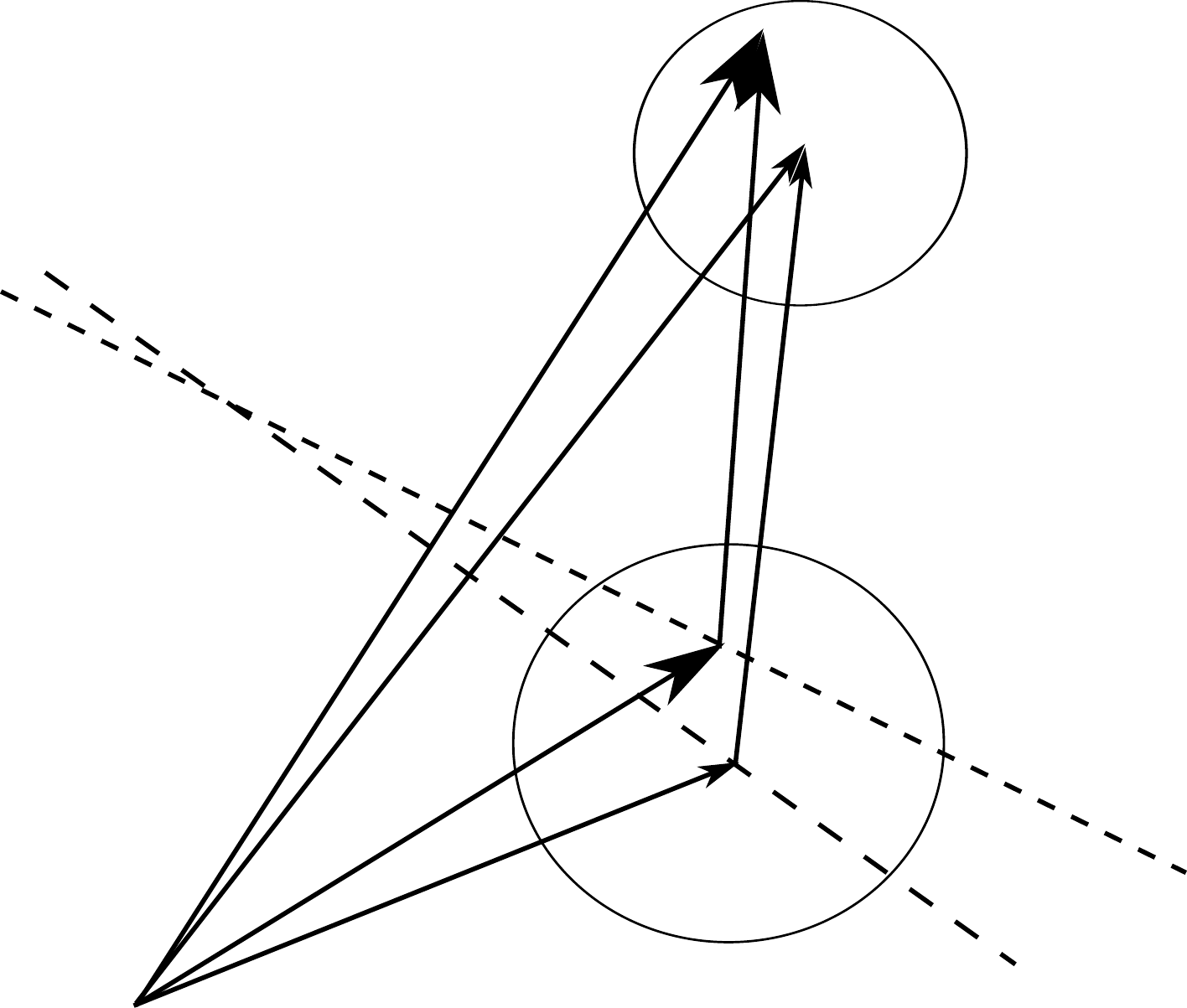}
\caption{Constructing the vector $\widetilde{z}_{2n-1}$. }
\label{fig:vectors2}
\begin{picture} (0,0) (0,0)
 	 \put(25,127){$\tau$}
     \put(12,145){$\widetilde{\tau}$}
      \put(12,54){$z_{2n-1}$}
       \put(15,81){$\widetilde{z}_{2n-1}$}
       
       \put(41, 135){$B(\tau,\epsilon)$}
          \put(43, 65){$B(z_{2n-1},\delta)$}
 \end{picture}  
\end{figure}
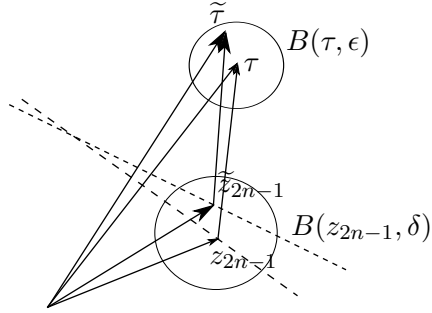

We have $\widetilde{z}_1  + \ldots +  \widetilde{z}_{2n} = 0$ and $\widetilde{\theta}_i \in \mathbb{Q} \pi$ for every $i$. Therefore for $\epsilon >0$ sufficiently small $\{\widetilde{z}_i\}_{i=1}^{2n}$ determines an isometrically immersed polygon $P_{\epsilon}$ with all sides meeting the real axis at angles that are rational multiples of $\pi$.\\

For sufficiently small $\epsilon$ there is a bilipschitz homeomorphism $f_{\epsilon}:P\rightarrow P_{\epsilon}$, which is linear on edges. For example, when $\epsilon>0$ is sufficiently small, we can take a triangulation of $P_{\epsilon}$ with the same combinatorics as the one used to construct $P$. Map each triangle of $P$ linearly to the corresponding triangle of $P_{\epsilon}$. That way we get a map $f_{\epsilon}$ which is linear, and thus bilipschitz, on each triangle. It follows that $f_{\epsilon}$ is bilipschitz on $P$ and as $\epsilon \rightarrow 0$ bilipschitz constant $K(f_{\epsilon})\rightarrow 1$. The surface $S$ is obtained by gluing the polygon $P$, so $S=P/_{\sim}$ where $t \sim \varphi(t)$ and $\varphi:\cup s_i \rightarrow \cup s_i$ is the gluing map. Define $\varphi_{\epsilon}=f_{\epsilon}\varphi f_{\epsilon}^{-1}:\cup f_{\epsilon}(s_i) \rightarrow \cup f_{\epsilon}(s_i)$. The map $f_{\epsilon}$ descends to $\bar{f_{\epsilon}}:P/ _\sim \rightarrow P_{\epsilon}/_{\sim _{\epsilon}}$ where $t\sim _{\epsilon}\varphi_{\epsilon}(t)$. Define a Euclidean cone metric $m_{\epsilon}$ on $S$ by pulling back the metric on $P_{\epsilon}/_{\sim _{\epsilon}}$ by $\bar{f_{\epsilon}}$.  The holonomy of $m_\epsilon$ is generated by rotations through angles in $\mathbb{Q} \pi$. Therefore there is $q_\epsilon \in \mathbb{Z}_+$ so that $m_\epsilon \in Flat(S, q_\epsilon)$ and $m_\epsilon \rightarrow m$ as $\epsilon \rightarrow 0$. 
\end{proof}

Theorem~\ref{thm:infinity} now follows easily from the previous theorem.

\begin{proof}[Proof of Theorem~\ref{thm:infinity}:] Suppose that we are given two distinct curves $\gamma, \gamma' \in \mathcal{C}(S)$ so that $\gamma \equiv_\infty \gamma'$. By Theorem~\ref{thm:polygon}, for every metric $m$ in $Flat(S)$ there is a sequence of metrics $\{m_n\in Flat(S, q_n)\}_{n=1}^\infty$ such that $id:(S, m_n) \rightarrow (S, m)$ is $K_n$-bilipschitz and $K_n \rightarrow 1$ as $n\rightarrow \infty$. Thus $\displaystyle \lim_{n \to \infty}l_{m_n}(\gamma) = l_m(\gamma)$ and $\displaystyle \lim_{n \to \infty}l_{m_n}(\gamma') = l_m(\gamma')$. Since $l_{m_n}(\gamma)=l_{m_n}(\gamma')$ by assumption then $l_{m}(\gamma)=l_{m}(\gamma')$. That implies $\gamma \equiv_\mathbb{R} \gamma'$ which is a contradiction by Theorem~\ref{thm:R}. 
\end{proof}

 Using the same idea as in  Theorem~\ref{thm:polygon} we can prove a stronger statement.

\begin{thm} \label{thm:polygon2} For every infinite sequence of distinct positive integers $\displaystyle \{q_i\}_{i=1}^\infty$,  $$\displaystyle\overline{\bigcup _{i=1}^\infty Flat(S,q_i)}=Flat(S).$$ \end{thm}

\begin{proof}[Proof:] The proof of this theorem follows the proof of Theorem~\ref{thm:polygon} if we approximate angles $\theta_j$ with angles $\widetilde{\theta}_j\in \mathbb{Z}\frac{2\pi}{q_l}$ for appropriately chosen $q_l\gg0$. 
\end{proof}

\begin{thm}\label{cor:infty2} Let $\displaystyle \{q_i\}_{i=1}^\infty$ be an infinite sequence of distinct positive integers. If $\gamma \equiv_{q_i} \gamma'$ for every $i=1,\, 2, \ldots,$ then $\gamma=\gamma'$ in $\mathcal{C}(S)$.\end{thm}

\begin{proof}[Proof:] This follows from Theorem~\ref{thm:polygon2} in the same arguments that prove Theorem~\ref{thm:infinity} from Theorem~\ref{thm:polygon}.\end{proof}

From Theorem~\ref{thm:clein2} we know that $h$-equivalence implies $2$-equivalence (and thus implies $1$-equivalence). However from Corollary~\ref{cor:infty2} we see that  $h$-equivalence implies $q$-equivalence for only finitely many $q\in \mathbb{Z}_+$. More precisely, we have the following corollary.

\begin{cor} Let $\gamma,\gamma' \in \mathcal{C}(S)$ be distinct curves with $\gamma \equiv_h \gamma'$.  Then for all but finitely many $q$, $\gamma \not \equiv_q \gamma'$. \end{cor}

\begin{proof}[Proof:] Assume that there exist infinitely many $q_i\in \mathbb{Z}_+$ so that $\gamma \equiv_h \gamma' \, \, \Rightarrow \, \, \gamma \equiv_{q_i} \gamma'$. Then, there exist two distinct curves $\gamma$ and $\gamma'$ so that $\gamma \equiv_{q_i} \gamma'$ for some infinite sequence $\{q_i\}_{i=1}^\infty$ of positive integers,  which is a contradiction to Corollary~\ref{cor:infty2}. 
\end{proof}


\section {Curves in $q$-differential metrics and relations $\equiv_q$}
\label{sec:q}

In this section we prove:

\begin{thm} \label{thm:thmq} For every $q_0,k\in\mathbb{Z}_+$ there are $k$ distinct homotopy classes of curves $\gamma_1,\ldots,\gamma_k \in \mathcal{C}(S)$ such that $\gamma_i\equiv_q\gamma_j$, for all $i,j$ and for every $q\leq q_0$. Thus for every $q \in \mathbb{Z}_+$, the relation  $\equiv_q$ is non-trivial. \end{thm}

The idea of the proof is the following.  For every $q$, we construct a $2$--complex $\Gamma$, where $\pi_1(\Gamma)$ is a rank $2$ free group, and a homotopy class of maps $\Gamma \to S$.  Then for every metric $m\in Flat(S,q)$, $q\leq q_0$, we will define a metric on the $2$-complex $\Gamma$ and a map $\Gamma \to S$ in the given homotopy class, so that with respect to the metric on $\Gamma$ and $m$ on $S$, the map is locally convex and locally isometric.  In particular, the length of a homotopy class of curves in $\Gamma$ is equal to the length of the image homotopy class in $S$.  The metrics that occur on $\Gamma$ are very restrictive, and it is easy to construct a set of homotopy classes of curves $w_0,w_1,\ldots,w_{k-1}$ in $\Gamma$ with equal lengths in any metric assigned to $\Gamma$ from the construction.  The image homotopy classes in $S$ then also have equal length for any metric $m \in Flat(S,q)$, and by a homological argument, we prove that the homotopy classes are all distinct.\\

\begin{lem}\label{lem:lem1} For any $q_0\in \mathbb{Z}_+$, there is a curve $\gamma \in \mathcal{C}(S)$ such that for every  $m\in Flat(S, q)$, $q\leq q_0$ the geodesic representative $\gamma_m$ of $\gamma$ contains a cone point and $[\gamma] \not= 0$ in $H_1(S,\mathbb{Z})$.\end{lem}

\begin{proof}[Proof:]  Fix a hyperbolic metric on $S$ and identify the universal cover as the hyperbolic plane $\mathbb{H}^2 \rightarrow S$. We will use the Poincar\'e disk model for $\mathbb{H}^2$ throughout the proof.  Fix $q_0 \in \mathbb{Z}_+$, and take $q_0+1$ bi-infinite geodesics $\alpha_1,\ldots , \alpha_{q_0+1}$ in $\mathbb{H}^2$ meeting at $0$. See Figure~\ref{fig:krug}. \\

\begin{figure}[htb]
\centering
\includegraphics[height=3in]{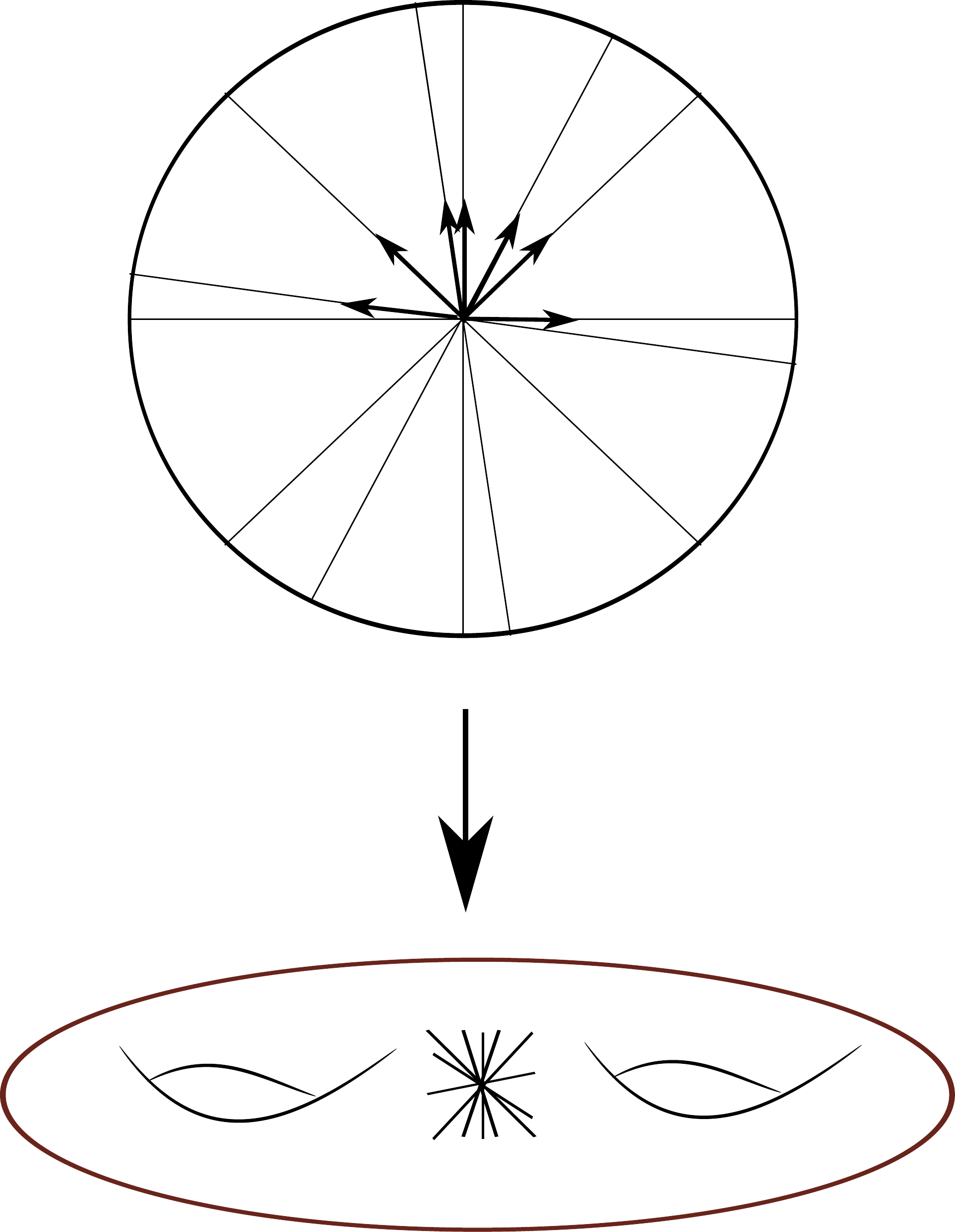}
\caption{Projecting  $\alpha_1,\ldots , \alpha_{q_0+1}$ into the surface $S$.}
\label{fig:krug}
\end{figure}

By ergodicity of the geodesic flow, there is a dense geodesic on $S$ \cite{hopf}.  Given $\epsilon > 0$, we can therefore construct an
$\epsilon$-dense closed geodesic.\\

\begin{figure}[htb]
\centering
\includegraphics[height=1.5in]{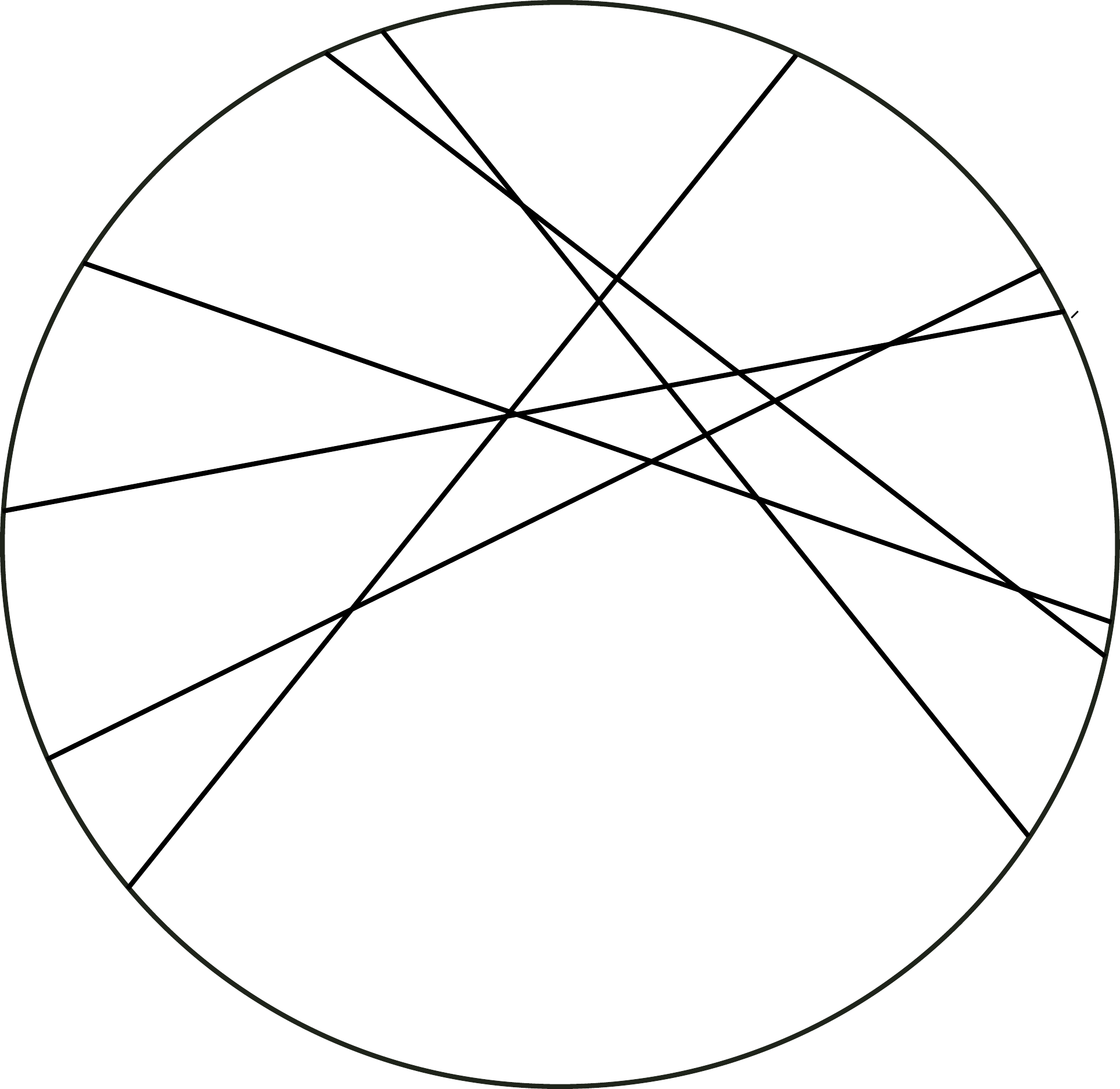}
\caption{Lifts of $\gamma_m$ to $\mathbb{H}^2$. }
\label{fig:krug2}
\end{figure}

From this it follows that we can construct a closed geodesic $\gamma$ on $S$ with lifts $\widetilde{\alpha}_1,\ldots,\widetilde{\alpha}_{q_0+1}$ having tangent vectors within $\epsilon$ distance from the tangent vectors of $\alpha_1,\ldots,\alpha_{q_0+1}$ at $0$. For $\epsilon$ sufficiently small the end points of  $\widetilde{\alpha}_1,\ldots,\widetilde{\alpha}_{q_0+1}$ pairwise link in the same pattern as $\alpha_1,\ldots,\alpha_{q_0+1}$. We say that two bi-infinite geodesics in $\mathbb{H}^2$ link each other if the endpoints of one geodesic separate the circle at infinity $S^1_\infty$ of $\mathbb{H}^2$ in two connected components so that the two endpoints of the second geodesic belong to different components. In particular, every pair $\widetilde{\alpha}_i, \widetilde{\alpha}_j$ intersects. For every $m \in Flat(S,q)$ and  $m$-geodesic representative $\gamma_m$ of $\gamma$ in $m$ there are lifts $(\widetilde{\gamma}_1)_m, \ldots, (\widetilde{\gamma}_{q_0+1})_m$ which are quasi-geodesics in $\mathbb{H}^2$, with the same endpoints as  $\widetilde{\alpha}_1,\ldots,\widetilde{\alpha}_{q_0+1}$, and hence with endpoints that all pairwise link. See Figure~\ref{fig:krug2}.\\

If $[\gamma] = 0$ in $H_1(S,\mathbb{Z})$, we replace $\gamma$ with a different curve as follows.  Let $\delta$ be any curve with $[\delta] \neq 0$ in $H_1(S,\mathbb{Z})$ that intersects $\gamma$.  Construct a new curve which runs $n$ times around $\gamma$, then at the intersection point switches and runs once around $\delta$. If $n$ is large enough, this constructs a curve which maintains the property of having $q_0+1$ lifts with endpoints that pairwise link.  This new curve is homologous to $\delta$, and so we replace $\gamma$ with this curve. \\

Now assume that there is $m\in Flat(S,q)$, $q \leq q_0$ so that $\gamma_m$ does not go through a cone point. Then there is an isometrically immersed Euclidean cylinder  $S^1 \times [0,a] \to S$, for some $a>0$, such that the image of $S^1 \times \{\frac{a}{2}\}$ is the geodesic $\gamma_m$ and the image of $S^1 \times [0,a]$ does not contain any cone points. It follows that $\gamma_m$ has only finitely many transverse self intersecting points. At any point $P$ of self intersection of $\gamma_m$, one can form the loops based at $P$ by following one of the arcs of $\gamma_m$ until returning to $P$.  The holonomy around this loop is just the rotation by the angle of self intersection of $\gamma_m$ at $P$. Thus, any tangent vector is rotated by the angle of self intersection. Because the holonomy group of $m$ is in $\langle \rho_{\frac{2\pi}{q}} \rangle$, every tangent vector gets rotated by some integer multiple of $\frac{2\pi}{q}$. Therefore, the angles of intersections are integer multiples of $\frac{2\pi}{q}$.\\

Lifts of $\gamma_m$ have neighborhoods that are lifted cylinders.  These are strips isometric to $\mathbb{R} \times [0,a]$ in the universal cover. We consider the $q_0+1$ lifts $(\widetilde{\gamma}_1)_m, \ldots, (\widetilde{\gamma}_{q_0+1})_m $ as above and fix one of the lifts $(\widetilde{\gamma}_1)_m$ and a strip about it. Let $P \in (\widetilde{\gamma}_1)_m \cap (\widetilde{\gamma}_2)_m $ be the point  of intersection of $(\widetilde{\gamma}_2)_m$ with $(\widetilde{\gamma}_1)_m$.  For every $j \geq 3$, parallel transport the tangent vector to $\widetilde{\gamma}_{j}$ at the point of intersection $(\widetilde{\gamma}_{1})_m \cap (\widetilde{\gamma}_{j})_m$ inside the strip to $P$. In the complement of the cone points in the universal cover, the holonomy is also in $\langle \rho_{\frac{2\pi}{q}} \rangle$. Thus, parallel transport of tangent vectors along $(\widetilde{\gamma}_k)_m$, $k=2,\ldots,q_0+1$, from the point of intersection with another lift of $\gamma_m$ to $P$ rotates them by some integer multiple of $\frac{2\pi}{q}$.  The point $P$ is not a cone point and there are no cone points inside the strip. If  $(\widetilde{\gamma}_{1})_m$, $(\widetilde{\gamma}_{i})_m$, $(\widetilde{\gamma}_{j})_m$ do not intersect in a common point then they form a triangle in $\widetilde{S}$. Since $\widetilde{S}$ is a $CAT(0)$ space the angle sum of a triangle is less than $\pi$, and hence $(\widetilde{\gamma}_i)_m$ and $(\widetilde{\gamma}_j)_m$ cannot make the same angle with $(\widetilde{\gamma}_1)_m$. Hence no two tangent vectors get transported to the same vector. See Figure~\ref{fig:strips}. \\

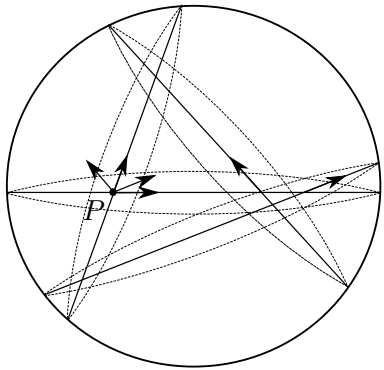
\begin{figure}[htb]
\centering
\includegraphics[height=1.9in]{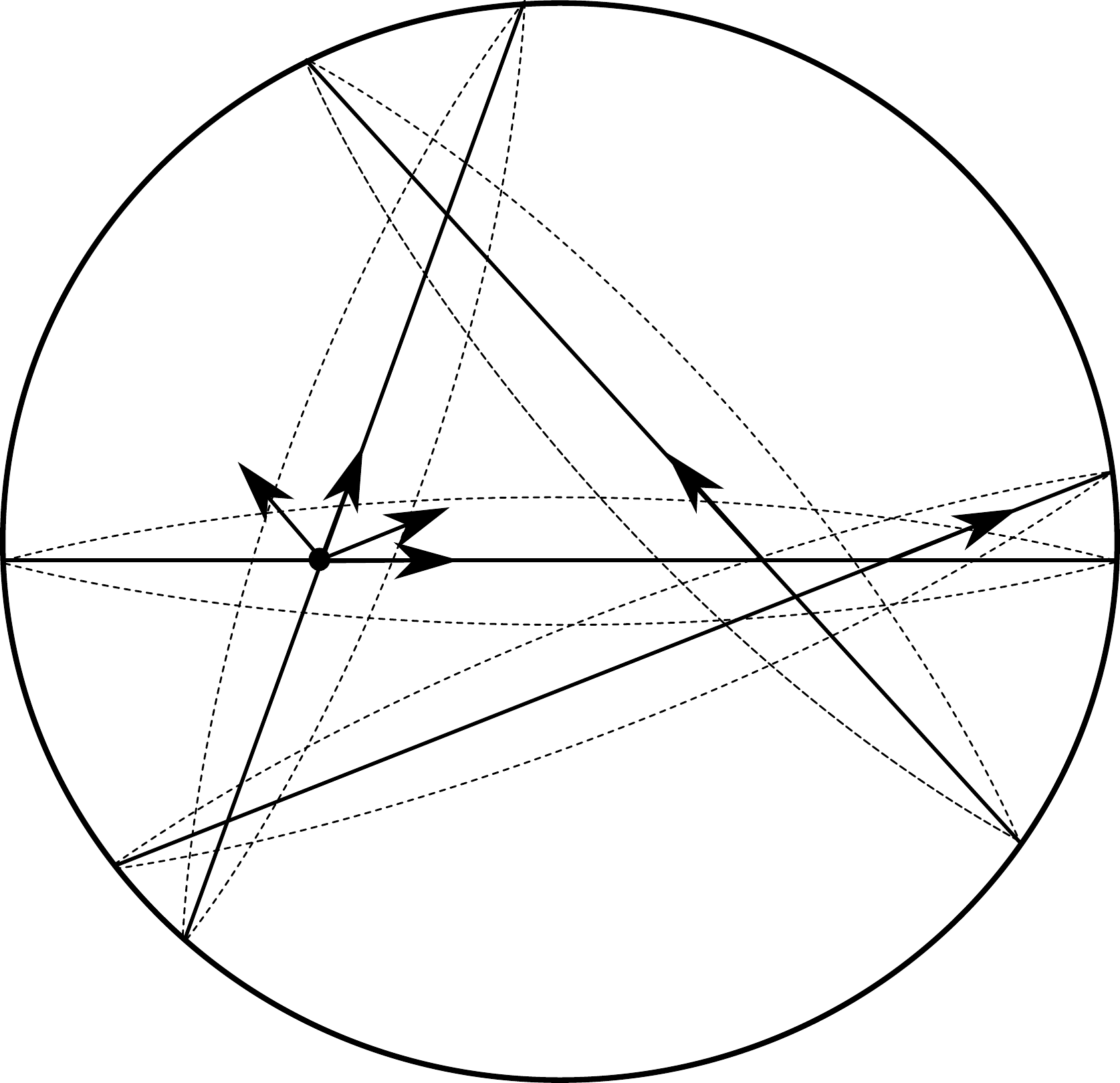}
\caption{Strips in the universal cover.}
\label{fig:strips}
\begin{picture}(0,0)(0,0)
\put(-42,93){$P$}
\end{picture}
\end{figure}

By the holonomy condition all angles of intersections are integer multiples of $\frac{2\pi}{q}$, and hence greater than or equal to $\frac{2\pi}{q_0}$. On the other hand we have $q_0+1$ vectors based at $P$, no two of them equal, therefore there is a pair of vectors with the angle between them less than $\frac{2\pi}{q_0}$. It follows that for some $j$, $(\widetilde{\gamma}_{1})_m$ and $(\widetilde{\gamma}_{j})_m$ make an angle which is not an integer multiple of $\frac{2\pi}{q}$. This is a contradiction and thus proves that $\gamma_m$ has to contain at least one cone point for every $m \in Flat(S, q)$, $q \leq q_0$.
\end{proof}


\begin{lem} \label{lem:lem2} Let $\gamma$ be the curve constructed in Lemma~\ref{lem:lem1}, and let $m$ be a metric in $Flat(S, q)$.  Given two lifts $\widetilde{\gamma}$, $\widetilde{\gamma}_0$ with linking endpoints and  stabilizers generated by $h$ and $g$, respectively, let $w\in \widetilde{\gamma} \cap \widetilde{\gamma}_0$. Then the $m$-geodesic from $h^{(q+2)}(w)\in \widetilde{\gamma}$ to any point along $\widetilde{\gamma}_0$ must run along a positive length segment of $\widetilde{\gamma}$. Moreover, this geodesic meets  $\widetilde{\gamma}_0$ on $[g^{-(q+2)}(w),g^{(q+2)}(w)] \subset\widetilde{\gamma}_0 $.\end{lem} 

\begin{proof}[Proof:]

 It could happen that $l_m(\widetilde{\gamma}_0\cap\widetilde{\gamma})>0$ as in Figure~\ref{fig:krug4}. We first claim that $l_m(\widetilde{\gamma}_0 \cap \widetilde{\gamma})<$ $l_m(\gamma)$. To show this let $v$, $u$ be cone points at each end of $\widetilde{\gamma}_0 \cap \widetilde{\gamma}$. If $l_m(\widetilde{\gamma}_0 \cap \widetilde{\gamma}) \geq $ $l_m(\gamma)$ then since the translation length of $g$ and $h$ is equal to $l_m(\gamma)$, we have $g(v)=h(v)$ or $g^{-1}(v)=h(v)$ which is a contradiction as $\pi_1(S)$ acts freely on $\widetilde {S}$.\\

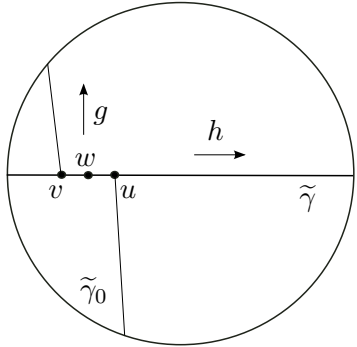
\begin{figure}[htb]
\centering
\includegraphics[height=1.8in]{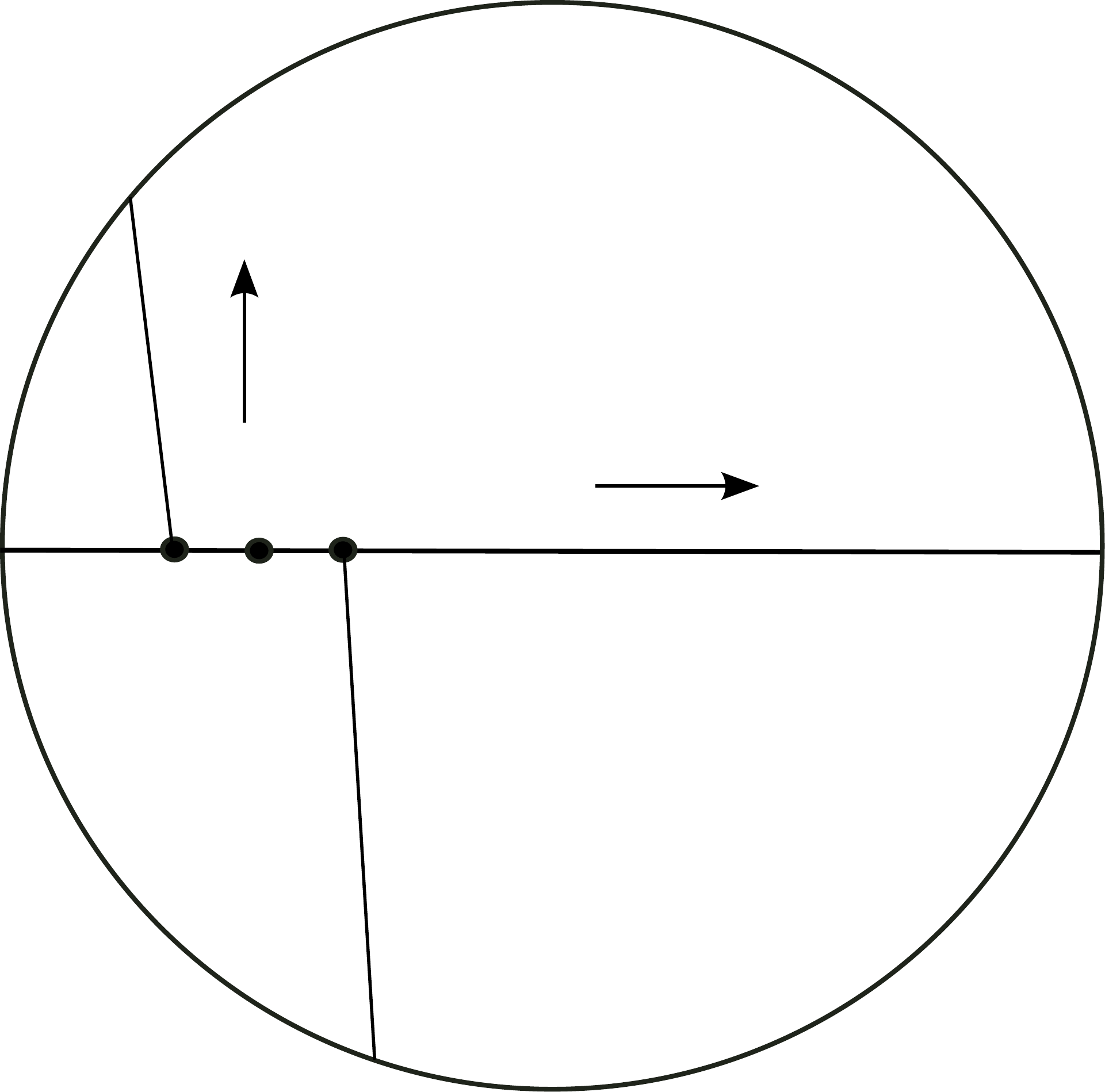}
\caption{Two lifts of $\gamma$ to the universal cover.}
\label{fig:krug4}
\begin{picture}(0,0)(0,0)
 	 \put(-38,55){$\widetilde{\gamma}_0$}
     \put(45,90){$\widetilde{\gamma}$}
      \put(10,115){$h$}
      \put(-33,123){$g$}
      \put(-40,105){$w$}
      \put(-50,93){$v$}
      \put(-23,93){$u$}
     \end{picture}
\end{figure}

Now let $v_1, \ldots , v_n$ be all different cone points along $[h^{(q+1)}w, h^{(q+2)}(w)) \subset \widetilde{\gamma}$   and let $\theta_i$, $i=1,\ldots n$, be the angles $\widetilde{\gamma}$ makes at $v_i$ all on one side of $\widetilde{\gamma}$, and let $\theta_i'$ be all the angles at $v_i$ on the other side. Observe that there is some $i$ so that $\theta_i>\pi$. Otherwise there is a small neighborhood of one side of $\widetilde {\gamma}$ with no cone points and we can find a geodesic homotopic to $\gamma$ with no cone points by doing a straight line homotopy as shown in Figure~\ref{fig:krug5}. Similarly, there is $j$ so that $\theta_j'> \pi$. \\ 

\begin{figure}[htb]
\centering
\includegraphics[height=1.5in]{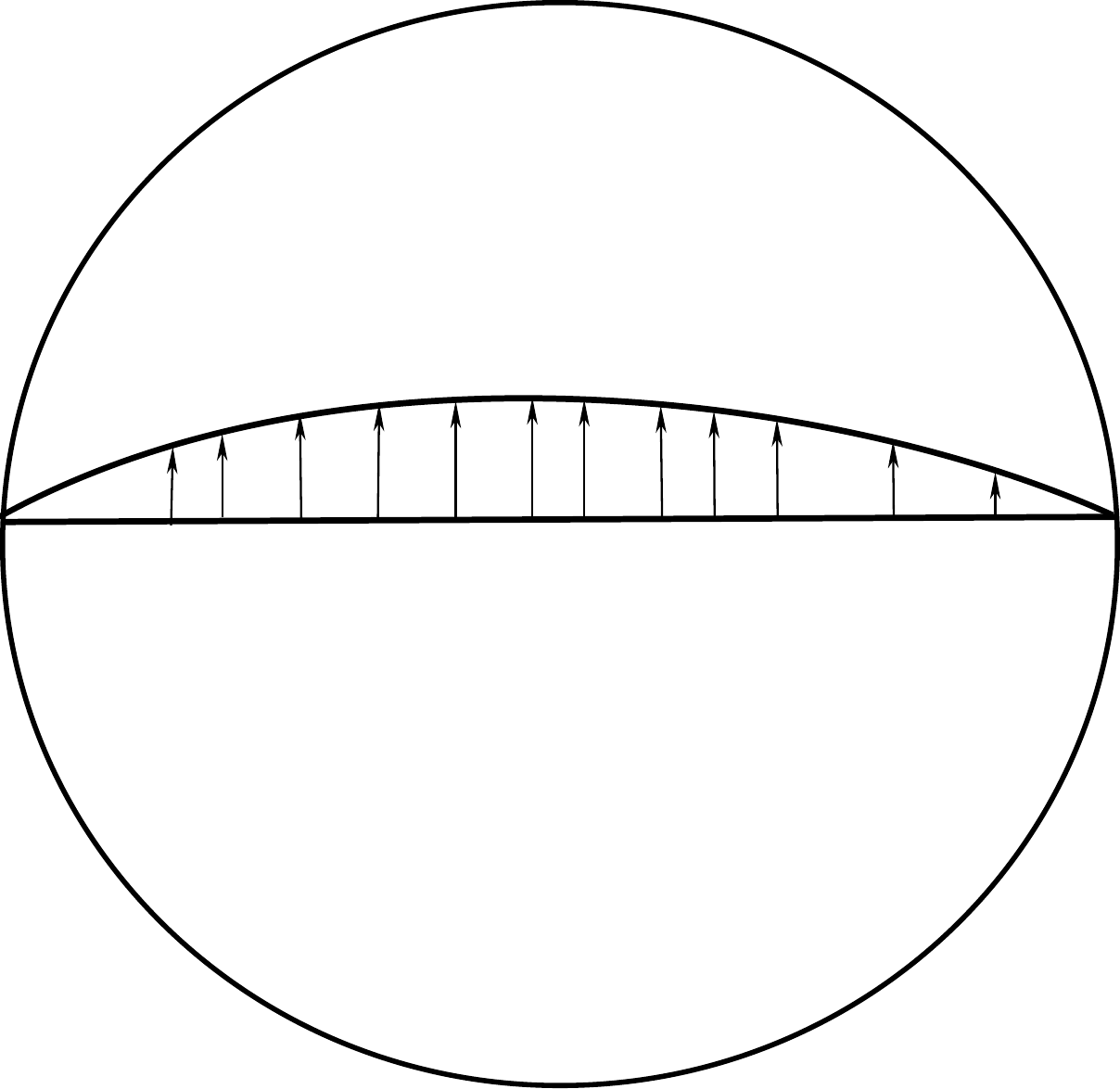}
\caption{Homotopy.}
\label{fig:krug5}
\end{figure}

Let $v=v_i$ with $\theta_i > \pi$ and $v'=v_j$ with $\theta_j '> \pi$. Let $\sigma$ be the continuation of the geodesic segment $[h^{(q+2)}(w),v]\subset \widetilde {\gamma}$ to a geodesic ray starting at $h^{(q+2)}(w)$  and at every cone point past $v$ having angle $\pi$ on the right. See Figure~\ref{fig:krug8}. We define $\sigma'$ similarly as the continuation of $[h^{(q+2)}(w),v']\subset \widetilde {\gamma}$ with angle $\pi$ on the left. \\

\begin{figure}[htb]
\centering
\includegraphics[height=2.5in]{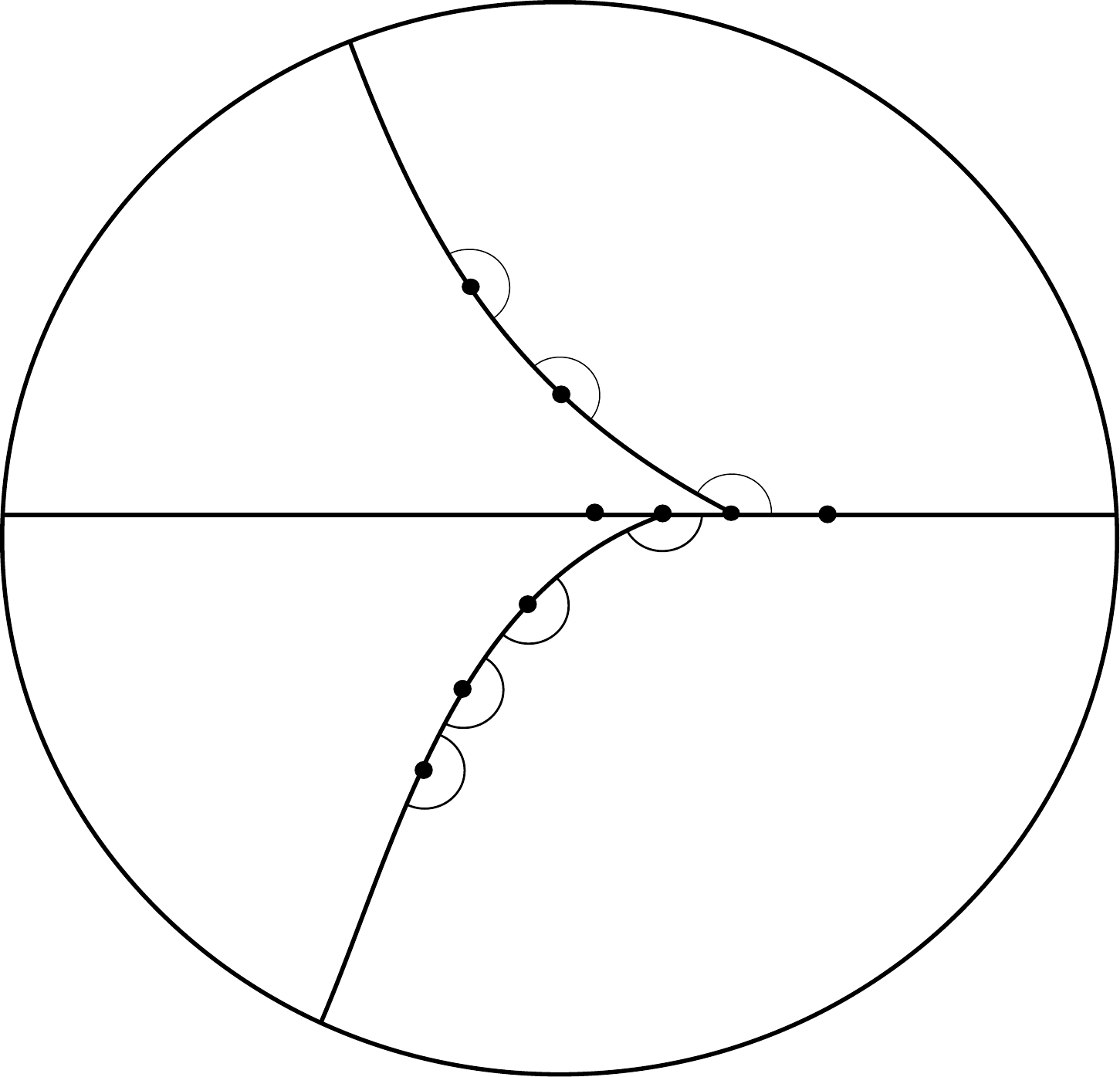}
\caption{Geodesic rays as continuations of the geodesic segment $\gamma$. }
\label{fig:krug8}
\begin{picture}(0,0)(0,0)
 	 
     \put(80,120){$\widetilde{\gamma}$}
     
     \put(25,123){$v$}
     \put(38,135){$h^{(q+2)}(w)$}
     
      \put(16,115){$\pi$}
      \put(2,105){$\pi$}
       \put(-8,95){$\pi$}
      \put(-13,82){$\pi$}
      
        \put(30,140){$\pi$}
       \put(7,155){$\pi$}
      \put(-5,170){$\pi$}
      
      \put(-41,200){$\sigma$}
      \put(-48,53){$\sigma'$}
     \end{picture}
\end{figure}

\emph{Claim:} The rays $\sigma$ and $\sigma'$ do not intersect $\widetilde{\gamma}_0$.\\ 

\emph{Proof:} Assume that $\sigma \cap \ \widetilde{\gamma}_0 \not= \emptyset$. Then $\widetilde{\gamma}$, $\widetilde{\gamma_0}$ and $\sigma$ form a triangle in $\widetilde{S}$. Denote the angles of the triangle by $\beta_1$, $\beta_2$ and $\beta_3$. See Figure~\ref{fig:krug9}.\\

\begin{figure}[htb]
\centering
\includegraphics[height=2.1in]{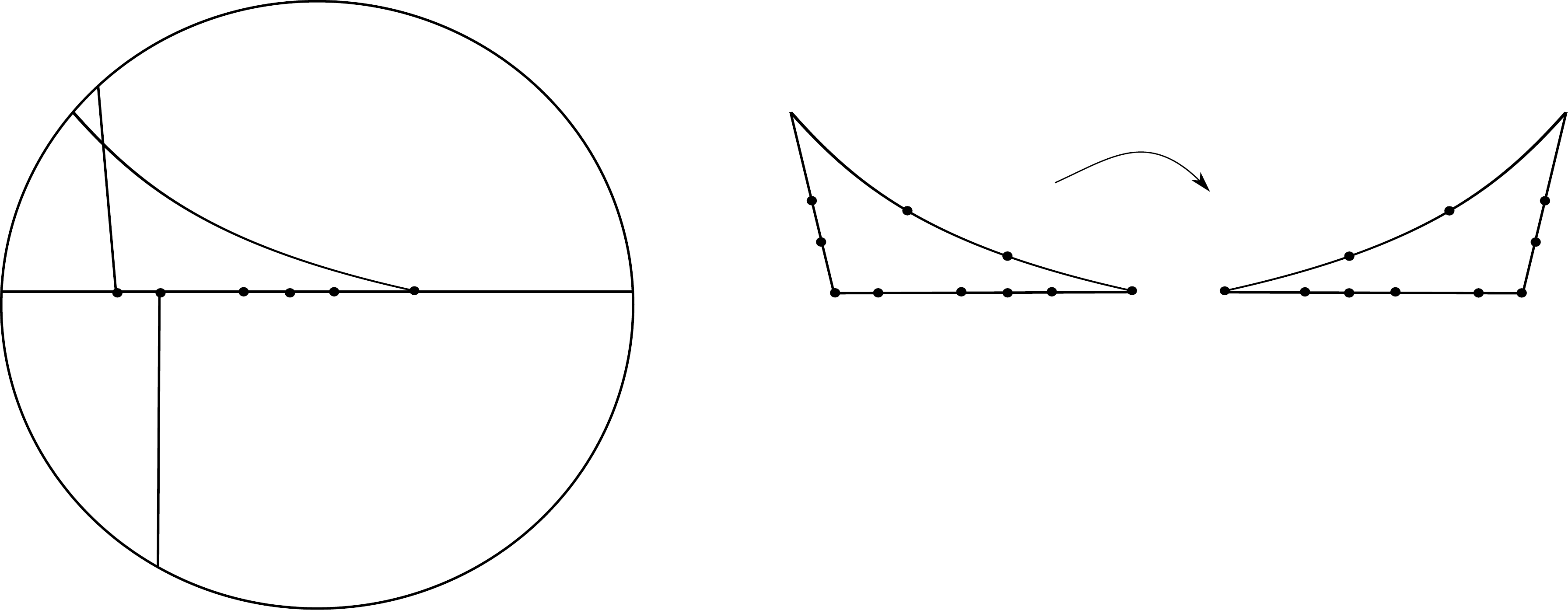}
\caption{Gluing two triangles to get a sphere.}
\label{fig:krug9}
\begin{picture}(0,0)(0,0)
 	 
     \put(-127,117){$\beta_1$}
      \put(-165,118){$\beta_2$}
      \put(-168,138){$\beta_3$}
      
        \put(45,118){$\beta_1$}
      \put(13,118){$\beta_2$}
      \put(7,140){$\beta_3$}
      
       \put(137,117){$\beta_1$}
      \put(172,118){$\beta_2$}
      \put(178,139){$\beta_3$}
     \end{picture}
\end{figure}

 Take two copies of the given triangle and glue them together to get a sphere $R$ with an induced Euclidean cone metric. See Figure~\ref{fig:krug9}. The Euler characteristic of $R$ is 2, and by the Gauss-Bonnet formula (see Proposition~\ref{prop:Gauss}) $$-2\pi \chi (R) = \displaystyle \sum_{x \in X} (c(x)-2\pi),$$ where $X$ is the set of all cone points of $R$. Three cone points come from the triangle vertices and they have cone angles $2\beta_l$, $l=1,2,3$. The rest come from cone points along the sides or the inside of the triangles and have cone angles $>2\pi$. In particular we have cone points $h^{-1}(v_j), \ldots, h^{-q}(v_j)$ with cone angles $2\theta_j$, $j=1,\ldots, n$. Therefore we get the following inequality:  
$$ -4\pi= \displaystyle \sum _{x\in X}(c(x)-2\pi) \geq \displaystyle \sum _{i=1}^q \sum_{j=1}^n(2\theta_j-2\pi) + (2(\beta_1 + \beta_2 + \beta_3)-6\pi).$$ It follows that
$$\pi-(\beta_1 + \beta_2 + \beta_3)\geq  \displaystyle \sum _{i=1}^q \sum_{j=1}^n (\theta_j-\pi).$$ 
The holonomy for every metric in $Flat (S, q)$ is a rotation through an angle of the form $\frac{2\pi}{q} k$, $k\in \mathbb{Z}$. Therefore we get $\displaystyle\sum_{j=1}^n(\theta_j-\pi)=\frac{2\pi}{q} k$ for some integer $k>0$, and so $\displaystyle\sum_{j=1}^n(\theta_j-\pi) \geq \frac{2\pi}{q}$.\\

It follows that $\displaystyle\sum _{i=1}^q\sum_{j=1}^n(\theta_j-\pi)= q \sum_{j=1}^n(\theta_j-\pi)\geq q \frac{2\pi}{q} =2\pi$, and thus: $$\pi-(\beta_1 + \beta_2 + \beta_3)\geq 2\pi$$
which is a contradiction. Therefore $\sigma \cap \ \widetilde{\gamma}_0 = \emptyset$. The same argument shows $\sigma' \cap \ \widetilde{\gamma}_0 = \emptyset$. This proves the claim.\\

Now we will show that every geodesic from $h^{(q+2)}(w)$ to $\widetilde{\gamma}_0$ has to share a positive length geodesic segment with  $\widetilde{\gamma}$.\\

There is a unique geodesic between any 2 distinct points in $\widetilde{S}$, since it is a $CAT(0)$ space. Let $\zeta$ be a geodesic from $h^{(q+2)}(w)$ to some point on $\widetilde{\gamma}_0$. Let $\sigma_0=\sigma-\widetilde{\gamma}$ and $\sigma'_0=\sigma'-\widetilde{\gamma}$. Assume $\zeta$ does not share a positive length geodesic segment of $\widetilde{\gamma}$. By the previous claim $\zeta$ must intersect one of $\sigma_0$ and $\sigma'_0$. See Figure~\ref{fig:krug10}. Without loss of generality assume that $A \in \zeta \cap \sigma_0$ is such a point. The path from $h^{(q+2)}(w)$ to  $v$ following the geodesic segments on $\widetilde{\gamma}$  and then from $v$ to $A$ along $\sigma_0$ is a geodesic by construction since all cone angles are $\geq \pi$ on both sides of the path. Since the initial arc of $\zeta$ to $A$ provides a different geodesic, this contradicts the uniqueness of geodesics in $\widetilde{S}$ and we can see that $\zeta$ cannot cross either of $\sigma_0$ or $\sigma'_0$ and hence must share a positive length segment with  $\widetilde{\gamma}$.\\

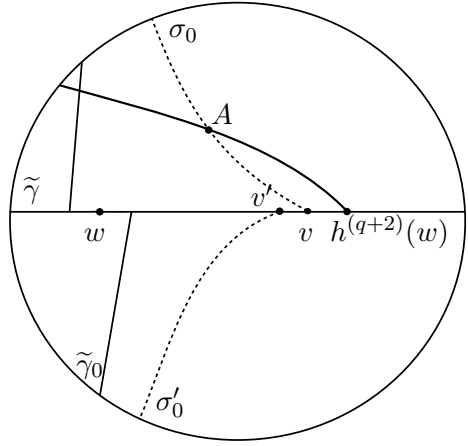
\begin{figure}[htb]
\centering
\includegraphics[height=2.3in]{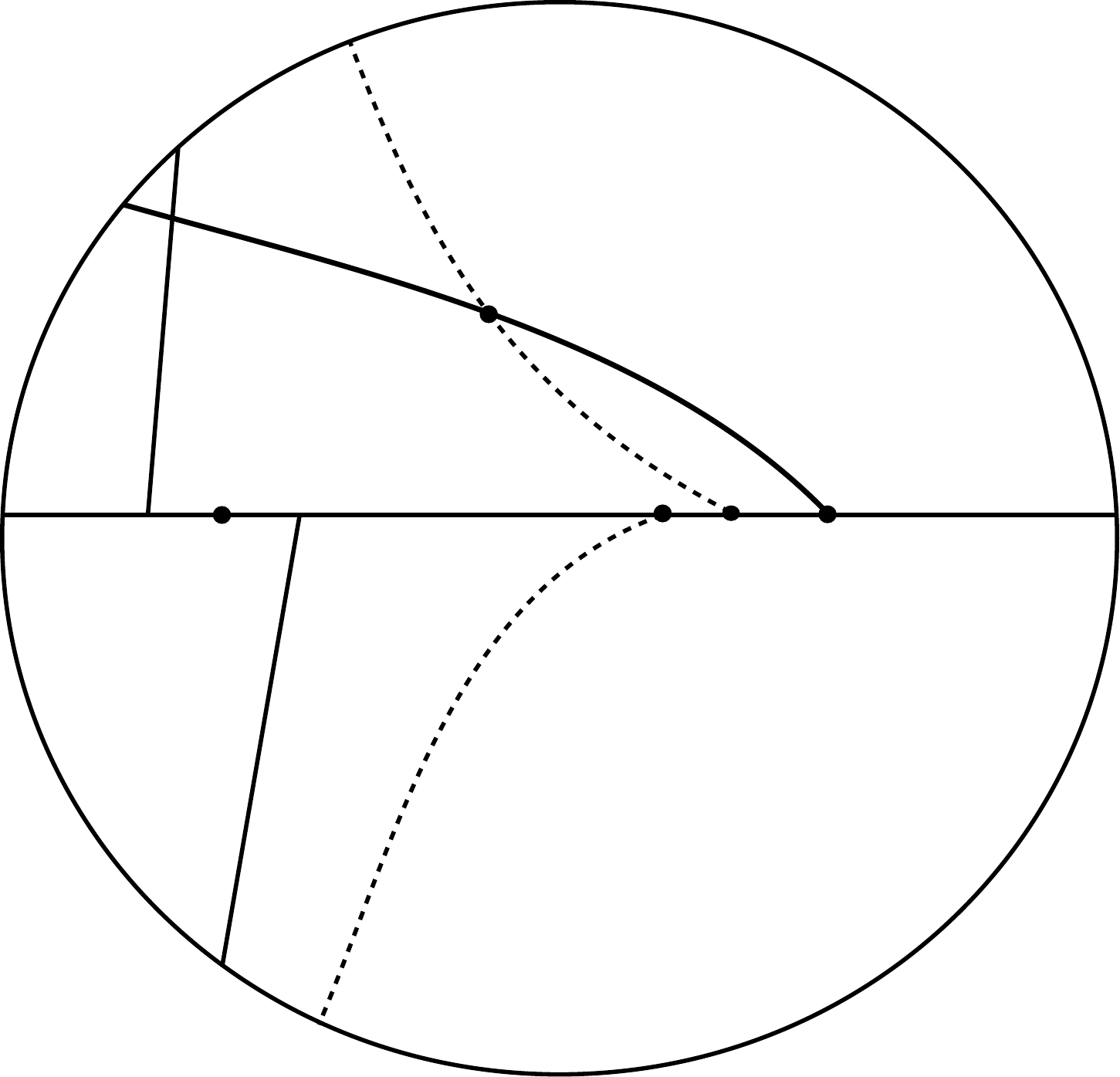}
\caption{A geodesic segment from $\widetilde\gamma$ to $\widetilde \gamma_1$.}
\label{fig:krug10}
\begin{picture}(0,0)(0,0)
 	 \put(-58,113){$w$}
     \put(23,113){$v$}
      \put(35,113){$h^{(q+2)}(w)$}
      \put(-10,157){$A$}
      \put(-81,129){$\widetilde{\gamma}$}
      \put(-61,63){$\widetilde{\gamma}_0$}
      \put(-31,49){$\sigma'_0$}
      \put(-26,190){$\sigma_0$}
      \put(5,126){$v'$}
     \end{picture}
\end{figure}

For the same reason the geodesics from $g^{(q+2)}(w)$ or $g^{-(q+2)}(w)$ to $\widetilde{\gamma}$ must share a positive length segment with  $\widetilde{\gamma}_0$. Therefore every geodesic from $h^{(q+2)}(w)$ to a point on $\widetilde{\gamma}_0$ meets $\widetilde{\gamma}_0$ on $[g^{-(q+2)}(w),g^{(q+2)}(w)]$.
\end{proof}

Now we can prove the main theorem.

\begin{proof}[Proof of Theorem~\ref{thm:thmq}:] 

 Let $\gamma \in \mathcal{C}(S)$ be as in Lemma~\ref{lem:lem1}. Let $\widetilde {\gamma}$ and $\widetilde {\gamma}_0$ be lifts of $\gamma$ with endpoints that link. Suppose the stabilizer of $\widetilde{\gamma}$ is generated by $h$.  Let $\widetilde{\gamma}_{-1}$ and $\widetilde{\gamma}_1$ be defined by $\widetilde{\gamma}_i = h^{i(q+2)}(\widetilde{\gamma_0})$, ${i=\pm 1}$. See Figure~\ref{fig:krug3}. Denote  $h_{-1}$ and $h_1$ conjugate elements of $\pi_1(S)$ that generate the stabilizers of  $\widetilde{\gamma}_{-1}$ and $\widetilde{\gamma}_1$ respectively.\\

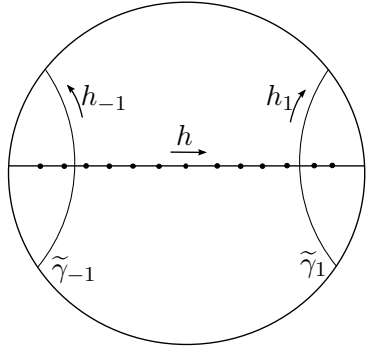
\begin{figure}[htb]
\centering
\includegraphics[height=1.8in]{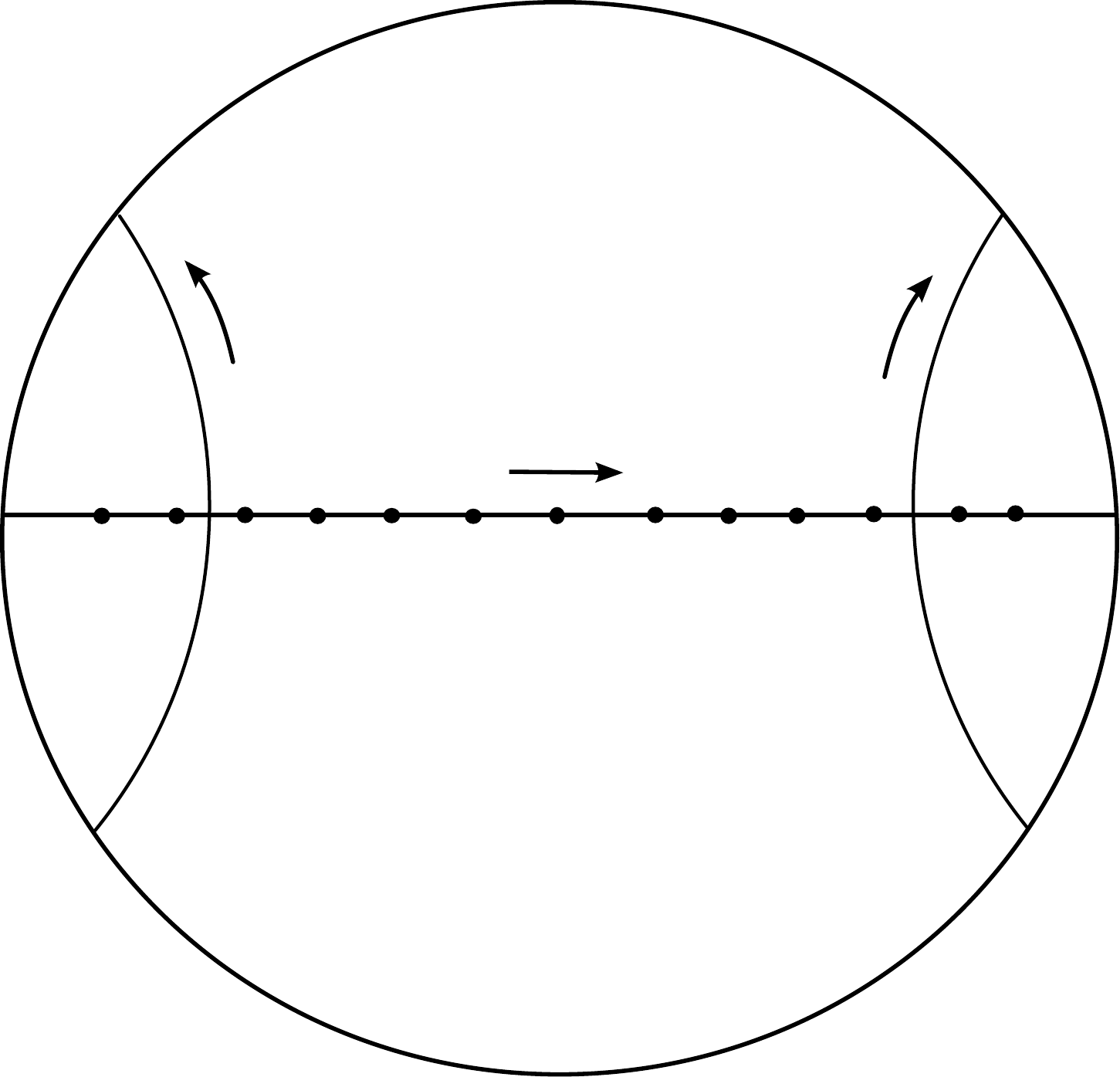}
\caption{Lifts of $\gamma$.}
\label{fig:krug3}
\begin{picture}(0,0)(0,0)
 	 \put(-51,63){$\widetilde{\gamma}_{-1}$}
 	 \put(43,65){$\widetilde{\gamma}_{1}$}
     \put(-4,113){$h$}
     \put(-40,128){$h_{-1}$}
     \put(30,128){$h_1$}
     \end{picture}
\end{figure}

Let $F(a,b)$ denote the free group on two generators $a$ and $b$. Define $\phi: F(a,b)\rightarrow \pi_1(S)$ by $a\rightarrow h_{-1}^{2(q+2)}$ and $b\rightarrow h_{1}^{2(q+2)}$.\\

\noindent \emph{Claim:} The $\phi$--images of the words $$w_0=(ab)^k, \, w_1=(ab)^{k-1}(ab^{-1}), \ldots, w_{k-1}=(ab)(ab^{-1})^{k-1}$$ represent distinct elements in $\mathcal{C}(S)$ and have the same length in every metric $m\in Flat(S,q)$, $q \in \mathbb{Z}_+$.\\

Let $m\in Flat(S,q)$. Let $(\widetilde{\gamma}_i)_m$, $(\widetilde{\gamma})_m$ be corresponding lifts of the $m$-geodesic representatives of $\gamma$. We have $(\widetilde{\gamma}_1)_m=h^{2(q+2)}((\widetilde{\gamma}_{-1})_m)$. Let $z'$ be a point on $(\widetilde{\gamma})_m$ so that $(\widetilde{\gamma}_i)_m \cap (\widetilde{\gamma})_m \in [h^{i(q+1)}(z'), h^{i(q+2)}(z')]$, ${i=\pm 1}$. From Lemma~\ref{lem:lem2} we know that every geodesic from $z'$ to $(\widetilde{\gamma}_i)_m$, $i=\pm 1$, has to run along positive line segment of $(\widetilde{\gamma})_m$ and meet $(\widetilde{\gamma}_i)_m$ on $[h_i^{-(q+2)}(z'), h_i^{(q+2)}(z')]$.\\

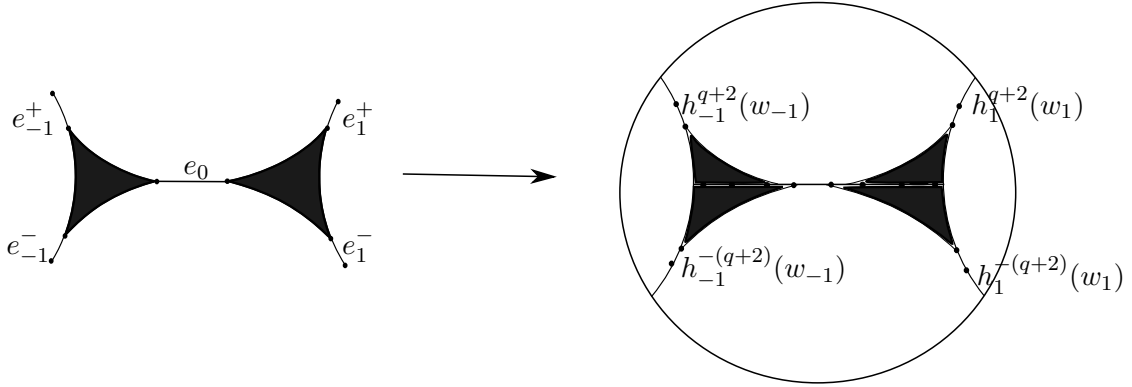
\begin{figure}[htb]
\centering
\includegraphics[height=2in]{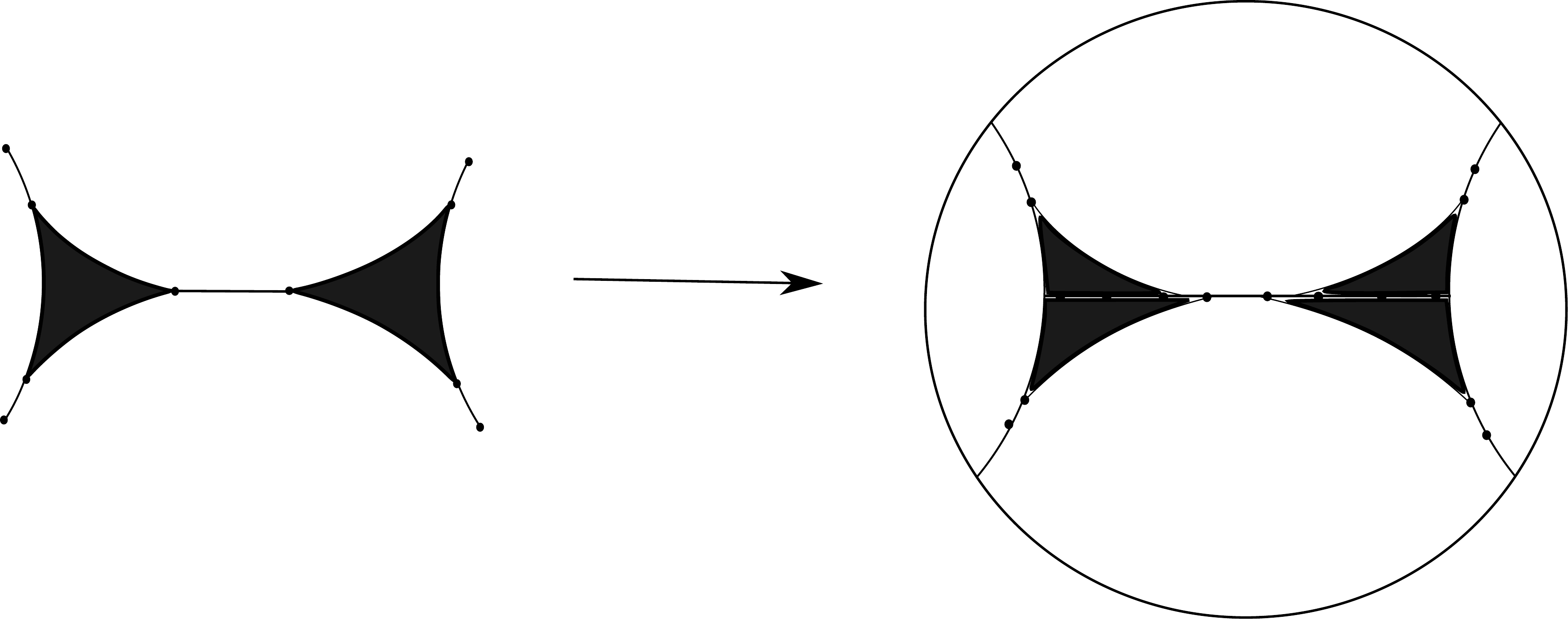}
\caption{The $m$-convex hull of $(\widetilde{\gamma}_{-1})_m$ and $(\widetilde{\gamma}_1)_m$ on the right, and the 2-complex $\Gamma'$ on the left.}
\label{fig:krug12}
\begin{picture} (-72,-13) (-72,-13)
\put(-52,140){$h_{-1}^{q+2}(w_{-1})$}
\put(58,140){$h_1^{q+2}(w_1)$}
\put(-52,78){$h_{-1}^{-(q+2)}(w_{-1})$}
\put(60,75){$h_{1}^{-(q+2)}(w_{1})$}

\put(-305,135){$e_{-1}^+$}
\put(-307,87){$e_{-1}^-$}
\put(-180,87){$e_{1}^-$}
\put(-180,135){$e_{1}^+$}
\put(-240,117){$e_0$}
\end{picture}
\end{figure}

It follows that the $m$--convex hull of $(\widetilde{\gamma}_{-1})_m$ and $(\widetilde{\gamma}_1)_m$ consists of $(\widetilde{\gamma}_{-1})_m\cup(\widetilde{\gamma}_1)_m$ together with an arc of $(\widetilde{\gamma})_m$ and two (possibly degenerate) triangles. Choose a point $w_i\in (\widetilde{\gamma})_m\cap(\widetilde{\gamma}_i)_m$ for $i=\pm1$ and consider the point $h_i^{\pm(q+2)}(w_i)$ along $(\widetilde{\gamma}_i)_m$. Let $\Gamma'$ be the metric $2$-complex with two (possibly degenerate) triangles and five edges determined by $h_i^{\pm(q+2)}(w_i)$ as shown in the Figure~\ref{fig:krug12} and view this as mapping into $\widetilde{S}$. Let $\Gamma$ denote the metric $2$-complex obtained by gluing $h_i^{-(q+2)}(w_i)$ to $h_i^{(q+2)}(w_i)$, for $i=\pm1$. The inclusion $\Gamma' \rightarrow \widetilde{S}$ descends to a locally convex, local isometry $f:\Gamma \to S$.  Identifying $\pi_1(\Gamma) = F(a,b)$ as shown in Figure 26 we have $f_* = \phi$.\\


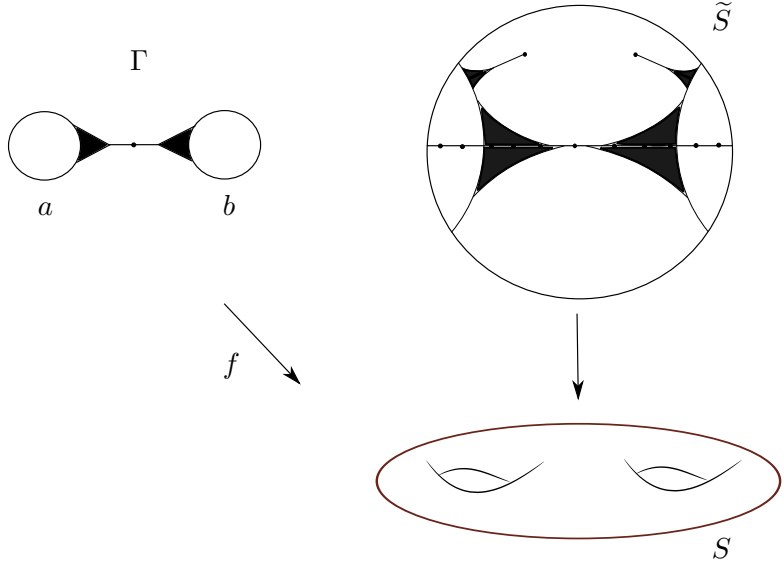
\begin{figure}[htb]
\centering
\includegraphics[height=2.8in]{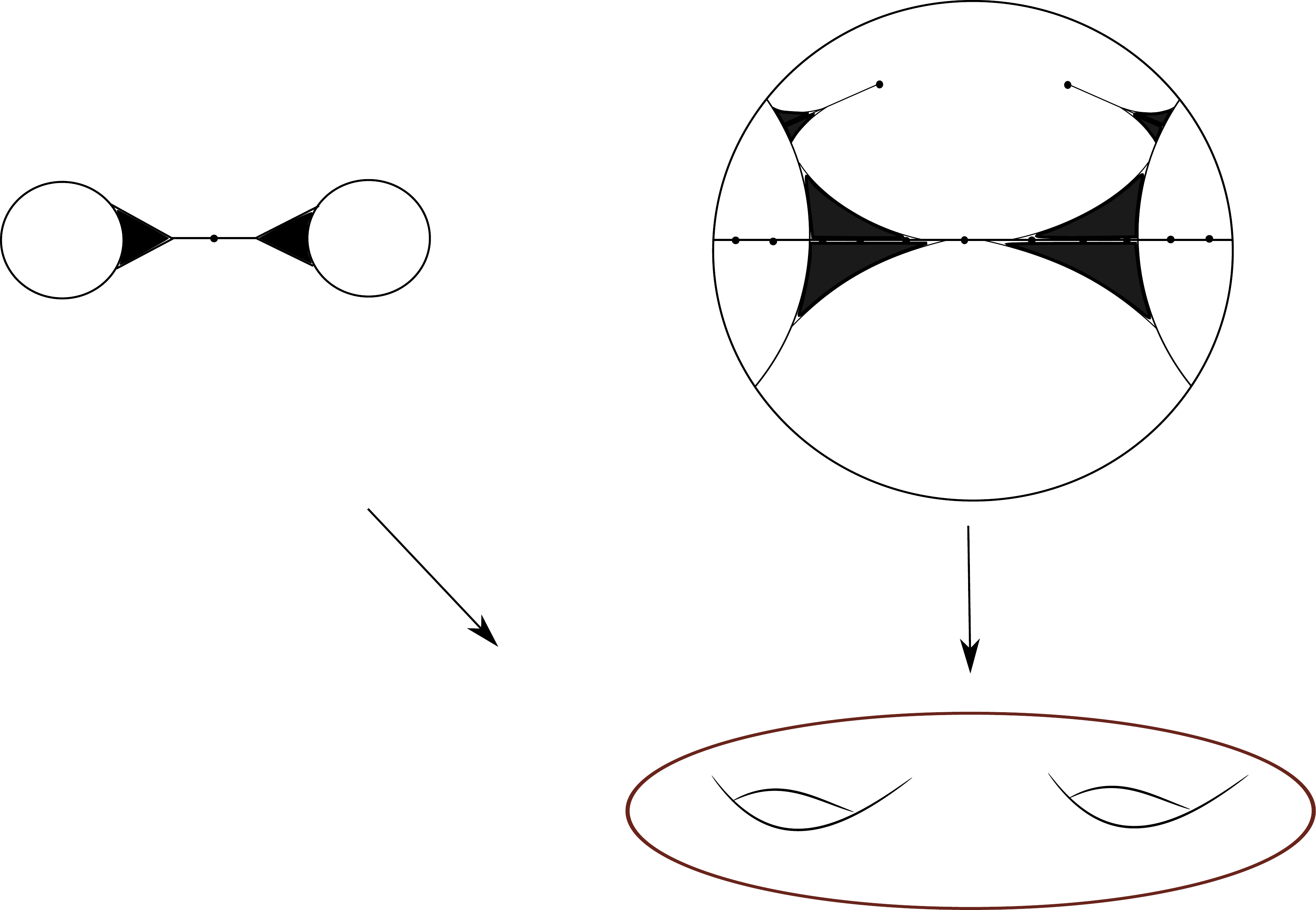}
\caption{Constructing 2-complex $\Gamma$ from the curve $\gamma$.}
\label{fig:krug11}
\begin{picture} (0,0) (0,0)
\put(-135,160){$a$}
\put(-65,160){$b$}
\put(-100,215){$\Gamma$}
\put(120,230){$\widetilde{S}$}
\put(120,30){$S$}
\put(-65,100){$f$}

\end{picture}
\end{figure}

By the previous construction $[f_*(a)]=[f_*(b)]=\pm2(q+2)[\gamma]$ in $H_1(S,\mathbb{Z})$. It follows that $[f_*(w_j)]=\pm2(q+2)(2k-2j)[\gamma]$. By construction $[\gamma]\not= 0$, thus $\{f_*(w_j)\}_{j=0}^{k-1}$ are distinct classes of non-null homotopic curves. Since $f$ is a locally convex, local isometry, by measuring lengths on $\Gamma$ instead of $S$ we can see $l_m(f_*(w_j))=l_m(f_*(w_i))$ for every $i,j=0,\ldots,k-1$.
\end{proof}

\begin{cor} Let $q_1, q_2 \in \mathbb{Z}_+$. If $q_1|q_2$ then $\equiv_{q_2} \, \, \Rightarrow \, \, \equiv_{q_1}$. The reverse implication is not true in general.\end{cor}

\begin{proof}[Proof:] If $q_1|q_2$ then $Flat(S, q_1)\subset Flat(S,q_2)$. Thus the first part of the statement follows. \\
 \indent Assume the reverse statement is true for all $q_1|q_2$. Let $q \in \mathbb{Z}_+$. Then $\equiv_q \, \, \Rightarrow \, \, \equiv_{q^i}$ for every positive integer $i$. By Theorem~\ref{thm:thmq} there are two distinct curves $\gamma, \gamma'\in\mathcal{C}(S)$ so that  $\gamma \equiv_q \gamma'$. By assumption we get $\gamma \equiv_{q^i}\gamma'$, for all $i$. This is in contradiction to Theorem~\ref{thm:polygon2} if we take our infinite sequence to be $\{q^i\}_{i=1}^\infty$. 
 \end{proof}

\begin{thm}\label{thm:qnoth} For every $q \in \mathbb{Z}_+$, there exist $\gamma, \gamma'\in \mathcal{C}(S)$ so that  $\gamma\equiv_q\gamma'$ but  $\gamma\not\equiv_h\gamma' $.\end{thm}

\begin{proof}[Proof:] Let $\gamma$ and $\gamma'$ be from the construction in Theorem~\ref{thm:thmq}. If $\gamma\equiv_h\gamma'$ then $\gamma$ and $\gamma'$ can be oriented so that they represent the same homology class \cite{clein1}.  In our construction $\gamma$ and $\gamma'$ have different homology representatives, thus those curves can not be $h$-equivalent. 
\end{proof}


\section{Punctured surfaces}
\label{sec:punctures}

In this section we will describe the main results for punctured surfaces. Some of the results in this case are actually stronger, and in general the proofs go through with little change.  The main technical difference is in the structure of geodesics, which is more complicated in this setting.\\

Let $\widehat{S}$ denote a closed, oriented surface of genus $g$. Let $S$ be the surface obtained by removing a finite set of points from  $\widehat{S}$.\\

Define $Flat(S)$ to be the set of metrics on $S$ with the following properties:
\begin{enumerate}
\item[(i)] the metric completion of $m\in Flat(S)$ is a Euclidean cone metric $\widehat{m}$ on $\widehat{S}$,

\item[(ii)] cone points of $\widehat{m}$ contained in $S \subset \widehat{S}$ have cone angles $\geq 2\pi$.
\end{enumerate}

For $q\in\mathbb{Z}_+$, define $Flat(S,q)$ to be the set of metrics  $m\in Flat(S)$, such that the holonomy around every loop in $S\setminus\{$cone points$\}$ is in $\langle \rho_{\frac{2\pi}{q}}\rangle$. Cone angles of the cone points on $\widehat{S}$ are therefore of the form $k\frac{2\pi}{q}, \, k>0$. Cone angles of points in $S$ have $k\geq q$, but this is not required for points in $\widehat{S}\setminus S$. One can show that metrics in $Flat(S,q)$ are precisely those that come from meromorphic $q$-differentials on $\widehat{S}$, all of whose poles (if any) are contained in $\widehat{S}\setminus S$ and have order at most $q-1$. For $k < q$ a cone point with cone angle $k\frac{2\pi}{q}$ is a pole of order $q-k$.\\

Let $\mathcal{C}(S)$ be a set of all homotopy classes of non-trivial, non-peripheral curves on $S$. We will define the length of  $\gamma\in \mathcal{C}(S)$ in the metric $m\in Flat(S)$ as the infimum of all lengths of representatives of $\gamma$ in $m$: $$l_m(\gamma)=\inf_{\sigma\in \gamma}length_m(\sigma).$$  All of the equivalence relations on $\mathcal{C}(S)$ can be defined as before (see Section~\ref{sec:intro}).\\

Because $S$ is incomplete, there may not be a geodesic representative in $S$ for every $\gamma \in \mathcal{C}(S)$.  On the other hand, given a sequence of closed curves $\gamma_n$ representing $\gamma$, by the Arzela-Ascoli Theorem we can extract a limiting curve $\widehat \gamma$ {\em in $\widehat{S}$}.  Unfortunately, the homotopy class of $\gamma$ can not be recovered from $\widehat{\gamma}$, but this can be remedied by working in the universal covering as we now explain.\\

Let $p \colon \widetilde{S} \to S$ denote the universal covering.  Write $\breve{S}$ to denote the metric completion of $\widetilde{S}$ (with respect to some $m \in Flat(S,q)$, for some $q$), which is a $CAT(0)$ space, and note that the action of $\pi_1(S)$ on $\widetilde{S}$ extends to an action of $\pi_1(S)$ on $\breve{S}$.  We further observe that the universal covering extends to the completion $p \colon \breve{S} \to \widehat{S}$, though this is no longer a covering map.  Using this projection, we can see that any two metrics $m,\, m'$ give rise to metric completions which are homeomorphic by a homeomorphism which is the identity on $\widetilde{S}$, and so we view $\breve{S}$ as independent of the metric $m$.\\

Given $\gamma \in \mathcal{C}(S)$, we let $h_\gamma \in \pi_1(S)$ denote an element which represents the conjugacy class determined by $\gamma$ (after arbitrarily choosing an orientation).  This element has a geodesic axis $\breve{\gamma}$.  It may enter and exit a completion point (i.e.~a point of $\breve{S} \setminus \widetilde{S}$), but when it does so, the two geodesic sub--rays emanating from that point have one well--defined angle which must be greater than or equal to $\pi$ (see Figure~\ref{fig:punctured}).  Composing with $p$ produces a geodesic $\widehat \gamma = p \circ \breve{\gamma}$ as in the previous paragraph, but we note that the axis does determine the homotopy class $\gamma$ uniquely.   See \cite{rafi} for this discussion in the case $q=2$.\\

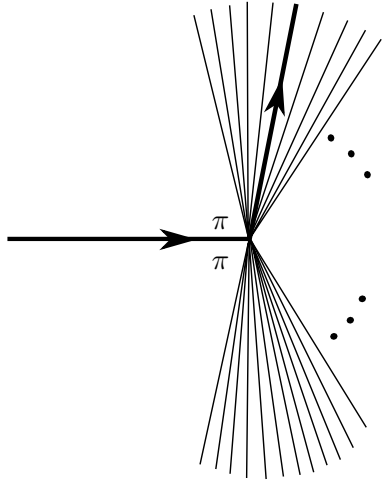
\begin{figure}[htb]
\centering
\includegraphics[height=2.5in]{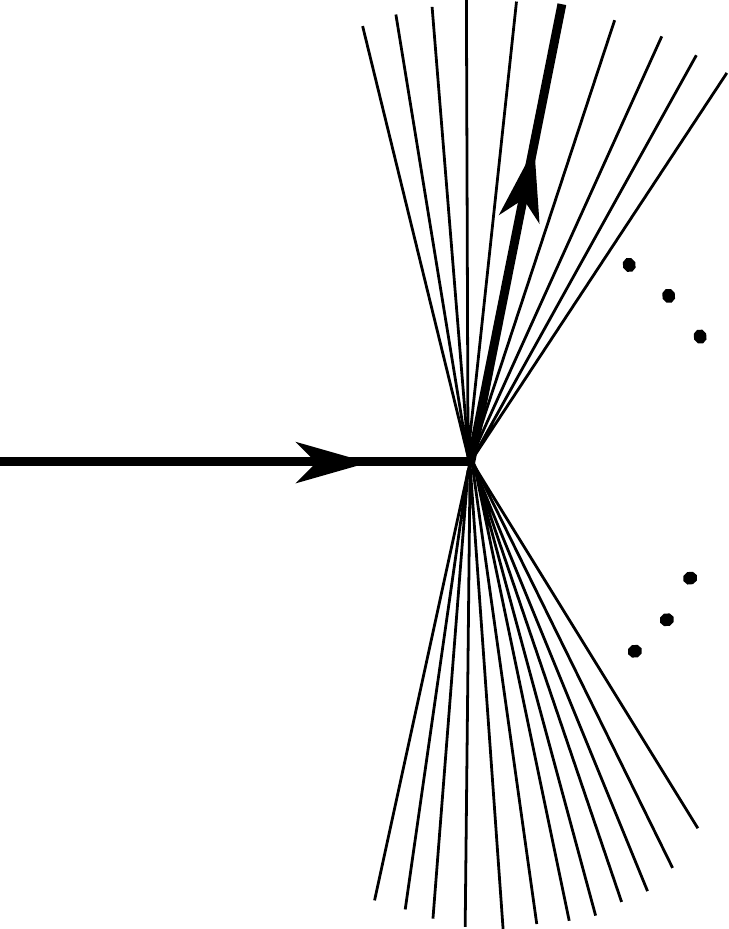}
\caption{A segment of a geodesic in the universal cover through a completion point. Every other ray emanating from the point in this figure forms another geodesic with the ray merging into that point since the angle is at least $\pi$.}
\label{fig:punctured}
\begin{picture} (-1,-12) (-1,-12)
\put(5,147){$\pi$}
\put(5,132){$\pi$}

\end{picture}
\end{figure}

It will be useful to choose representatives $\gamma_n$ of $\gamma$ constructed from $\breve{\gamma}$, which we do as follows.  Let $\{\epsilon_n\}_{n=1}^\infty$ be any sequence of positive real numbers limiting to $0$.  For each point $x \in \widehat{S} \setminus S$, consider the $\epsilon_n$--ball $B_{\epsilon_n}(x)$ about $x$ in $\widehat{S}$.  We assume that all $\epsilon_n$ are sufficiently small so that each ball $B_{\epsilon_n}(x)$ is isometric to some cone with cone angle $c(x) > 0$ as in Section~\ref{sec:Background1}, and for any two $x,x' \in \widehat{S} \setminus S$, $B_{\epsilon_n}(x) \cap B_{\epsilon_n}(x') = \emptyset$.  There is an $h_\gamma$--invariant path $\widetilde \gamma_n$ in $\widetilde{S}$ that follows $\breve{\gamma}$ in the complement of the preimage of these balls, and monotonically traverses the boundary of the preimage the balls between arcs in the complement.  See Figure~\ref{fig:puncturedgeod}.  Composing with $p$ we obtain a sequence of representatives $\gamma_n = p \circ \widetilde \gamma_n$ in $S$ of $\gamma$ with length limiting to the translation length of $h_\gamma$, which is precisely $l_m(\gamma)$.\\

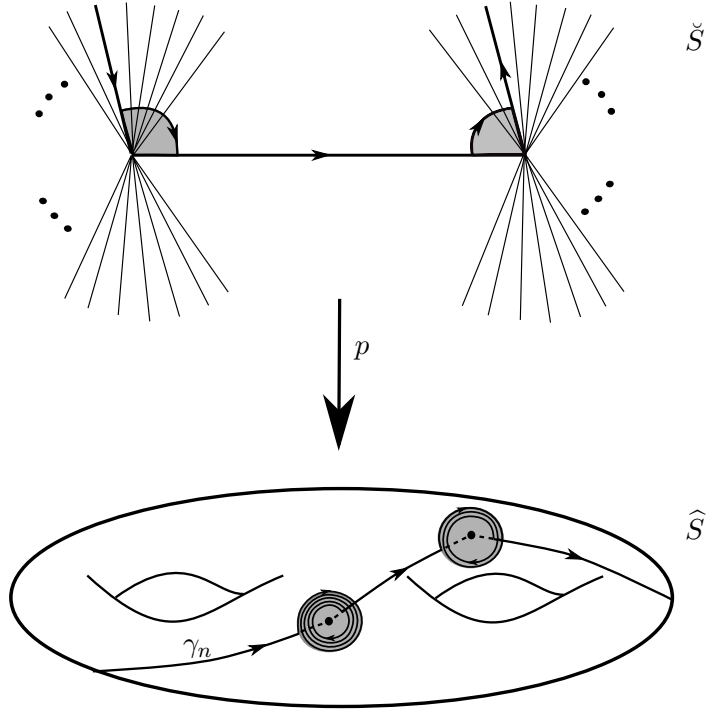
\begin{figure}[htb]
\centering
\includegraphics[height=3.7in]{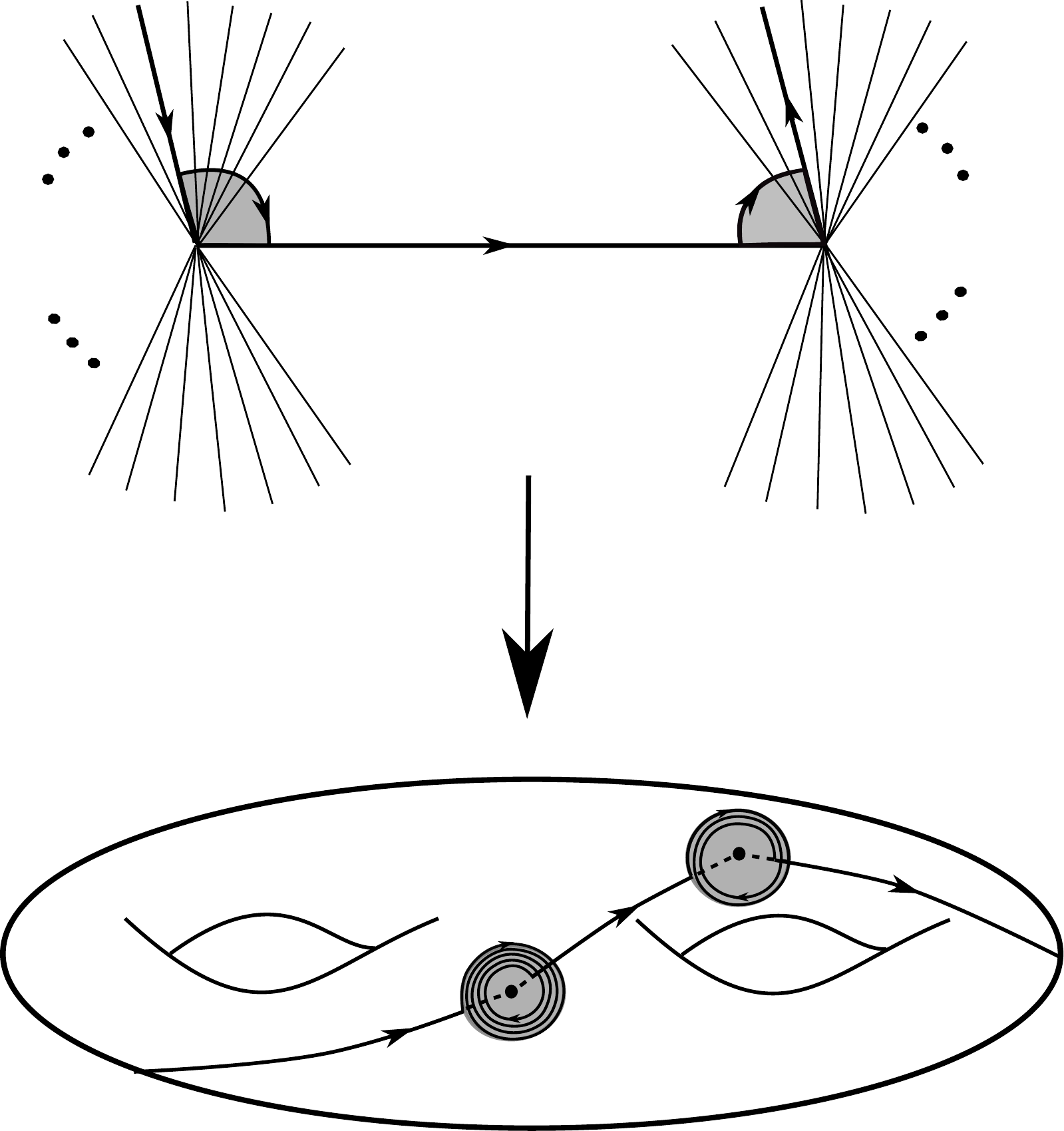}
\caption{A segment of a curve $\gamma_n$ in $S$ approximating a geodesic through cone points in $\widehat{S}$ and its preimage in $\breve{S}$. }
\label{fig:puncturedgeod}
\begin{picture}(0,0) (0,0)
\put(130,300){$\breve{S}$}
\put(130,115){$\widehat{S}$}
\put(-60,72){$\gamma_n$}
\put(5,185){$p$}
\end{picture}
\end{figure} 

We can view $\gamma_n$ as a concatenation of geodesic segments between balls $\{B_{\epsilon_n}(x)\}_{x \in \widehat{S} \setminus S}$ which makes some number of ``turns'' around the boundaries of the balls in-between segments.  Here we can formally define the {\em turn} of $\gamma_n$ as it passes around one of the balls $B_{\epsilon_n}(x)$ to be the positive real number which is the length of the path around the boundary of $B_{\epsilon_n}(x)$, divided by the length of the boundary of $B_{\epsilon_n}(x)$ (which is precisely $c(x) \cdot \epsilon_n$).  This is independent of $n$, and in this way we can think of the turn as a number associated to the geodesic representative.  This also provides a method for constructing geodesics:  a concatenation of geodesic segments between boundary of balls $\{B_{\epsilon_n}(x)\}$ followed by a monotone path around the boundary of the ball making $k$ turns, so that $k\cdot c(x)\geq \pi$, determines a geodesic by letting $n \to 0$. The point is that every time the limiting path of the lifting to $\widetilde{S}$ in $\breve{S}$ enters a completion point, it makes an angle equal to the turn times $c(x)$, which is therefore at least $\pi$. Note that for a metric in $Flat(S,q)$, every cone angle $c(x)$ is at least $\frac{2\pi}{q}$, therefore $\frac{q}{2}$ turns around $x$ with $q$ even (and $\frac{q+1}{2}$ turns when $q$ is odd), is a sufficient condition  for any curve to have a lift whose geodesic representative in $\breve{S}$ goes through a completion point in the preimage of $x$. With this understanding of geodesics, we are ready to explain the extension of the results for punctured surfaces. \\

In case $q=1$, all cone angles are greater than or equal to $2\pi$. A Euclidean cone metric with all cone angles $\geq2\pi$ on a sphere contradicts the Gauss-Bonnet formula (see Proposition~\ref{prop:Gauss}).  Therefore, there is no $Flat(S,1)$ metric on a punctured sphere. However, Theorem~\ref{thm:12} is true for punctured surfaces of genus $g\geq1$:

\begin{thm}   Let S be a surface with genus $g\geq1$ and $n>0$ punctures. For every $\gamma,\gamma'\in\mathcal{C}(S)$, $\gamma \equiv_{1} \gamma' \, \, \Leftrightarrow \, \, \gamma\equiv_{2}\gamma'$.  \end{thm}

\begin{proof}[Proof:]
The implication $\gamma \equiv_{si} \gamma' \, \, \Rightarrow \, \, \gamma\equiv_{2}\gamma'$ is also true for punctured surfaces, and follows, for example, from  Lemma~9 in \cite{clein2}. Theorem~\ref{thm:Thurston} and proof are equally valid for any punctured surface containing a nonseparating curve (so when the genus is at least $1$). To prove $\gamma \equiv_{1} \gamma' \, \, \Rightarrow \, \, \gamma\equiv_{si}\gamma'$ we follow the proof of Theorem~\ref{thm:12}. In this case we choose a homeomorphism $f_\epsilon :\widehat{S} \to X_g^\epsilon$, where $X_g^\epsilon$ is obtained by gluing the rectangle $Y_g^\epsilon$, so that the points $f(\widehat{S}\setminus S)$ belong to the image on  $X_g^\epsilon$ of the vertical $\epsilon$-length edges on $Y_g^\epsilon$. In case of a punctured torus, $Y_1^\epsilon$ is a rectangle, as $X_1^\epsilon$ is obtained from  $Y_1^\epsilon$ by gluing opposite sides together. 
\end{proof}

\vspace{0.2cm}

The equivalence relation $\equiv_{\infty}$ is trivial on punctured surfaces as well. Most of the proof of Theorem~\ref{thm:R} goes through, but we require a few modifications. Start with a complete, finite area hyperbolic metric on $S$ with $n$ cusps that correspond to $n$ punctures. Any two distinct homotopy classes of closed curves have distinct closed geodesic representatives in this metric and we denote these $\gamma$ and $\gamma'$. We deform a hyperbolic metric to a hyperbolic cone metric replacing cusps with hyperbolic cones (see \cite{judge}, for example), so that $\gamma$  and $\gamma'$ do not intersect the cone points. The rest of the proof is the same and we construct a Euclidean cone metric in which geodesic representatives of  $\gamma$ and $\gamma'$ on $\widehat{S}$ have different lengths and do not contain cone points coming from punctures. The proof of Theorem~\ref{thm:polygon} and Theorem~\ref{thm:polygon2} is the same for punctured surfaces when we regard the punctures as cone points. Thus we have:

\begin{thm} For every finite area, punctured surface $S$, the equivalence relation $\equiv_{\infty}$ on $\mathcal{C}(S)$ is trivial. \end{thm}

Regarding the equivalence relation $\equiv_q,\, q\in\mathbb{Z}_+$, a stronger statement than Theorem~\ref{thm:thmq} is true for punctured surfaces:

\begin{thm} For every finite area, punctured surface $S$ and for every $q_0\in\mathbb{Z}_+$ there are infinitely many distinct homotopy classes of curves $\gamma_i\in \mathcal{C}(S)$, $i=1,2\ldots$, such that $\gamma_i\equiv_q\gamma_j$, for all $i,j$ and for all $q\leq q_0$. \end{thm}

\begin{proof}[Proof:] Similarly to the proof of Theorem~\ref{thm:thmq} we will construct a  rank $2$ free subgroup of $\pi_1(S)$ and find infinitely many words in this free subgroups corresponding to curves on S that all have the same length for any metric $m \in Flat(S,q)$.  \\

We start by picking a reference metric $m_0 \in Flat(S,q_0)$ and an arc $\breve{\delta}$ between two distinct completion points in $\breve{S}$ which have endpoints projecting to a single point $x$ in $S$.   Let $\sigma$ be the boundary of the $\epsilon$--ball $B_{\epsilon}(x)$ for some small $\epsilon > 0$ and let $\delta$ be the restriction of $p(\breve{\delta})$ to the subpath from $\sigma$ to itself.  Let $\Gamma$ be a graph consisting of two circles joined by an edge (see Figure~\ref{fig:puncturedgraph}).   Define a map $F \colon \Gamma \to S$ by sending the middle arc to $\delta$ and the circles to loops that traverse $\sigma$ $\left[ \frac{q_0+1}{2} \right]$ times.  Here $[k]$ is the integer part of $k$.\\

Picking a basepoint on the middle edge of $\Gamma$, we let $a$ and $b$ be generators for $\pi_1(\Gamma)$ represented by loops that run from the basepoint, out to one or the other of the circles, and traverse the circle once before returning to the basepoint.  Let $w_n = a^nb$ for $n \in \mathbb{Z}$.\\

For any metric $m \in Flat(S,q)$, let $\breve{\delta}_m$ be the straightening of $\breve{\delta}$ to a $m$--geodesic between the endpoints in $\breve{S}$, and let $l_m(\delta)$ be the length of $\breve{\delta}_m$.  Note that the geodesic representative of the image of $w_n$ simply traverses the image in $\widehat{S}$ of $\breve{\delta}_m$ once in both directions.  This is because the corresponding approximates in $S$ make $\left[ \frac{q_0 + 1}{2} \right] > \left[ \frac{q + 1}{2} \right]$ turns around $x$.  Thus, for every $m \in Flat(S,q)$, $l_m(F_*(w_n)) = 2 l_m(\delta)$ for every $n \in \mathbb{Z}$. See Figure~\ref{fig:puncturedgraph} for lifts of $F_*(w_n)$ to $\breve{S}$ for 3 different $n$. \\

\begin{figure}[htb]
\centering
\includegraphics[height=3.5in]{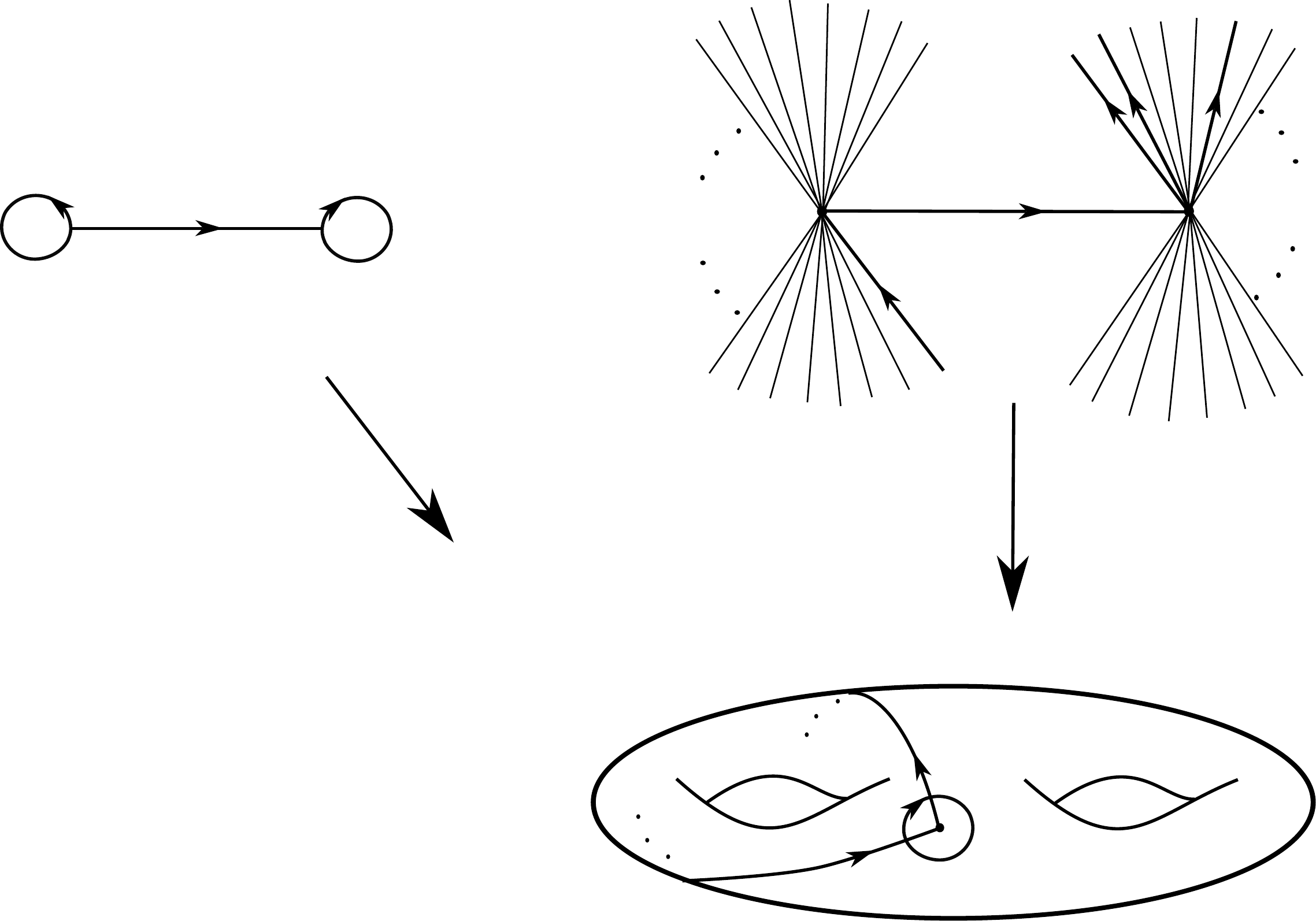}
\caption{Mapping the graph $\Gamma$ to S and some lifts of different $w_n$ to $\breve{S}$.}
\label{fig:puncturedgraph}
\begin{picture}(0,0) (0,0)
\put(-175,205){$a$}
\put(-83,205){$b$}
\put(97,238){$\breve{\delta}$}
\put(77,57){$x$}
\put(-130,245){$\Gamma$}
\put(180,273){$\breve{S}$}
\put(180,47){$\widehat{S}$}
\put(-88,154){$F$}
\put(103,150){$p$}
\end{picture}
\end{figure}
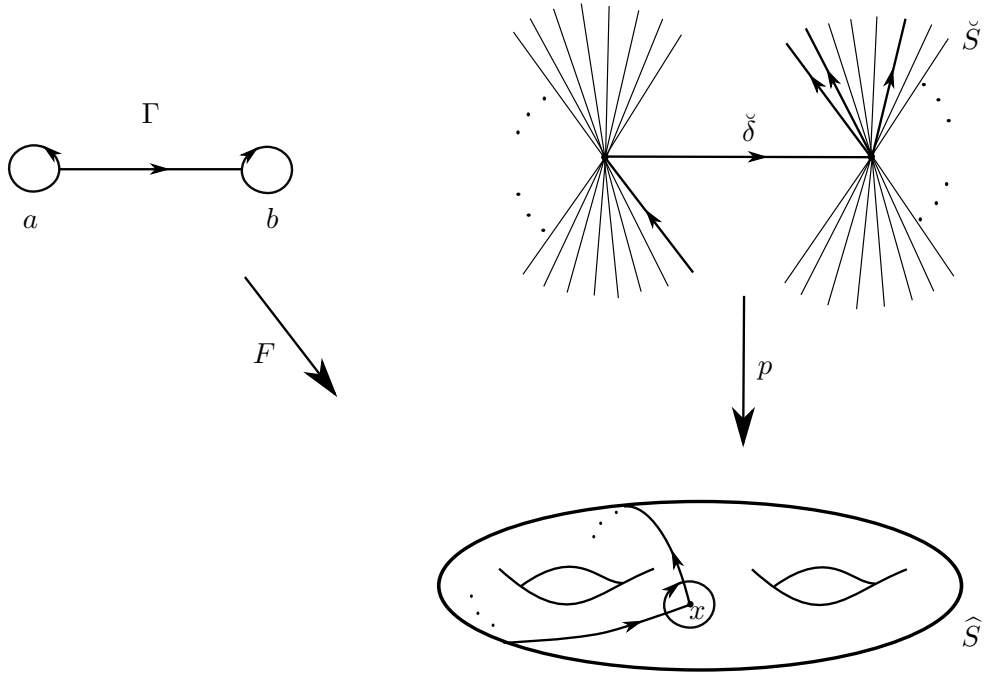

The only thing left to prove now is that there is an infinite subsequence of $w_{n_k}$ of $w_n$, such that each two $F_*(w_{n_k})$ represent different homotopy classes of curves in $\mathcal{C}(S)$. Let $m'$ be an arbitrary hyperbolic metric on $S$. The hyperbolic metric is complete, and thus  every closed curve has a geodesic representative. As $n \to \infty$ the lengths $l_{m'}(F_*(w_{n_k}))$ tend to infinity since they wrap more and more around the cusps. Therefore we can find a subsequence of geodesics whose $m'$-lengths are all different and thus cannot be homotopic. This completes our proof. \end{proof}

\bibliographystyle{amsplain}
\bibliography{paper}

\providecommand{\bysame}{\leavevmode\hbox to3em{\hrulefill}\thinspace}
\providecommand{\MR}{\relax\ifhmode\unskip\space\fi MR }
\providecommand{\MRhref}[2]{%
  \href{http://www.ams.org/mathscinet-getitem?mr=#1}{#2}
}
\providecommand{\href}[2]{#2}
\begin{thebibliography}{10}

\bibitem{abraham}
R.~Abraham, \emph{Bumpy metrics}, Global {A}nalysis ({P}roc. {S}ympos. {P}ure
  {M}ath., {V}ol. {XIV}, {B}erkeley, {C}alif., 1968), Amer. Math. Soc.,
  Providence, R.I., 1970, pp.~1--3. \MR{0271994 (42 \#6875)}

\bibitem{anderson}
James~W. Anderson, \emph{Variations on a theme of {H}orowitz}, Kleinian groups
  and hyperbolic 3-manifolds ({W}arwick, 2001), London Math. Soc. Lecture Note
  Ser., vol. 299, Cambridge Univ. Press, Cambridge, 2003, pp.~307--341.
  \MR{2044556 (2005h:57002)}

\bibitem{anosov}
D.~V. Anosov, \emph{Generic properties of closed geodesics}, Izv. Akad. Nauk
  SSSR Ser. Mat. \textbf{46} (1982), no.~4, 675--709, 896. \MR{670163
  (84b:58029)}

\bibitem{bonahon}
Francis Bonahon, \emph{The geometry of {T}eichm\"uller space via geodesic
  currents}, Invent. Math. \textbf{92} (1988), no.~1, 139--162. \MR{931208
  (90a:32025)}

\bibitem{bridson}
Martin~R. Bridson and Andr{\'e} Haefliger, \emph{Metric spaces of non-positive
  curvature}, Grundlehren der Mathematischen Wissenschaften [Fundamental
  Principles of Mathematical Sciences], vol. 319, Springer-Verlag, Berlin,
  1999. \MR{1744486 (2000k:53038)}

\bibitem{clein2}
Moon Duchin, Christopher~J. Leininger, and Kasra Rafi, \emph{Length spectra and
  degeneration of flat metrics}, Invent. Math. \textbf{182} (2010), no.~2,
  231--277. \MR{2729268 (2011m:57022)}

\bibitem{farkas}
H.~M. Farkas and I.~Kra, \emph{Riemann surfaces}, second ed., Graduate Texts in
  Mathematics, vol.~71, Springer-Verlag, New York, 1992. \MR{1139765
  (93a:30047)}

\bibitem{FLP}
A.~Fathi, F.~Laudenbach, and V.~Po\'enaru, \emph{Travaux de {T}hurston sur les
  surfaces}, Soci\'et\'e Math\'ematique de France, Paris, 1991, S\'eminaire
  Orsay, Reprint of {\it Travaux de Thurston sur les surfaces}, Soc.\ Math.\
  France, Paris, 1979 Ast\'erisque No. 66-67 (1991). \MR{MR1134426 (92g:57001)}

\bibitem{hatcher}
Allen Hatcher, \emph{Algebraic topology}, Cambridge University Press,
  Cambridge, 2002. \MR{1867354 (2002k:55001)}

\bibitem{hopf}
Eberhard Hopf, \emph{Statistik der geod\"atischen {L}inien in
  {M}annigfaltigkeiten negativer {K}r\"ummung}, Ber. Verh. S\"achs. Akad. Wiss.
  Leipzig \textbf{91} (1939), 261--304. \MR{0001464 (1,243a)}

\bibitem{horowitz}
Robert~D. Horowitz, \emph{Characters of free groups represented in the
  two-dimensional special linear group}, Comm. Pure Appl. Math. \textbf{25}
  (1972), 635--649. \MR{0314993 (47 \#3542)}

\bibitem{judge}
Christopher~M. Judge, \emph{Conformally converting cusps to cones}, Conform.
  Geom. Dyn. \textbf{2} (1998), 107--113 (electronic). \MR{1657563 (99m:53026)}

\bibitem{kapovich}
Ilya Kapovich, Gilbert Levitt, Paul Schupp, and Vladimir Shpilrain,
  \emph{Translation equivalence in free groups}, Trans. Amer. Math. Soc.
  \textbf{359} (2007), no.~4, 1527--1546 (electronic). \MR{2272138
  (2008d:20045)}

\bibitem{clein1}
Christopher~J. Leininger, \emph{Equivalent curves in surfaces}, Geom. Dedicata
  \textbf{102} (2003), 151--177. \MR{2026843 (2004j:57022)}

\bibitem{masters}
Joseph~D. Masters, \emph{Length multiplicities of hyperbolic 3-manifolds},
  Israel J. Math. \textbf{119} (2000), 9--28. \MR{1802647 (2001k:57019)}

\bibitem{masur}
Howard Masur and John Smillie, \emph{Hausdorff dimension of sets of nonergodic
  measured foliations}, Ann. of Math. (2) \textbf{134} (1991), no.~3, 455--543.
  \MR{1135877 (92j:58081)}

\bibitem{minsky}
Yair~N. Minsky, \emph{Harmonic maps, length, and energy in {T}eichm\"uller
  space}, J. Differential Geom. \textbf{35} (1992), no.~1, 151--217.
  \MR{1152229 (93e:58041)}

\bibitem{neumann}
W.~D. Neumann, \emph{Notes on geometry and 3-manifolds}, Low dimensional
  topology ({E}ger, 1996/{B}udapest, 1998), Bolyai Soc. Math. Stud., vol.~8,
  J\'anos Bolyai Math. Soc., Budapest, 1999, With appendices by Paul Norbury,
  pp.~191--267. \MR{1747270 (2001a:57032)}

\bibitem{penner}
R.~C. Penner and J.~L. Harer, \emph{Combinatorics of train tracks}, Annals of
  Mathematics Studies, vol. 125, Princeton University Press, Princeton, NJ,
  1992. \MR{1144770 (94b:57018)}

\bibitem{rafi}
Kasra Rafi, \emph{A characterization of short curves of a {T}eichm\"uller
  geodesic}, Geom. Topol. \textbf{9} (2005), 179--202. \MR{2115672
  (2005i:30072)}

\bibitem{randol}
Burton Randol, \emph{The length spectrum of a {R}iemann surface is always of
  unbounded multiplicity}, Proc. Amer. Math. Soc. \textbf{78} (1980), no.~3,
  455--456. \MR{553396 (80k:58100)}

\bibitem{strebel}
Kurt Strebel, \emph{Quadratic differentials}, Ergebnisse der Mathematik und
  ihrer Grenzgebiete (3) [Results in Mathematics and Related Areas (3)],
  vol.~5, Springer-Verlag, Berlin, 1984. \MR{743423 (86a:30072)}

\end{thebibliography}

\end{document}